\documentclass[a4paper,10pt]{amsart}
 \usepackage{amsfonts,mathrsfs}
 \usepackage{amsthm}
 \usepackage{amssymb}
 \usepackage{enumerate}
 \usepackage{enumitem}
 \usepackage{tikz-cd}
 \usepackage{ulem}
 \usepackage{stmaryrd}	 
 \usepackage{subcaption}
 \usepackage{booktabs}
 \usepackage{hyperref}
 \usepackage{cleveref}
 
\hypersetup{colorlinks=true, citecolor=black, linkcolor=black,
    filecolor=black,urlcolor=blue,} 
 
\theoremstyle{definition}
\newtheorem{definition}{Definition}[section]
\newtheorem{example}[definition]{Example}
\newtheorem{remark}[definition]{Remark}

\newtheorem{notation}[definition]{Notation}
\theoremstyle{plain}
\newtheorem{proposition}[definition]{Proposition}
\newtheorem{theorem}[definition]{Theorem}
\newtheorem*{thm*}{Theorem}
\newtheorem{lemma}[definition]{Lemma}

\newtheorem{corollary}[definition]{Corollary}
\newtheorem*{cor*}{Corollary}

\newtheorem*{con*}{Conjecture}

\newtheorem*{verm*}{Vermutung}

\usepackage{mathrsfs}

\newcommand{\Int}{\operatorname{int}}

\newcommand{\rank}{\operatorname{rank}} 

\newcommand{\T}{\operatorname{T}}

\newcommand{\C}{{\mathbb C}}
\newcommand{\R}{{\mathbb R}}

\newcommand{\Z}{{\mathbb Z}}

\newcommand{\K}{{\mathbb K}}

\renewcommand{\T}{{\mathbb T}}

\title[Hyperbolic plane curves near the non-singular tropical limit]{Hyperbolic plane curves near the non-singular tropical limit}

\DeclareMathOperator{\Conv}{Conv}

\DeclareMathOperator{\Area}{Area}

\newcommand{\pr}{\mathbb{P}}

\newcommand{\E}{\mathcal{E}}

\renewcommand{\c}{\mathbf{c}}

\renewcommand{\qedsymbol}{\(\blacksquare\)}

\DeclareMathOperator{\Edge}{Edge}

\DeclareMathOperator{\Adm}{Adm}
\DeclareMathOperator{\Div}{Div}
\DeclareMathOperator{\Hyp}{Hyp}

\DeclareMathOperator{\Trop}{Trop}

\DeclareMathOperator{\val}{val}

\author{C\'edric Le Texier}
\address{Universitetet i Oslo, Norway}
\email{cedricle@math.uio.no}

\begin{document}
	
 \subjclass[2010]{Primary: 14T05, 14P25}

\begin{abstract}
We determine necessary and sufficient conditions for real algebraic curves near the non-singular tropical limit to be hyperbolic with respect to a point, thus generalising Speyer's classification of stable curves near the tropical limit.
In order to obtain the conditions, we develop tools of real tropical intersection theory.
We introduce the tropical hyperbolicity locus and the signed tropical hyperbolicity locus of a real algebraic curve near the non-singular tropical limit. 
In the case of honeycombs, we characterise the tropical hyperbolicity locus in terms of the set of twisted edges on the tropical limit.
\end{abstract}

\maketitle

\tableofcontents

\section{Introduction}

The aim of this article is to study hyperbolic plane curves near the non-singular tropical limit by the means of real intersection theory of tropical curves.

\begin{definition}
Let $\mathcal{C} \subset \pr^2$ be a real algebraic curve of degree $d$.
Denote by $\mathcal{C} (\R )$ the set of real points of the curve $\mathcal{C}$, and by $\pr^2 (\R)$ the set of real points of the projective plane $\pr^2$.
We say that $\mathcal{C}$ is \emph{hyperbolic with respect to} a real point $p \in \pr^2 (\R) \backslash \mathcal{C} (\R)$ if for every real line $\mathcal{L} \subset \pr^2$ containing $p$, the intersection $\mathcal{L} \cap \mathcal{C}$ consists of $d$ real points, counted with multiplicity.
We say that $\mathcal{C}$ is \emph{hyperbolic} if there exists a real point $p\in \pr^2 (\R) \backslash \mathcal{C} (\R )$ such that $\mathcal{C}$ is hyperbolic with respect to $p$.
The \emph{hyperbolicity locus} of $\mathcal{C}$ is the subset of $\pr^2 (\R) \backslash \mathcal{C} (\R)$ of real points $p$ such that $\mathcal{C}$ is hyperbolic with respect to $p$.
We say that $\mathcal{C}$ is \emph{stable} if the hyperbolicity locus of $\mathcal{C}$ contains the whole positive orthant $(\R_{>0})^2$ of the algebraic torus $(\R^\times )^2 \subset \pr^2 (\R)$.
\end{definition}

Hyperbolic varieties are studied in several domains such as convex optimization (\cite{guler1997hyperbolic}, \cite{renegar2004hyperbolic}),
topology of real algebraic varieties (\cite{rokhlin1978complex}, \cite{helton2007linear}) and determinantal representations (\cite{lewis2005lax}, \cite{helton2007linear}).
Stable polynomials, which gave the name to the stable hypersurfaces they define, have a particularly interesting connection to matroid theory (\cite{choe2004homogeneous}, \cite{wagner2009criterion},\cite{branden2011obstructions}).
Speyer \cite{speyer2005horn} and Br{\"a}nd{\'e}n \cite{branden2010discrete} studied the combinatorial properties of the tropicalisation of stable polynomials, and Rinc{\'on}, Vinzant, Yu \cite{rincon2021positively} studied the tropicalisation of positively hyperbolic varieties, an analogue of stable varieties in higher codimension, and their relation with positroids.

The relative topology of non-singular hyperbolic curves in the projective plane is well-known.
Helton and Vinnikov \cite[Theorem 5.2]{helton2007linear} showed that the real part of a non-singular hyperbolic curve $\mathcal{C}$ of degree $d$ in $\pr^2$ consists of $\left\lfloor \frac{d}{2} \right\rfloor$ nested \emph{ovals}, which are connected components of $\mathcal{C} (\R )$ homeomorphic to $S^1$ and disconnecting $\pr^2 (\R )$, plus a \emph{pseudo-line} (which are connected components of $\mathcal{C} (\R )$ homeomorphic to $S^1$ and not disconnecting $\pr^2 (\R )$) if $d$ is odd.
Rokhlin \cite{rokhlin1978complex} proved that a non-singular real algebraic curve $\mathcal{C}$ of degree $d$ in $\pr^2$ with $\left\lfloor \frac{d}{2} \right\rfloor$ nested ovals is \emph{dividing}, meaning that the real part of $\mathcal{C}$ disconnects the complex part of $\mathcal{C}$.
Moreover, Orevkov \cite[Proposition 1.1]{orevkov2007arrangements} showed that a non-singular real curve $\mathcal{C}$ of degree $d$ is hyperbolic if $\mathcal{C}$ is dividing and $\mathcal{C}(\R)$ has exactly $\left\lfloor \frac{d}{2} \right\rfloor$ ovals (see \Cref{PropHypRok}).
We obtain then an equivalence between hyperbolicity and being dividing with $\left\lceil \frac{d}{2} \right\rceil$ connected components in the real part (see \Cref{CorHypRok})  

This last fact can be translated into a combinatorial condition for hyperbolic curves in $\pr^2$ \emph{near the non-singular tropical limit} (see \Cref{DefTropLimit} and \Cref{NotationRealPhase}).
%
Haas \cite[Section 5.4]{haas1997real} gave a characterisation of dividing curves obtained via combinatorial patchworking (see \Cref{ThDividing}), which can be formulated in terms of \emph{twisted edges} on a non-singular tropical curve $C$ as follows (see \Cref{DefTwist} and \Cref{amoebaedge2}).
A curve $\mathcal{C}$ near the non-singular tropical limit is dividing if and only if every cycle of $C$ has an even number of twisted edges, for $C$ the non-singular tropical curve given as support of the tropical limit.
Thus we call a set of twisted edges $T$ on a non-singular tropical curve $C$ \emph{dividing} if every cycle of $C$ contains an even number of element of $T$.
Renaudineau and Shaw \cite[Theorem 7.2]{renaudineau2018bounding} introduced a way to compute the number of real components of a real algebraic curve $\mathcal{C}$ near the non-singular tropical limit with support a non-singular tropical curve $C$ with set of twisted edges $T$, using a homomorphism $\partial_T$ of tropical homology groups determined by $T$ (see \Cref{ThConnectHom}).
Thanks to those results, we can first characterise hyperbolic plane curves near the non-singular tropical limit in terms of twisted edges as follows.

\begin{proposition}[\Cref{PropHypTwist}]
\label{IntroPropHypTwist}
Let $\mathcal{C}$ be a real algebraic curve of degree $d$ in $\pr^2$ near the non-singular tropical limit with support a non-singular tropical curve $C$ of degree $d$ in the tropical projective plane $\T \pr^2$.
Let $T$ be the set of twisted edges on $C$ induced by $\mathcal{C}$.
The curve $\mathcal{C}$ is hyperbolic if and only if $T$ is dividing and the kernel of the homomorphism $\partial_T$ determined by $T$ has dimension $\left\lceil \frac{d}{2} \right\rceil - 1$.
\end{proposition}

For $\mathcal{C}$ a hyperbolic curve in $\pr^2$ near the non-singular tropical limit with support a tropical curve $C$, we define its \emph{signed tropical hyperbolicity locus} $\R H$ as the interior of the innermost connected component of the real part of $C$ in the real projective plane, and its \emph{tropical hyperbolicity locus} $H$ as a subset of the complement $\T \pr^2 \backslash C$ so that each connected component of $H$ has a symmetric copy in $\R H$ (\Cref{DefTropHypLocus}).
The closure of a tropical hyperbolicity locus is a tropical spectrahedron in the sense of \cite{allamigeon2020tropical}, ie.~ the tropicalisation of a spectrahedron, and the signed tropical hyperbolicity locus $\R H$, as an open set, satisfies the conditions to be signed tropically convex \cite{loho2019signed}.

The description of hyperbolicity in \Cref{IntroPropHypTwist} is not geometric.
In fact, it involves the computation of the kernel of a symmetric matrix modulo 2 (see \Cref{ThConnectHom}, which is a reformulation of \cite[Theorem 7.2]{renaudineau2018bounding}), hence it does not give any information on the tropical hyperbolicity locus. 
This motivates our study of the intersection theory of real algebraic curves near the non-singular tropical limit. 

Brugallé and L{\'o}pez de Medrano \cite[Proposition 3.11]{brugalle2012inflection} gave an upper bound on the number of intersection points of two algebraic curves tropicalising to two distinct tropical curves, based on the intersection locus of these tropical curves.
This bound is sharp whenever the intersection locus is compact.
Starting from this result (recalled in \Cref{PropCpxMult}), we give an analogous statement in the real case (\Cref{PropRealMult}).
Next in \Cref{CorTransMult2}, we compute the number of real and complex conjugated intersection points when the real algebraic curves tropicalise to non-singular tropical curves intersecting in a \emph{transverse} point (\Cref{DefTransverse}),
in terms of \emph{real phase structure} on the edges of the                                                                                                                                                                                                                                                                                                                                                                                                                                                                                                                                                                                                                                                                                                                                                                                                                                non-singular tropical curves (see \Cref{DefRealPhase}, \Cref{RealPhase} and \Cref{RkRealPhase}).

A real phase structure $\E$ on a non-singular tropical curve curve $C$ is a collection of $\Z_2$-affine spaces $(\E_e)_e$, for $\Z_2 := (\Z / 2\Z )$ and $e$ running through the edges of $C$, such that $\E_e$ is a 1-dimensional affine subspace of $\Z_2^2$ parallel to the $\Z_2$-vector space spanned by the direction modulo 2 of the edge $e$, and satisfying an additional condition around each 3-valent vertex of $C$.
A real phase structure $\varepsilon$ on a tropical point $v$ is an element of $\Z^2_2$.
Thus in the following, we will specify for real algebraic curves near the non-singular tropical limit the pair $(C,\E)$ given as tropical limit (see \Cref{NotationRealPhase}), and we will specify for real points near the tropical limit the pair $(v,\varepsilon)$ given as tropical limit. 
Moreover, the support of a tropical limit is the first element of the pair, ie.~ either a non-singular tropical curve or a tropical point depending on context. 

In order to study the hyperbolicity locus, we also need to compute the number of real intersection points of two tropical curves with real phase structures in two non-transverse cases.
In the following, we consider that any edge of a tropical curve contains its adjacent vertices.
Said otherwise, the edges are considered to be closed in the Euclidean topology on $\R^2$.
In the case where the connected component $E$ of the intersection is an edge of a non-singular tropical curve $C$, strictly contained in an edge of a non-singular tropical curve $C'$, we obtain the following result.

\begin{theorem}[\Cref{Th4.3.9}]
\label{IntroTh4.3.9}
Let $C , C'$ be two non-singular tropical curves in $\R^2$ such that there exists a closed bounded edge $e$ of $C$ contained in the interior of an edge $e'$ of $C'$ (see \Cref{nontransverse2}).
Let $\E , \E'$ be real phase structures on $C , C'$, inducing sets of twisted edges on $C$ and $C'$. 
\begin{enumerate}
\item \label{IntroItem4.3.9.1} If $\E_e \neq \E_{e '} '$, or if $\E_e = \E_{e '} '$ and $e$ is twisted, then for any two real algebraic curves $\mathcal{C} , \mathcal{C}' \subset (\C^\times )^2$ near the non-singular tropical limits $(C,\E )$ and $(C' , \E ')$, the intersection $\mathcal{C}\cap \mathcal{C}'$ has exactly two distinct real points near the tropical limit with support in $e$.
\item \label{IntroItem4.3.9.2} If $\E_e = \E_{e '} '$ and $e$ is non-twisted, then for $\mathcal{C}, \mathcal{C'} \subset (\C^\times )^2$ two non-singular real algebraic curves near the tropical limits $(C,\E )$ and $(C' , \E ')$, the intersection $\mathcal{C}\cap \mathcal{C}'$ has exactly either two distinct real points, two distinct complex conjugated points, or a multiplicity 2 real point, near the tropical limit with support in $e$.
Moreover, there exist infinitely many curves satisfying the two first intersection types, and there exist exactly two pairs of curves satisfying the third intersection type.
\end{enumerate}
\end{theorem}

Furthermore, in the case where the real phase structures are distinct on the intersection, we obtain in \Cref{CorDistinctPhaseTwist} that the two distinct real intersection points are near the tropical limit with support the two vertices of the edge $e$.
We deduce from \Cref{IntroTh4.3.9} a necessary condition for the intersection edge $e$ to lift to a multiplicity 2 real point, see \Cref{CorNonTwistTang}.

Cueto, Len and Markwig have extensively studied the lifts of tropical bitangents, notably in the case of tropical quartic curves \cite{len2020lifting},\cite{cueto2020combinatorics}.
Using \Cref{CorNonTwistTang}, we study in \Cref{ExBitangent} the existence of lifts to totally real bitangents of the tropical quartic curve analysed in \cite[Example 1.3]{cueto2020combinatorics} in terms of real phase structure and twisted edges.

In a similar way as in \Cref{IntroTh4.3.9}, we treat the case of two non-singular tropical curves $C,C'$ intersecting in a segment $E$ strictly contained in an edge $e$ of $C$ and strictly contained in an edge $e'$ of $C'$ (see the underlying local model in \Cref{FigDefRel}).
We introduce for this case the notion of \emph{relatively twisted} connected component of the intersection (see \Cref{DefRelTwist} and \Cref{FigRelTwist}).
Notice that the roles of (relatively) twisted and non-twisted are switched compared to \Cref{IntroTh4.3.9}.

\begin{theorem}[\Cref{Th4.3.10}]
\label{IntroTh4.3.10}
Let $C, C'$ be two non-singular tropical curves in $\R^2$ such that there exists a segment $E \subset C\cap C'$, non-reduced to a point, strictly contained in both a closed edge $e$ of $C$ and a closed edge $e'$ of $C'$.
Let $\E , \E'$ be real phase structures on $C , C'$, inducing sets of twisted edges on $C$ and $C'$. 
\begin{enumerate}
\item \label{IntroItem4.3.10.1} If $\E_e \neq \E_{e '} '$, or if $\E_e = \E_{e '} '$ and $E$ is relatively non-twisted, then for any two real algebraic curves $\mathcal{C} , \mathcal{C}' \subset (\C^\times )^2$ near the non-singular tropical limits $(C,\E )$ and $(C' , \E ')$, the intersection $\mathcal{C}\cap \mathcal{C}'$ has exactly two distinct real points near the tropical limit with support in $E$.
\item \label{IntroItem4.3.10.2} If $\E_e = \E_{e '} '$ and $E$ is relatively twisted, then for $\mathcal{C}, \mathcal{C'} \subset (\C^\times )^2$ two real algebraic curves near the non-singular tropical limits $(C,\E )$ and $(C' , \E ')$, the intersection $\mathcal{C}\cap \mathcal{C}'$ has exactly either two distinct real points, two distinct complex conjugated points, or a multiplicity 2 real point, near the tropical limit with support in $E$.
Moreover, there exist infinitely many curves satisfying the two first intersection types, and there exist exactly two pairs of curves satisfying the third intersection type.
\end{enumerate}
\end{theorem}

Again, in the case where the real phase structures are distinct on the edges containing the intersection, we obtain in \Cref{CorRelDistRealPhase} that the two distinct real intersection points are near the tropical limit with support the two vertices of the segment $E$.
We also deduce from \Cref{IntroTh4.3.10} a necessary condition for the intersection edge $e$ to lift to a multiplicity 2 real point, see \Cref{CorNonRelTwistTang}.

Thanks to these intersection results, we compute the necessary and sufficient conditions for every real algebraic curve near a common non-singular tropical limit  to be hyperbolic with respect to real points near a common tropical limit.
In the following, we use a decomposition $\Sigma_v$ of the tropical projective plane $\T \pr^2$ with respect to a point $v \in \T \pr^2$, see \Cref{DefFanPt} and \Cref{SubdivisionPoint}.
The idea is to subdivide $\T \pr^2$ by three half-lines $\tau_\eta$ with origin $v$ and outward primitive integer direction $\eta\in \{ (1,0) , (0,1), (-1,-1) \} \subset \Z^2$, such that $\T \pr^2 \backslash \bigcup \tau_\eta$ has three connected components $\sigma_\eta$ with $\tau_\eta$ not contained in the boundary of $\sigma_\eta$.
We will say that a point $v$ in $\T \pr^2 \backslash C$ is \emph{generic with respect to} $C$ if the rays of $\Sigma_v$ intersect $C$ transversely in the interior of its edges.

\begin{theorem}[\Cref{Th4.4.4}]
\label{IntroTh4.4.4}
Let $C$ be a non-singular tropical curve of degree $d$ in the tropical projective plane $\T \pr^2$, and let $\E$ be a real phase structure on $C$. 
Let $v'$ be a tropical point in $\T \pr^2 \backslash C$, and let $\varepsilon \in \Z_2^2$.
Every real algebraic curve $\mathcal{C}$ of degree $d$ in $\pr^2$ and near the non-singular tropical limit $(C,\E )$ is hyperbolic with respect to a point $p$ near the tropical limit $(v',\varepsilon)$, if and only if for $\Sigma_v$ the polyhedral subdivision of $\T \pr^2$ with respect to a point $v$ generic with respect to $C$ in the same component of $\T \pr^2 \backslash C$ as $v'$, the following conditions are satisfied.
\begin{enumerate}
\item \label{IntroItemHyp1} Every vertex of $C$ lying in the interior of a face $\sigma_\eta$ of $\Sigma_v$ is incident to an edge of primitive integer direction $\eta$.

\item \label{IntroItemHyp2} For every edge $e$ of $C$ intersecting a face $\tau_\zeta$ of $\Sigma_v$, such that the primitive integer direction $\overrightarrow{e}$ of $e$ satisfies $|\det (\overrightarrow{e} | \zeta )| = 2$, the real phase structure on $e$ satisfies $\varepsilon \in \E_e$.

\item \label{IntroItemHyp3} For every bounded edge $e$ of $C$ of primitive integer direction $\eta$ intersecting the face $\sigma_\eta$ of $\Sigma_v$, the edge $e$ is twisted if $e \subset \sigma_\eta$, and otherwise the component $e\cap \sigma_\eta$ is relatively non-twisted with respect to the unique real tropical line $(L,\E ')$ through $(v,\varepsilon)$ with $(e\cap \sigma_\eta ) \subset e' \subset L$ and $\E_e = \E_{e'}'$. 
\end{enumerate}
\end{theorem}

A \emph{honeycomb} is a tropical curve with all edges of primitive integer direction in $\{ (1,0) , (0,1) , (1,1) \}$.
Speyer \cite[Section 5.3]{speyer2005horn} proved that a non-singular stable curve in $\pr^2$ near the tropical limit induces the data of a honeycomb (possibly with some contracted edges) in the tropical projective plane $\T \pr^2$, with constant distribution of signs on the dual subdivision.
Conversely, he proved that a non-singular honeycomb in the tropical projective plane $\T \pr^2$ with constant distribution of signs on the dual subdivision (hence  with every bounded edge twisted) is the non-singular tropical limit of a stable curve in $\pr^2$. 
We give in \Cref{ThStable} an alternative approach to Speyer's result in the case of non-singular tropical limit based on \Cref{IntroTh4.4.4}.

We can treat the case of honeycombs further, as these tropical curves automatically satisfy \Cref{IntroItemHyp1} and \Cref{IntroItemHyp2} of \Cref{IntroTh4.4.4} for any point $v'$ in $\T \pr^2$.
A multi-bridge on a honeycomb $C$ is a set $B$ of parallel bounded edges of $C$, dual to edges lying on a common line in the dual subdivision, and such that $C\backslash B$ has two connected components (see \Cref{DefMultiBridge}).
The multi-bridges on a honeycomb form a basis for dividing patchworking (\Cref{PropHoneyBridge}). 
We show that if a real algebraic curve near the non-singular tropical limit converges to a non-singular honeycomb, checking if a connected component belongs to the tropical hyperbolicity locus reduces to the following criterion on twisted edges. 
Let $\alpha = (\alpha_1 , \alpha_2)$ be an integer point in the dual subdivision of a non-singular honeycomb $C$ of degree $d$ in $\T \pr^2$, and let $\alpha^\vee$ be the connected component of $\T \pr^2 \backslash C$ dual to $\alpha$.
In the following result, we say that a multi-bridge $B$ is on the \emph{left} of $\alpha^\vee \subset \T \pr^2 \backslash C$ if the dual set of edges $B^\vee \subset \Delta_C$ lies in a vertical line through an integer point $(\beta_1 , \alpha_2) \in \Delta_C \cap \Z^2$ with $\beta_1 < \alpha_1$.
Similarly, we say that $B$ is \emph{below} $\alpha^\vee$ if $B^\vee$ lies in a horizontal line through a point $(\alpha_1 , \beta_2) \in \Delta_C \cap \Z^2$ with $\beta_2 < \alpha_2$, and we say that $B$ is \emph{diagonally above} $\alpha^\vee$ if $B^\vee$ lies in a line of direction $(1,-1)$ through a point $(\beta_1 , \beta_2) \in \Delta_C \cap \Z^2$ with $\alpha_1 + \alpha_2 < \beta_1 + \beta_2$.
\begin{corollary}[\Cref{CorHypHoneyBridge}]
\label{IntroCorHypHoneyBridge}
Let $C$ be a non-singular honeycomb of degree $d$ in the tropical projective plane $\T \pr^2$, and let $\E$ be a real phase structure on $C$ inducing the dividing set of twisted edges $T$ on $C$.
Let $\alpha = (\alpha_1 , \alpha_2 ) \in \Delta_C \cap \Z^2$ be an integer point in the dual subdivision of $C$.
For every real algebraic curve $\mathcal{C}$ of degree $d$ in $\pr^2$ near the non-singular tropical limit $(C,\E )$, the tropical hyperbolicity locus of $\mathcal{C}$ contains the connected component $\alpha^\vee \subset (\T\pr^2 \backslash C)$ dual to $\alpha$, if and only if for every multi-bridge $B$ on the left, below and diagonally above $\alpha^\vee$ on $C$, the edges of $B$ are twisted.
\end{corollary}


The article will be organised as follows.
We begin by recalling standard notions about tropical curves in \Cref{Sec2}.
In \Cref{Sec3}, we recall how to use combinatorial patchworking in order to construct the real part of a tropical curve, as well as the notion of twisted edges a tropical curve, and finally how to go back and forth between the various descriptions.
We study in \Cref{SectionIntersection} the intersection locus of real algebraic curves near the non-singular tropical limit and prove \Cref{IntroTh4.3.9} and \Cref{IntroTh4.3.10}. 
We consider in \Cref{SecHyperbolic} hyperbolic plane curves near the non-singular tropical limit, starting by stating the hyperbolicity criterion in terms of twisted edges (\Cref{IntroPropHypTwist}).
We then define the tropical hyperbolicity locus and the signed tropical hyperbolicity locus, and relate this to the notions of tropical spectrahedra and signed tropical convexity (\Cref{RkTropSpec} and \Cref{RkTropConv}).
We continue by proving the intersection-theoretic characterisation of the hyperbolicity locus (\Cref{IntroTh4.4.4}) and deduce from it Speyer's criterion for stable curves, in the case when the tropical limit is non-singular (\Cref{ThStable}).
Then, we introduce a $\Z_2$-vector space structure on the set of admissible twists on a tropical curve, in order to study the honeycomb case and prove  \Cref{IntroCorHypHoneyBridge}.

\bigskip

\noindent \textbf{Acknowledgements.}
The author would like to thank Kris Shaw and Fr{\' e}d{\' e}ric Bihan for their support, interest and useful discussions in this project, as well as Matilde Manzaroli for pointing out the reference for Orevkov's hyperbolicity criterion, Hannah Markwig for the help in Singular computations, and Georg Loho and Mateusz Skomra for discussing their new notion of signed tropical convexity.
The research of the author is supported by the Trond Mohn Stiftelse (TMS) project ``Algebraic and topological cycles in complex and tropical geometry"

\section{Tropical curves}
\label{Sec2}

We recall in this section several definitions in tropical geometry.
For more details and examples, one can read \cite[Section 2]{brugalle2015brief}, \cite[Chapter 1]{itenberg2009tropical}, \cite{maclagan2009introduction}.

\subsection{Tropical curves in $\R^2$}

A \emph{tropical polynomial in two variables} is 
\[ P(x,y) = \max\limits_{i,j \in A} (a_{i,j} + ix +jy) ,    \]
where $A$ is a finite subset of $(\Z_{\geq 0})^2$ and the coefficients $a_{i,j}$ are in the tropical semi-ring $\T = (\R \cup \{ -\infty \} , \max , +)$.
Note that we use the ``max" operation, while some references in the literature use the ``min" operation instead.
Thus, a tropical polynomial is a convex piecewise affine function, with corner locus 
\[ V_P = \{ (x_0 , y_0)\in \R^2 ~ | ~ \exists (i,j)\neq (k,l), ~ P(x_0 , y_0) =  a_{i,j} + ix_0 + jy_0  =  a_{k,l} +  kx_0 + ly_0  \} .  \]
The set $V_P$ is a 1-dimensional polyhedral complex in $\R^2$.

\begin{definition}
The \emph{weight function} is defined on the edges $e$ of $V_P$ as 
\[ w(e) := \max\limits_{(i,j) , (k,l)} ( \gcd (|i-k|,|j-l|))     \]
for all pairs $(i,j)$ and $(k,l)$ such that the value of $P(x,y)$ on $e$ is given by the corresponding monomials.
The \emph{tropical curve} $C \subset \R^2$ defined by $P(x,y)$ is the 1-dimensional polyhedral complex $V_P$ equipped with the weight function $w$ on the edges.
\end{definition} 

\subsection{Dual subdivision}

Let $P(x,y) = \max\limits_{i,j \in A} (a_{i,j} +  ix + jy)$ be a tropical polynomial defining a tropical curve $C \subset \R^2$.
The \emph{Newton polygon} of $P(x,y)$, denoted by $\Delta (P)$, is defined as the convex hull
\[ \Delta (P) = \Conv \{  (i,j) \in (\Z_{\geq 0})^2 ~ | ~ a_{i,j} \neq -\infty \} \subset \R^2 .   \]
The tropical polynomial $P$ also determines a subdivision of $\Delta (P)$, as follows.
Given $(x_0 , y_0) \in \R^2$, let
\[ \Delta(x_0 ,y_0 ) = \Conv \{ (i,j) \in (\Z_{\geq 0})^2 ~ | ~ P(x_0 , y_0) =  a_{i,j} + i x_0 +j y_0 \} \subset \Delta (P).  \]
The tropical curve $C$ induces a polyhedral decomposition of $\R^2$, and the polygon $\Delta (x_0 ,y_0 )$ only depends on the cell $F \ni (x_0 ,y_0)$ of the decomposition given by $C$.
Thus, we define the face $F^\vee := \Delta (x_0 ,y_0 )$ dual to $F$ for $(x_0 , y_0) \in F$.
The \emph{dual subdivision} $\Delta_C$ of the tropical curve $C\subset \R^2$ is the union
\[ \Delta_C := \bigcup_F  F^\vee  ,  \]
for $F$ running through the faces of the decomposition of $\R^2$ induced by $C$.

\begin{definition}
We say that a tropical curve $C$ is \emph{non-singular} if its dual subdivision $\Delta_C$ is a triangulation with each 2-dimensional cell of Euclidean area $1/2$. 
\end{definition}

Equivalently, a tropical curve $C \subset \R^2$ with Newton polytope $\Delta$ is non-singular if and only if $C$ has $2 \Area (\Delta )$ vertices.
In particular, if a tropical curve $C \subset \R^2$ is non-singular, every vertex of $C$ is 3-valent and every edge $e$ of $C$ has weight $w(e) = 1$, and all integer points of $\Delta_C$ are vertices in the subdivision.

\subsection{Puiseux series and tropicalisation}

\begin{definition}[\cite{itenberg2009tropical}]
A \emph{locally convergent generalized Puiseux series} is a formal series of the form
\[ \alpha (t) = \sum\limits_{r\in R} \alpha_r t^{r}  \]
where $R\subset \R$ is a well-ordered set, the coefficient $\alpha_r$ belongs to $\C$, and the series is convergent for $t\in \R_{>0}$ small enough.
We denote by $\K$ the set of all locally convergent generalized Puiseux series.
An element $\alpha := \alpha (t)$ of $\K$ is said to be \emph{real} if all its coefficients $\alpha_r$, with $r\in R$, belong to $\R$.
We denote by $\K_\R$ the subset of real locally convergent generalized Puiseux series.
\end{definition}

The sets $\K$ and $\K_\R$ are both fields of characteristic 0. 
The field $\K$ is algebraically closed, and the subfield $\K_\R$ is real closed (ie. its field extension by the square root of $-1$ is algebraically closed, \cite[Theorem 1.2.2]{bochnak2013real}).
The field $\K$ has a \emph{non-archimedean valuation} defined as follows (see \cite{bochnak2013real}):
\begin{align*}
\val : \K & \rightarrow \R \cup \{ \infty \} \\
0 & \mapsto \infty \\
\sum_{r\in R} \alpha_r t^r \neq 0 & \mapsto  \min_R \{ r ~ | ~ \alpha_r \neq 0 \} . 
\end{align*} 
The restriction of $\val$ to the subfield $\K_\R$ allows to define an \emph{ordering} on $\K_\R$ as follows (see \cite{bochnak2013real}).
An element $\alpha \in \K_\R$ is \emph{positive}, and we denote $\alpha > 0$, if the leading coefficient $\alpha_r \in \R, r :=  \val (\alpha )$ of $\alpha$ is positive with respect to the standard ordering of $\R$. 
Similarly, an element $\alpha \in \K_\R$ is \emph{negative}, and we denote $\alpha < 0$, if the leading coefficient $\alpha_r \in \R, r := \val (\alpha )$ of $\alpha$ is negative with respect to the standard ordering of $\R$.
Then for any two elements $\alpha , \beta \in \K_\R$, we have $\alpha < \beta$ if and only if $\beta - \alpha > 0$. 

Since the elements of $\K$ are convergent for $t\in \R_{>0}$ small enough, an algebraic hypersurface $\mathcal{X}$ over $\K$ can be seen as a one-parameter family $(\mathcal{X}_t)_t$ of algebraic hypersurfaces over $\C$, and an algebraic hypersurface $\mathcal{X}$ over $\K_\R$ can be seen as a one-parameter family $(\mathcal{X}_t)_t$ of algebraic hypersurfaces over $\R$.

\begin{definition}
The valuation map $\val$ allows to define a map from $(\K^\times )^2$ to $(\T^\times )^2$ as follows:  
\begin{align*}
\Trop_0 : (\K^\times )^2 & \rightarrow (\T^\times )^2 = \R^2 \\
(z , w) & \mapsto (-\val (z) , -\val (w)).
\end{align*}
The \emph{tropicalisation} $\Trop (p)$ of a point $p \in (\K^\times )^2$ is the image of $p$ via $\Trop_0$.
For $\mathcal{C}$ an algebraic curve in the algebraic torus $(\K^\times )^2$, the \emph{tropicalisation} $\Trop (\mathcal{C})$ of $\mathcal{C}$ is the tropical curve $C \subset \R^2$ obtained as the image of $\mathcal{C}$ via $\Trop_0$ equipped with the weight function $w$ on its edges. 
\end{definition}


\begin{theorem}[\cite{kapranov2000amoebas}]
\label{ThKapMik}
Let $\mathcal{P}$ be a polynomial in $\K [z,w]$ defined as 
\[ \mathcal{P} (z,w) = \sum_{(i,j) \in A} a_{i,j} z^i w^j . \]
Let $P$ be the tropical polynomial in $\T [x,y]$ defined as
\[ P (x,y) = \max_{(i,j) \in A} ( -\val (a_{i,j}) + ix + jy) \in \T [x,y]. \]
For $V (\mathcal{P})$ the zero set of the polynomial $\mathcal{P}$ and for $C$ the tropical curve defined by the tropical polynomial $P$, we have 
\[ \Trop (V (\mathcal{P})) = C \subset \R^2 . \]
\end{theorem}

\subsection{Compactification in the tropical projective plane}


\begin{definition}
We denote again by $\Trop_0$ the extension of the map $\Trop_0$ to the toric surface $\pr_\K^2$.
The \emph{tropical projective plane} \[ \T \pr^2 := \frac{\T^3 \backslash (-\infty , -\infty , -\infty)}{\T^\times} \] is the image of $\pr_\K^2$ via the map $\Trop_0$, such that $\T \pr^2$ has a stratification compatible with the toric stratification on $\pr_\K^2$, and $\T \pr^2$ has maps to affine charts $(\T^\times)^2 = \R^2$.
\end{definition}

A point $v$ in $\T \pr^2$ is written of the form $v = [v_0  :  v_1  :  v_2]$, with $v_i \in \T$ for $i=0,1,2$ and not all $v_i$ equal to $- \infty$.
The equivalence relation induced by $\T^\times = \R$ identifies the point $[v_0  :  v_1  :  v_2]$ with the point $[v_0 + \lambda  :  v_1 + \lambda  :  v_2 + \lambda ]$ for any $\lambda \in \R$.
The 2-dimensional simplex in \Cref{tropline} gives a representation of $\T \pr^2$.
Each black edge represents a coordinate hyperplane $x_i = -\infty$, for $i=0,1,2$.

We refer to \cite{payne2009analytification} for more details on this extended tropicalisation, and to \cite{maclagan2009introduction},\cite{mikhalkin2009tropical} for the direct construction of tropical toric varieties (ie.~ without going through tropicalisation).

\begin{definition}
Let $\Delta_d$ be the 2-dimensional lattice polytope defined as the convex hull of the points $(0,0),(d,0),(0,d)$, for $d \in \Z_{\geq 1}$.
A \emph{tropical curve of degree $d$ in $\T \pr^2$} is the compactification in $\T \pr^2$ of a tropical curve $C'$ in $\R^2$ with Newton polygon $\Delta_d$.
\end{definition}

\begin{figure}
\centering
	\includegraphics[width=0.3\textwidth]{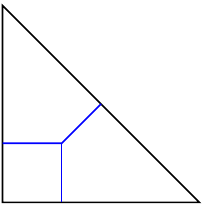}
	\caption{Tropical line in $\T \pr^2$ in \Cref{ExTropLine}.}
	\label{tropline}
\end{figure}

\begin{example}
\label{ExTropLine}
A tropical line in $\T \pr^2$ is a tropical curve of degree 1 in $\T \pr^2$.
A defining tropical polynomial of a tropical line $L$ in $\T \pr^2$ is a homogeneous tropical polynomial in $\T [x_0 , x_1 , x_2]$ of the form 
\[ P (x_0 , x_1 , x_2) = \max (x_0 + a_0 , x_1 + a_1 , x_2 + a_2) . \]
Every tropical line $L$ in $\T \pr^2$ is given as the compactification of a tropical line $L'$ in $\R^2$, such that $L'$ has defining tropical polynomial in $\T [x_1 , x_2]$ of the form 
\[ P' (x_1 , x_2) = \max (a_0 , x_1 + a_1 , x_2 + a_2) . \]
An example of tropical line in $\T \pr^2$ is represented in \Cref{tropline} inside $\T \pr^2$. 
\end{example}

\begin{definition}
Let $C$ be a non-singular tropical curve of degree $d$ in $\T \pr^2$, with $C$ the compactification of a non-singular tropical curve $C'$ in $\R^2$ with Newton polygon $\Delta_d$.
We say that an edge $e$ of $C$ is \emph{bounded} if $e$ is bounded in $C'$, and \emph{unbounded} if $e$ is unbounded in $C'$.
The set of bounded edges of $C$ is denoted $\Edge^0 (C)$.
\end{definition}

\section{Real part of a tropical curve}
\label{Sec3}

We recall in this section the notions of combinatorial patchworking, twisted edges on tropical curves, and the relation between those two notions.
One can read \cite{viro1980curves} for the first introduction of Viro's patchworking and \cite[Chapter 2]{itenberg2009tropical} for a detailed survey.
See also \cite{viro2001dequantization} for a link between amoebas and tropical geometry, and \cite{brugalle2015brief}, \cite{renaudineau2017haas}, \cite{mikhalkin2018tropical} for several approaches to twisted edges on non-singular tropical curves.


\subsection{Combinatorial patchworking}

Given a non-singular tropical curve $C$ in $\R^2$ or $\T \pr^2$ and a distribution of signs $\delta$ on the integer points of the dual subdivision $\Delta_C$, we recall the construction of the \emph{real part} of the tropical curve $C$ with respect to $\delta$.

Let $\mathcal{R}^2$ be the group of symmetries in $\R^2$ generated by the reflections with respect to the coordinate hyperplanes.
Throughout the article, for $A$ a topological set, we will denote by $A^*$ the disjoint union of symmetric copies
\[ A^* := \bigsqcup\limits_{\varepsilon \in \mathcal{R}^2} \varepsilon (A) .\]    
For $\Z_2 := \Z / 2\Z$, we identify $\mathcal{R}^2$ with $\Z_2^2$ as follows.
If $\varepsilon = (\varepsilon_1 , \varepsilon_2 ) \in \Z_2^2$, then $\varepsilon \in \mathcal{R}^2$ is the symmetry defined by 
\[ \varepsilon (x,y) = ((-1)^{\varepsilon_1} x , (-1)^{\varepsilon_2} y) , \]
for $\varepsilon_1$ and $\varepsilon_2$ seen as integers. 
For $\Delta_C$ the dual subdivision of a non-singular tropical curve $C$, a map of the form
\[ \delta : \Delta_C \cap \Z^2 \rightarrow \{ +1,-1 \} \]
is called a \emph{distribution of signs} on $\Delta_C$.
We extend a distribution of signs $\delta$ to a distribution of signs on the integer points of the union of symmetric copies $\Delta_C^*$ by the following rule.
For $v = (v_1, v_2 )$ an integer point in $\Delta_C^* \cap \Z^2$ and $\varepsilon = (\varepsilon_1 , \varepsilon_2) \in \mathcal{R}^2$ a symmetry, we have
\[
\delta (\varepsilon (v)) = (-1)^{\varepsilon_1 v_1 + \varepsilon_2 v_2} \delta (v) ,
\]
again with $\varepsilon_1$ and $\varepsilon_2$ seen as integers (see for example the signs on \Cref{RealPart}).
A face of $\Delta_C^*$ is said to be \emph{non-empty} if at least two of its vertices have opposite signs by $\delta$.
We can now construct the real part of $C$ with respect to $\delta$ as follows.
For every edge $e$ of $C$, let $e^\vee$ be its dual edge in $\Delta_C$.
The real part $\R e_{\delta}$ of $e$ is defined as 
\[ \R e_\delta := \bigcup\limits_{\varepsilon } \varepsilon (e) \subset (\R^2)^*, \]
for $\varepsilon$ running through the symmetries of $\mathcal{R}^2$ such that $\varepsilon (e^\vee )$ is non-empty (see for example the blue edges in \Cref{RealPart}).

\begin{definition}
Let $C$ be a non-singular tropical curve in $\R^2$ satisfying the same assumptions as above.
The \emph{real part of} $C$ with respect to the distribution of signs $\delta$ is the set
\[ \R C_{\delta} := \overline{ \bigcup_{e} \R e_{\delta} } \subset (\R^2 )^* , \]
for $e$ running through the edges of $C$.
\end{definition}

If $C$ is a non-singular tropical curve in $\T \pr^2$ instead of $\R^2$, the real part $\R C_\delta$ is given by taking the closure inside $(\T \pr^2)^*/ \sim$ instead of $(\R^2 )^*$, where $\sim$ is given as follows.
If $\Gamma$ is a face in the stratification of $(\T \pr^2)^*$, then for all integer vectors $(\alpha_1 , \alpha_2 )$ orthogonal to all primitive integer direction in $\Gamma$, we have $\Gamma \sim \varepsilon (\Gamma )$ for $\varepsilon = (\bar{\alpha_1},\bar{\alpha_2}) \in \mathcal{R}^2$.
Recall that the \emph{sign} of an element $\alpha \in \K_\R$ is given by the sign of the coefficient of $\alpha$ of order $\val (\alpha )$.

\begin{theorem}[Viro's patchworking theorem \cite{viro2001dequantization}]
\label{ThViro}
Let $\mathcal{C} := (\mathcal{C}_t)_t$ be a non-singular real algebraic curve in $(\K^\times )^2$ (or $\pr_\K^2$), with tropicalisation a non-singular tropical curve $C$ in $\R^2$ (or $\T \pr^2$).
Let $\delta$ be a distribution of signs on $\Delta_C$ induced by the signs of the monomials of a defining polynomial $\mathcal{P}$ of $\mathcal{C}$.
For $t>0$ small enough, there exists a homeomorphism $(\R^\times )^2 \simeq (\R^2 )^*$ (or $\pr^2 (\R ) \simeq (\T \pr^2)^* / \sim$ ) mapping the set of real points $\mathcal{C}_t (\R )$ onto $\R C_{\delta}$.
\end{theorem}

The topological set $(\T \pr^2)^*/ \sim$ is homeomorphic to the (topological) real projective plane $\R \pr^2$.
Hence along the article, we will write $\pr^2 (\R)$ for the real part of the algebraic projective plane, and we will write $\R \pr^2$ for the topological set $(\T \pr^2)^*/ \sim$. 
The homeomorphism between the real part $Y_\Sigma (\R)$ of a toric variety $Y_\Sigma$ defined by a fan $\Sigma$ and a topological set of the form $(\T Y_\Sigma)^* / \sim$, for $\T Y_\Sigma$ the tropical toric variety defined by $\Sigma$, is constructed in \cite[Theorem 11.5.4]{gelfand2014discriminants} in terms of the polytope $\Delta$ dual to $\Sigma$.  

\begin{example}
\label{ExCombPatch}
Let $C$ be a non-singular tropical curve of degree 2 in $\T \pr^2$ as pictured in \Cref{RealPhase}.
The sign distribution $\delta$ on the dual subdivision $\Delta_C$ is pictured in \Cref{SignDistrib}.
Extending the sign distribution to the symmetric copies of $\Delta_C$, we obtain the real part $\R C_\delta$ pictured in color in \Cref{RealPart}.
\end{example}

\begin{figure}
\centering
\begin{subfigure}[t]{0.3\linewidth}
	\centering
	\includegraphics[width=\textwidth]{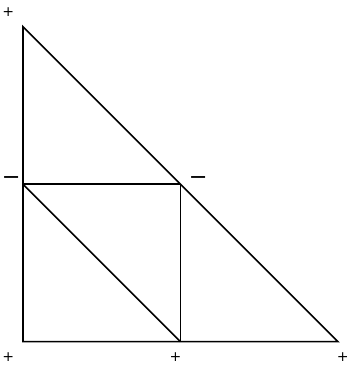}
	\caption{Sign distribution $\delta$ on $\Delta_C$.}
	\label{SignDistrib}
\end{subfigure}\hfill
\begin{subfigure}[t]{0.3\linewidth}
	\centering
	\includegraphics[width=\textwidth]{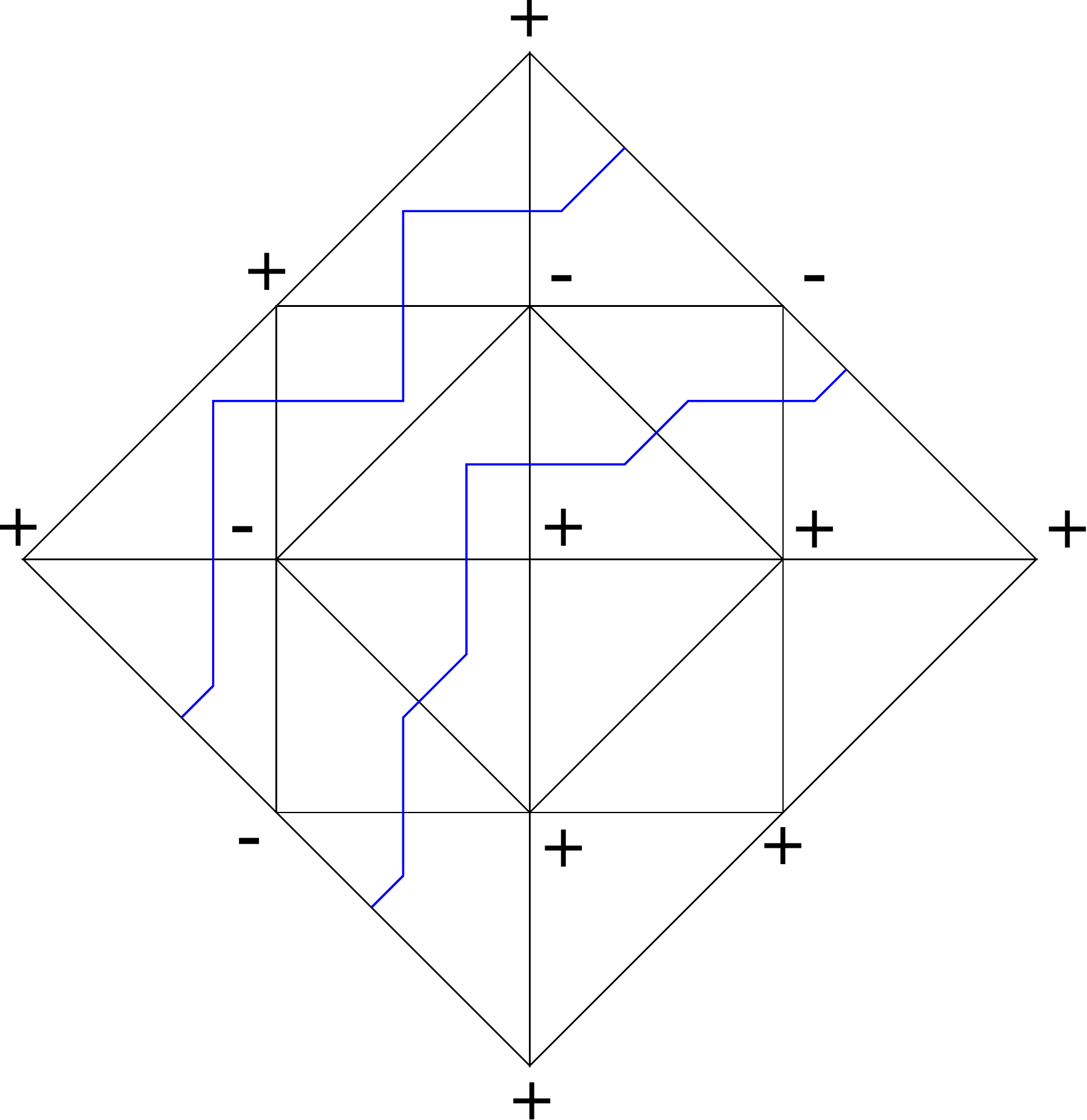}
	\caption{Real part $\R C_\delta$ of $C$ on $\pr^2 (\R )$.}
	\label{RealPart}
\end{subfigure}\hfill
\begin{subfigure}[t]{0.3\linewidth}
\centering
\includegraphics[width=\textwidth]{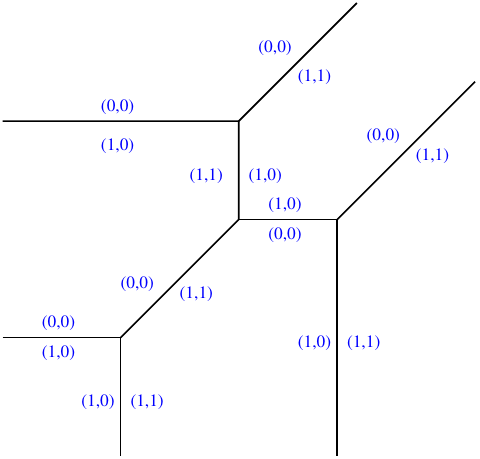}
\caption{Real phase structure on $C$.}
\label{RealPhase}
\end{subfigure}
\caption{\Cref{ExCombPatch} and \Cref{ExCombPatch2}.}
\end{figure}

\begin{definition}
\label{DefTropLimit}
Let $\mathcal{C} := (\mathcal{C}_t)_t \subset (\K^\times )^2$ be a real algebraic curve with tropicalisation a non-singular tropical curve $C\subset \R^2$ and inducing the distribution of signs $\delta$ on the dual subdivision $\Delta_C$.
We say that for $t>0$ small enough (in the sense of \Cref{ThViro}), the curve $\mathcal{C}_t \subset (\C^\times )^2$ is \emph{near the tropical limit} $(C,\delta )$.
\end{definition}

Given a distribution of signs $\delta$ on the dual subdivision $\Delta_C$ of a non-singular tropical curve $C$, \cite{renaudineau2017haas}, \cite{renaudineau2017tropical} and \cite{renaudineau2018bounding} define a real phase structure directly on $C$. 
We define this real phase structure as follows:

\begin{definition}
\label{DefRealPhase}
Let $C$ be a non-singular tropical curve in $\R^2$ or $\T \pr^2$.
For $\delta$ a distribution of signs on the dual subdivision $\Delta_C$, a \emph{real phase structure} on $C$ is a collection $\E := \{ \E_e \}_{e}$, for $e$ running through the edges of $C$, of $\Z_2$-affine spaces defined as
\[ \E_e = \{ \varepsilon \in \mathcal{R}^2 ~ | ~ \varepsilon (e ) \subset \R e_\delta   \} .  \]
The couple $(C,\E )$ is called a non-singular \emph{real tropical curve}.
A real algebraic curve $\mathcal{C}$ in $(\K^\times )^2$ or $\pr^2_\K$ with tropicalisation $C$ and such that the distribution of signs $\delta_{\mathcal{P}}$ of a defining polynomial $\mathcal{P}$ of $\mathcal{C}$ induces the real phase structure $\E$ is called a \emph{realisation} of $(C,\E)$.
\end{definition}

\begin{remark}
\label{RkRealPhase}
If $C$ is a non-singular tropical curve, we can define directly a real phase structure $\E$ on $C$ without knowing the distribution of signs on the dual subdivision, as described in \cite[Definition 3.1]{renaudineau2018bounding}.
Such a real phase structure $\E := \{ \E_e \}_e$ must satisfy the two following properties:
\begin{itemize}
\item for $e$ an edge of $C$ and $\overrightarrow{e} \in \Z_2^2$ its direction modulo 2, we have $\varepsilon + \varepsilon ' = \overrightarrow{e}$ for $\varepsilon , \varepsilon '$ the two elements of $\E_e$ (using the identification $\mathcal{R}^2 \cong \Z_2^2$);
\item for $v$ a 3-valent vertex of $C$ incident to an edge $e$, for any element $\varepsilon \in \E_e$, there exists a unique edge $e' \neq e$ such that $v$ is incident to $e'$ and $\varepsilon \in \E_{e '}$ (see for instance \Cref{amoebavertex}).
\end{itemize}  
\end{remark}

\begin{example}
\label{ExCombPatch2}
Let $C$ be a non-singular tropical curve of degree 2 in $\T \pr^2$ as pictured in \Cref{RealPhase}.
For $\delta$ the distribution of signs on $\Delta_C$ represented in \Cref{SignDistrib}, we can recover the corresponding real phase structure $\E$ on $C$, pictured in color in \Cref{RealPhase}, by considering for each edge $e$ of $C$ its ``non-empty" symmetric copies $\varepsilon (e) , \varepsilon ' (e)$ (see \Cref{RealPart}).
Notice that the real phase structure $\E$ can be constructed directly using \Cref{RkRealPhase}, and we recover (up to multiplication by $-1$ on all integer points of $\Delta_C \cap \Z^2$) the distribution of signs $\delta$ from the data of edges $e$ with $(0,0) \in \E_e$.
\end{example}

\begin{figure}
\centering
\begin{subfigure}[t]{0.3\linewidth}
\centering
	\includegraphics[width=\textwidth]{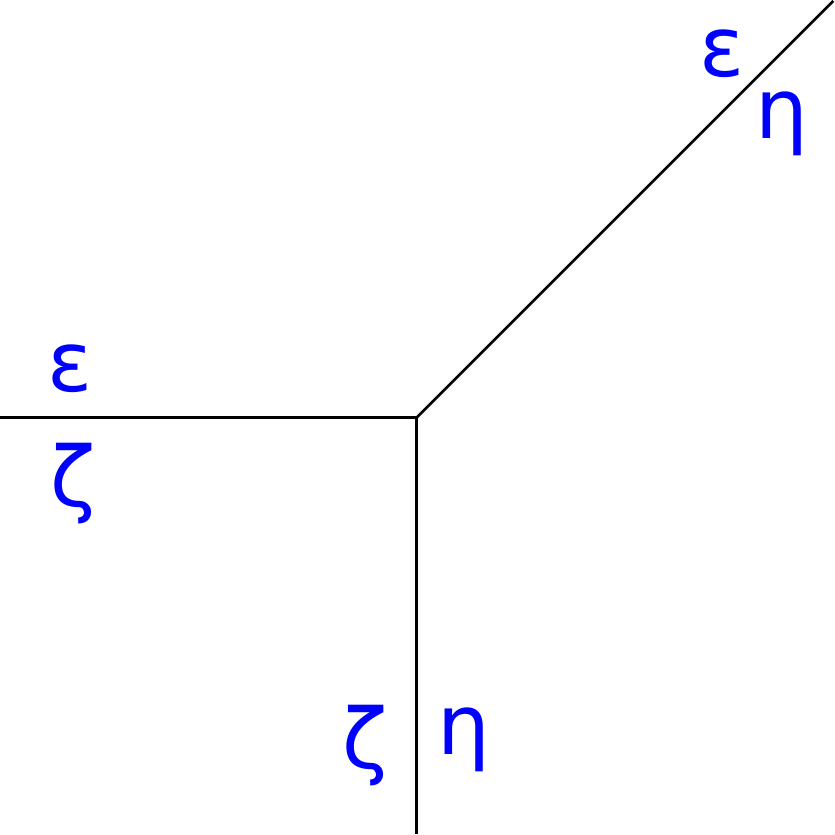}
	\caption{Real phase structure around a vertex.}
	\label{amoebavertex}
\end{subfigure}\hfill
\begin{subfigure}[t]{0.3\linewidth}
	\centering
	\includegraphics[width=\textwidth]{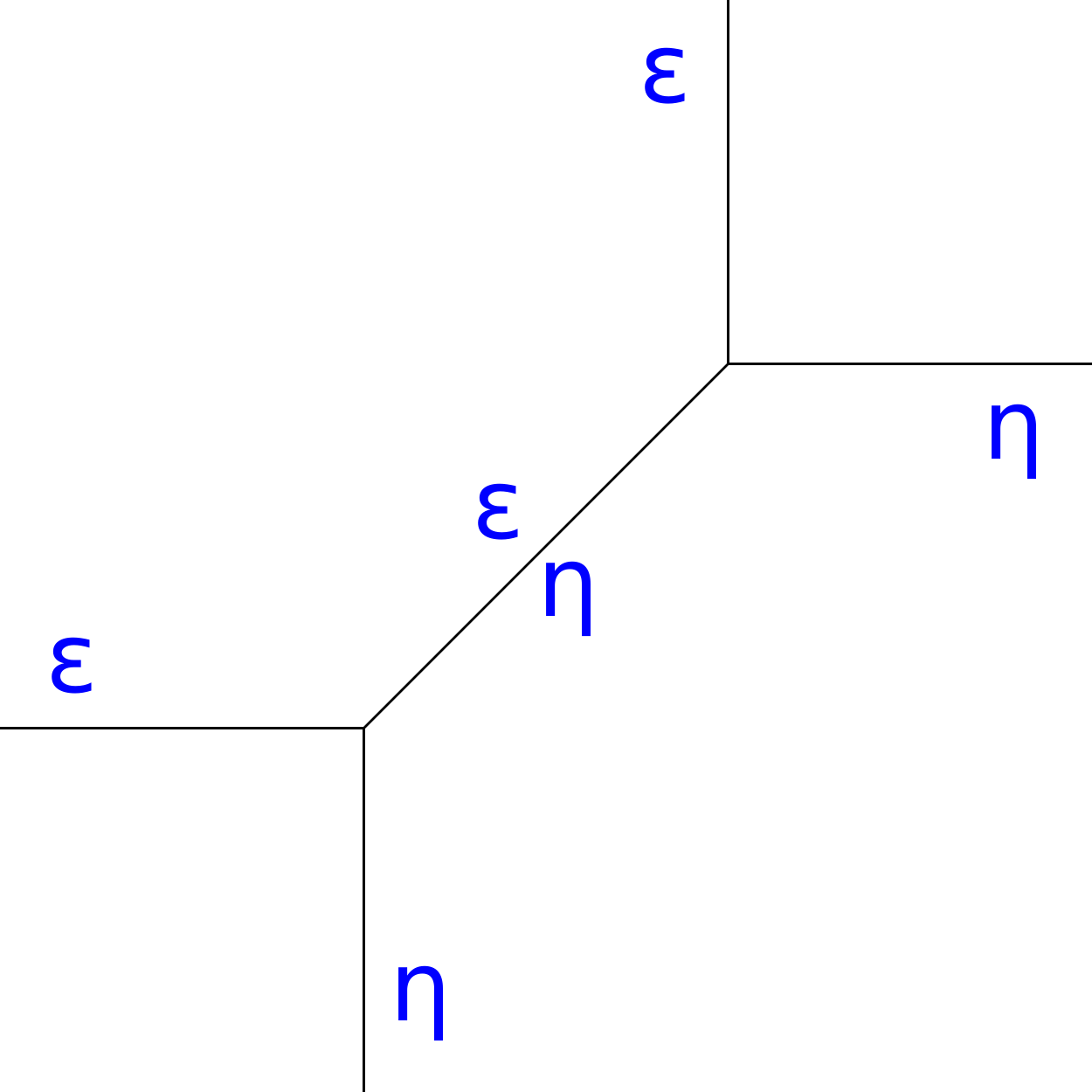}
	\caption{Real phase structure around a non-twisted bounded edge.}
	\label{amoebaedge1}
\end{subfigure}\hfill
\begin{subfigure}[t]{0.3\linewidth}
	\centering
	\includegraphics[width=\textwidth]{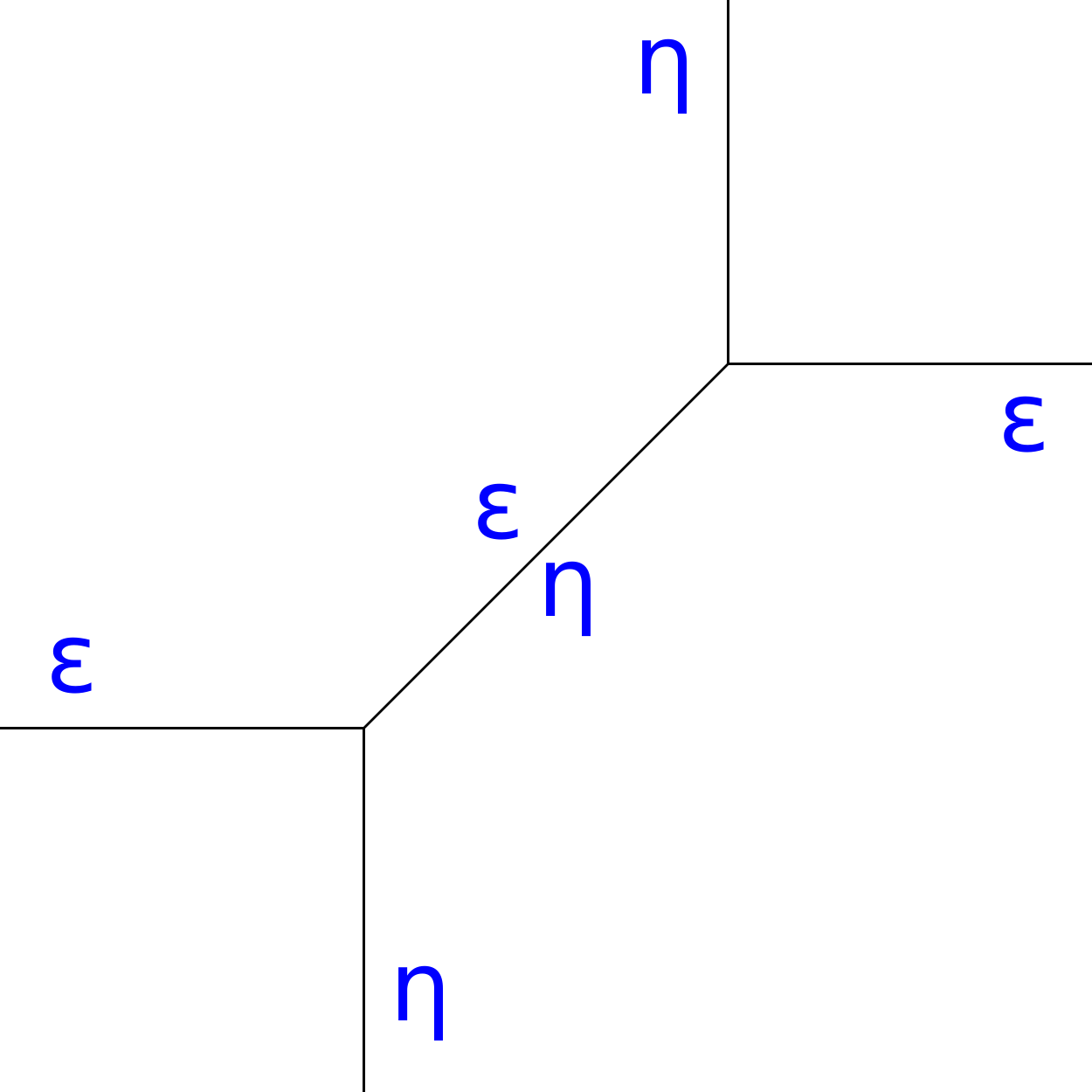}
	\caption{Real phase structure around a twisted bounded edge.}
	\label{amoebaedge2}
\end{subfigure}
\caption{Properties of real phase structure from \Cref{RkRealPhase} and \Cref{DefTwist}.}
\end{figure}

\begin{notation}
\label{NotationRealPhase}
From now on, we will write equivalently $\R C_\E$ or $\R C_\delta$ for the real part of a non-singular tropical curve $C$ with real phase structure $\E$ or sign distribution $\delta$, see \cite[Lemma 1]{renaudineau2017tropical} and \cite[Remark 3.8]{renaudineau2018bounding} for details of the connection between distribution of signs and real phase structure as defined in \Cref{RkRealPhase}.
Similarly, we will say that a real curve is \emph{near the tropical limit} $(C,\E )$ if $\delta$ determines the real phase structure $\E$.
\end{notation}

\subsection{Twisted edges on a tropical curve}

\begin{definition}
\label{DefTwist}
Let $(C,\E )$ be a non-singular real tropical curve.
A bounded edge $e$ of $C$ is said to be \emph{twisted} if for any $\varepsilon \in \E_e$, the two edges $e_1 , e_2$ of $C$ adjacent to $e$ such that $\varepsilon \in \E_{e_1} , \varepsilon \in \E_{e_1}$ lie on distinct sides of the affine line containing $e$ (see for instance \Cref{amoebaedge2}).
Otherwise, the edge $e$ is said to be \emph{non-twisted}, and satisfies the fact for any $\varepsilon \in \E_e$, the two edges $e_1 , e_2$ of $C$ adjacent to $e$ such that $\varepsilon \in \E_{e_1} , \varepsilon \in \E_{e_2}$ lie on the same side of the affine line containing $e$ that (see for instance \Cref{amoebaedge1}). 
\end{definition}  

We can define the twisted edges on a non-singular tropical curve in terms of \emph{amoebas}, see for instance \cite[Section 3.2]{brugalle2015brief} and \cite[Section 4.A]{renaudineau2017haas}. 
For an edge $e$ of $C$, we denote by $\overrightarrow{e} \in \Z^2_2$ the direction modulo 2 of $e$.
Not any subset $T$ of $\Edge^0 (C)$ may arise as the set of twisted edges. 
The possible configurations of twists must satisfy the following condition. 

\begin{definition}[\cite{brugalle2015brief},\cite{renaudineau2017haas}]
Let $C$ be a non-singular tropical curve.
Let $T$ be a subset of $\Edge^0 (C)$.
We say that $T$ is \emph{admissible}, or \emph{twist-admissible}, if for any cycle $\gamma$ of $C$, we have 
\begin{equation}
\label{EqAdm}
\sum_{e \in \gamma \cap T} \overrightarrow{e} = \overrightarrow{0} \mod 2 .
\end{equation}
\end{definition}
%

We can recover the set of twisted edges of a non-singular real tropical curve $(C,\E)$ from any distribution of signs $\delta$ determining $\E$.
Let $C$ be a non-singular tropical curve with Newton polygon $\Delta$. 
For each bounded edge $e$ of $C$, denote by $v_1^e$ and $v_2^e$ the two vertices of the edge $e^\vee$ dual to $e$ in the dual subdivision $\Delta_C$. 
The edge $e^\vee$ is contained in exactly two 2-dimensional simplices of $\Delta_C$.
Let $v_3^e$ and $v_4^e$ denote the vertices of these two simplices which are distinct from $v_1^e, v_2^e$.

\begin{proposition}[\cite{brugalle2015brief}]
\label{PropTwistSign}
Let $\delta$ be a distribution of signs on the dual subdivision $\Delta_C$, and let $e$ be a bounded edge of $C$.
\begin{itemize}
\item If the coordinates modulo 2 of $v_3^e$ and $v_4^e$ are distinct, then $e$ is twisted if and only if $\delta (v_1^e) \delta (v_2^e) \delta (v_3^e) \delta (v_4^e) = +1$; 
\item if the coordinates modulo 2 of $v_3^e$ and $v_4^e$ coincide, then $e$ is twisted if and only if $\delta (v_3^e) \delta (v_4^e) = -1$. 
\end{itemize}    
\end{proposition}

Note that we can recover the twist-admissibility condition from \Cref{EqAdm} using \Cref{PropTwistSign}.

\section{Intersection of real tropical curves}
\label{SectionIntersection}

We recall several definitions and propositions from  \cite{richter2005first} and \cite{brugalle2012inflection} concerning the intersection of tropical curves, and give some new results about the real intersection points of real algebraic curves near the tropical limit.

\subsection{Transverse real intersection}
\label{SecTransverse}
%

\begin{definition}[\cite{richter2005first}]
\label{DefTransverse}
Let $C,C'$ be two tropical curves in $\R^2$. 
A \emph{tranverse} intersection point of $C$ and $C'$ is an isolated point of the set-theoretic intersection $C\cap C'$ lying in the interior of both a closed edge of $C$ and a closed edge of $C'$.
\end{definition}

\begin{definition}[\cite{brugalle2012inflection}]
\label{DefTropIntMult}
Let $C,C'$ be two tropical curves in $\R^2$, and let $v$ be a vertex of the set-theoretic intersection $C\cap C'$.
We call $v$ a \emph{tropical intersection point}.
The \emph{tropical intersection multiplicity} $(C \circ C')_v$ of $C$ and $C'$ at $v$ is defined as
\[ (C \circ C')_v := \Area (v^\vee ) - \xi_v , \]
for $\Area (\bullet )$ the Euclidean area, for $v^\vee \subset \Delta_{C\cup C'}$ the 2-dimensional face dual to $v$ in the \emph{mixed dual subdivision} associated to the union $C\cup C'$, and $\xi_v$ satisfying:
\begin{itemize}
\item if $v$ is a transverse intersection point of $C$ and $C'$, then $\xi_v = 0$;
\item if $v$ is a vertex $v'$ of $C$ but not $C'$ (or $C'$ but not $C$), then $\xi_v = \Area ((v')^\vee)$ for $(v')^\vee \subset \Delta_C$ (or $(v')^\vee \subset \Delta_{C'}$);
\item if $v$ is both a vertex $v'$ of $C$ and a vertex $v''$ of $C'$, then $\xi_v = \Area ((v')^\vee) + \Area ((v'')^\vee)$ for $(v')^\vee \subset \Delta_C$ and $(v'')^\vee \subset \Delta_{C'}$. 
\end{itemize}  
\end{definition}

In the literature, tropical intersections points are also called \emph{stable intersections points}, see \cite{richter2005first} for a justification of this terminology.
Note that if $v$ is a transverse intersection point of $C\cap C'$, its tropical intersection multiplicity $(C\circ C')_v$ can be computed as $|\det (\overrightarrow{e} | \overrightarrow{e} ')| . w(e) . w (e')$, for $e, e'$ the edges of $C,C'$ containing $v$ and $\overrightarrow{e} , \overrightarrow{e} '$ their primitive integer directions.

\begin{definition}
\label{DefSumTropIntMult}
Let $C, C'$ be two tropical curves in $\R^2$.
Let $E$ be a connected component of the set-theoretic intersection $C\cap C'$.
The \emph{tropical intersection multiplicity} $(C \circ C')_E$ of $C$ and $C'$ at $E$ is defined as 
\[ (C \circ C')_E = \sum\limits_{v} (C \circ C')_v ,  \]
for $v$ running through the tropical intersection points of $C\cap C'$ contained in $E$.
\end{definition}

\begin{proposition}[\cite{brugalle2012inflection}]
\label{PropCpxMult}
Let $\mathcal{C}$ and $\mathcal{C}'$ be two algebraic curves in $(\K^\times )^2$, with tropicalisation $C,C'$ respectively, such that $\mathcal{C}$ and $\mathcal{C}'$ intersect in a finite number of points.
Let $E$ be a connected component of $C\cap C'$.
Then the number of intersection points (counted with multiplicity) of $\mathcal{C}$ and $\mathcal{C}'$ with tropicalisation in $E$ is less or equal to $(C \circ C')_E$, with equality if $E$ is compact.
\end{proposition}

\begin{proposition}
\label{PropRealMult}
Let $\mathcal{C}$ and $\mathcal{C}'$ be two real algebraic curves in $(\K^\times )^2$, with tropicalisation $C,C'$ respectively, such that $\mathcal{C}$ and $\mathcal{C}'$ intersect in a finite number of points.
Let $E$ be a connected component of $C\cap C'$.
Then the number of real intersection points (counted with multiplicity) of $\mathcal{C}$ and $\mathcal{C}'$ with tropicalisation in $E$ is less or equal to $(C \circ C')_E$, with equality modulo 2 if $E$ is compact.
\end{proposition}

\begin{proof}
By \Cref{PropCpxMult}, the number of intersection points of $\mathcal{C}$ and $\mathcal{C}'$ with tropicalisation in $E$ is at most $(C \circ C')_E$, counted with multiplicity.
Moreover, we have equality if $E$ is compact.
Hence the number of real intersection points with tropicalisation in $E$ is at most $(C \circ C')_E$.
Assume that $E$ is compact.
Any pair $(p , \bar{p})$ of complex conjugated points satisfies $\Trop (p) = \Trop (\bar{p})$. 
In particular, the point $p$ has tropicalisation in $E$ if and only if its complex conjugated $\bar{p}$ has tropicalisation in $E$.
Since the intersection of the real curves $\mathcal{C}$ and $\mathcal{C}'$ consists of a finite number of points, we can split those points into the real ones and the pairs of complex conjugated ones.
Then for $E$ compact, the number of real points with tropicalisation in $E$ is equal to $(C \circ C')_E$ modulo 2.
\end{proof}

From now on, we will consider intersection of non-singular real tropical curves.  

\begin{proposition}
\label{CorTransMult2}
Let $(C,\E )$ and $(C' ,\E ')$ be two non-singular real tropical curves in $\R^2$, such that there exists a transverse intersection point $u$ in $C\cap C'$ of tropical intersection multiplicity $(C\circ C')_u= m\geq 1$.
Let $e$ be the edge of $C$ containing $u$, and let $e'$ be the edge of $C'$ containing $u$.
For every realisation $\mathcal{C}$ and $\mathcal{C}'$ in $(\K^\times )^2$ of $(C,\E )$ and $(C' ,\E ')$, there are exactly $m$ distinct points $p_1 , \ldots , p_m \in (\K^\times )^2$ with $\Trop (p_i) = u$ in the intersection $\mathcal{C} \cap \mathcal{C'}$, partitioned as follows,
\begin{itemize}
\item if $m$ is even, the set $\{ p_1 , \ldots , p_m \}$ consists of 2 real points and $\frac{m-2}{2}$ pairs of complex conjugated points if and only if $\E_e = \E_{e'}'$, and it consists of $\frac{m}{2}$ pairs of complex conjugated points if and only if $\E_e \neq \E_{e'}'$;
\item if $m$ is odd, the set $\{ p_1 , \ldots , p_m \}$ consists of 1 real point and $\frac{m-1}{2}$ pairs of complex conjugated points.
\end{itemize}
\end{proposition}

Note that Brugall\'{e} and L\'{o}pez de Medrano already proved that all the intersection points lifted from a transverse intersection point are of multiplicity 1, see \cite[Lemma 3.13]{brugalle2012inflection}.
Moreover, the condition $\E_e \neq \E_{e'}'$ in the case $m$ even is equivalent to $\E_e \cap \E_{e'}' = \emptyset$, since in that case the two edges have the same primitive integer direction modulo 2.

\begin{proof}[Proof of \Cref{CorTransMult2}]
We use a similar setting as the proof of \cite[Lemma 3.13]{brugalle2012inflection}.
After an affine linear transformation, we can assume that the edge $e$ has primitive integer direction $(1,1)$, the edge $e'$ has primitive integer direction $(1,-m+1)$, and those edges intersect at the point $(0,0) \in \R^2$.
Since the tropical curves are non-singular, the edges have all weight 1.
Then the local equations in $\K_\R [z,w]$ of $\mathcal{C}$ and $\mathcal{C}'$ are of the form
\begin{align}
\label{EqTrans1} z^k w^l (\delta_{k+1,l}z+ \delta_{k,l+1}w ) & = 0 \\
\label{EqTrans2} z^p w^q (\varepsilon_{p+m-1,q+1}z^{m-1}w+\varepsilon_{p,q}) & = 0
\end{align}
with the coefficients $\varepsilon_{i,j}, \delta_{i,j} \in \K_\R^\times$ of valuation 0.
Since we consider the solutions inside $(\K^\times )^2$, we obtain from \Cref{EqTrans1} that $w = -\frac{\delta_{k+1,l}}{\delta_{k,l+1}}z$. 
Again since we consider the solutions inside $(\K^\times )^2$, we can substitute $w = -\frac{\delta_{k+1,l}}{\delta_{k,l+1}}z$ in the second term of \Cref{EqTrans2} to obtain the following new equation in one variable:
\begin{align*}
\left( \frac{-\varepsilon_{p+m-1,q+1}\delta_{k+1,l}}{\delta_{k,l+1}}  \right) z^m  +  \varepsilon_{p,q} & = 0 \\
\Leftrightarrow z^m - \left( \frac{\varepsilon_{p,q} \delta_{k,l+1} }{\varepsilon_{p+m-1,q+1}\delta_{k+1,l}}  \right) & = 0 .
\end{align*}

The coefficient $\left( \frac{\varepsilon_{p,q} \delta_{k,l+1} }{\varepsilon_{p+m-1,q+1}\delta_{k+1,l}} \right)$ is non-zero, hence the polynomial above (seen in $\K_\R [z]$) has $m$ distinct roots of multiplicity 1 in $\K^\times$. 
Moreover, if $m$ is odd, we have exactly one real root and $\frac{m-1}{2}$ pairs of complex conjugated roots, while if $m$ is even, we have exactly two real roots if and only if $\left( \frac{\varepsilon_{p,q} \delta_{k,l+1} }{\varepsilon_{p+m-1,q+1}\delta_{k+1,l}} \right) > 0$ and no real roots if and only if $\left( \frac{\varepsilon_{p,q} \delta_{k,l+1} }{\varepsilon_{p+m-1,q+1}\delta_{k+1,l}} \right) < 0$.
For $m$ even, seeing the signs of the coefficients $\varepsilon_{ij} , \delta_{kl}$ as the distributions of signs on the corresponding dual subdivisions of the local model, we obtain two real roots if and only if $\E_e = \E_{e'}'$ and no real roots if and only if $\E_e \neq \E_{e'}'$.
\end{proof}

\subsection{Twists and non-transverse real intersection}

\label{SecNonTransverse}

Let us now consider non-singular real tropical curves $(C,\E) , (C' , \E ')$ with a non-transverse connected component $e\subset C\cap C'$, where $e$ is a bounded edge of $C$ contained in the interior of a closed edge of $C'$ (see for example \Cref{nontransverse2}). 
Since the tropical curves $C$ and $C'$ are non-singular, we obtain that $(C\circ C')_e = 2$ (using the balancing condition and the fact that all edges have weight 1). 
By \Cref{PropCpxMult} and \Cref{PropRealMult}, every realisation of $(C,\E)$ and $(C' , \E ')$ intersect in two points (counted with multiplicity) with tropicalisation in $e$, which are either two distinct real points, two distinct complex conjugated points, or a multiplicity 2 real point with tropicalisation in $e$.
Note that we cannot obtain a multiplicity 2 non-real intersection point $p$, as this would imply that its complex conjugated $\overline{p}$ is also a multiplicity 2 intersection point, hence by \Cref{PropCpxMult} we would obtain $(C\circ C')_e > 2$, leading to a contradiction.

\begin{theorem}[\Cref{IntroTh4.3.9}]
\label{Th4.3.9}
Let $(C,\E ) , (C' , \E ')$ be two non-singular real tropical curves in $\R^2$ such that there exists a closed bounded edge $e$ of $C$ contained in the interior of a closed edge $e'$ of $C'$ (see \Cref{nontransverse2}).
\begin{enumerate}
\item \label{Item4.3.9.1} If $\E_e \neq \E_{e '} '$, or is $\E_e = \E_{e'}'$ and $e$ is twisted, then all realisations of $(C,\E )$ and $(C' , \E ')$ intersect in two distinct real intersection points of multiplicity 1 with tropicalisation in $e$.
\item \label{Item4.3.9.2} If $\E_e = \E_{e '} '$ and $e$ is non-twisted, then a realisation of $(C,\E )$ can intersect a realisation of $(C' , \E ')$ in either two distinct real points of multiplicity 1 with tropicalisation in $e$, two distinct complex conjugated points of multiplicity 1 with tropicalisation in $e$, or a multiplicity 2 real point with tropicalisation in $e$.
Moreover, there exist infinitely many realisations satisfying the first two intersection types, and there exist exactly two pairs of realisations satisfying the third intersection type.
\end{enumerate}
\end{theorem}

\begin{figure}
\centering
\begin{subfigure}[t]{0.4\textwidth}
	\centering
	\includegraphics[width=\textwidth]{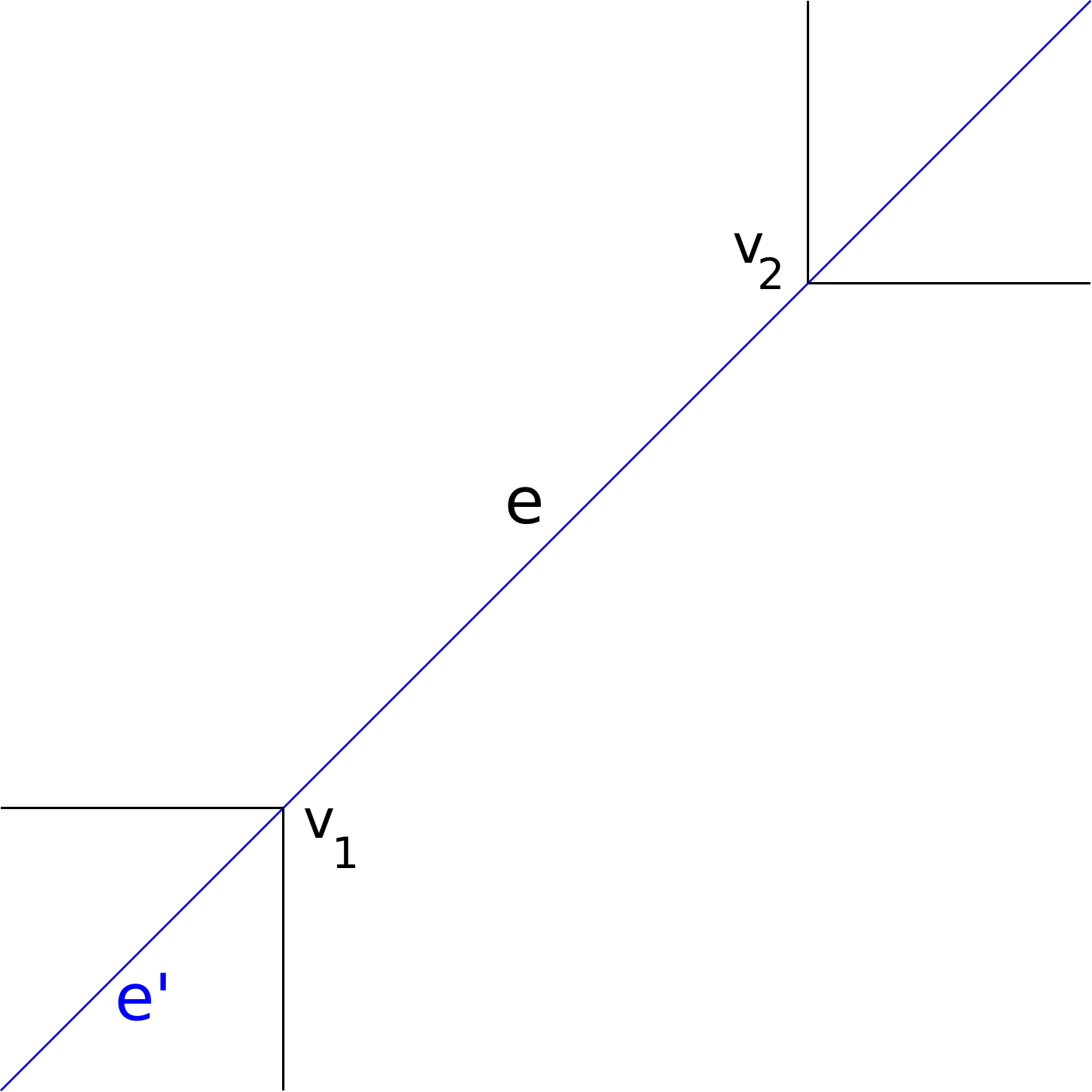}
	\caption{A non-transverse intersection of $C$ (in black) with $C'$ (in blue).}
	\label{nontransverse2}
\end{subfigure}\hfill
\begin{subfigure}[t]{0.4\textwidth}
	\centering
	\includegraphics[width=\textwidth]{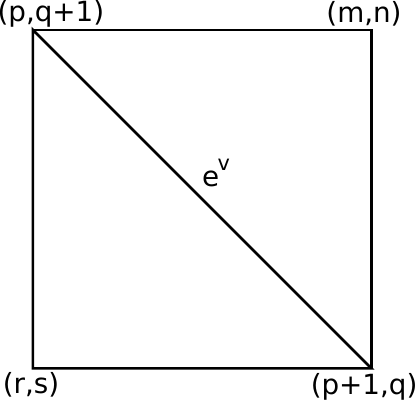}
	\caption{Local dual subdivision of $C$.}
	\label{FigDualNonTransverse2}
\end{subfigure}
\caption{Intersection in \Cref{Th4.3.9}.}
\label{FigNonTransverse2}
\end{figure}

\begin{proof}\renewcommand{\qedsymbol}{}

After an affine linear transformation, we can assume that the edges $e$ and $e'$ have primitive integer direction $(1,1)$ and that the vertices of $e$ are $v_1 = (0,0)$ and $v_2 = (b,b)$ with $b>0$.
Then the local equations of some realisations in $(\K^\times )^2$ of $(C,\E)$ and $(C' ,\E ')$ are of the form
\begin{align}
\label{EqTwist1} z^{k} w^l \left( \varepsilon_{k+1,l}z + \varepsilon_{k,l+1} w \right) & = 0 \\
\label{EqTwist2} \delta_{m,n}t^{b}z^m w^n + \delta_{p+1,q}z^{p+1} w^q + \delta_{p,q+1} z^{p} w^{q+1} + \delta_{r,s} z^r w^s & = 0
\end{align}
with $m+n = p+q+2=r+s+2$ and the coefficients $\varepsilon_{ij}, \delta_{ij} \in \K_\R^\times$ are of valuation $0$ (see \Cref{FigDualNonTransverse2}).
Note that the conditions on the exponents are implied by the fact that $C$ is non-singular.
In $(\K^\times )^2$, \Cref{EqTwist1} is equivalent to $w = -\frac{\varepsilon_{k+1,l}}{\varepsilon_{k,l+1}}z$.
After substituting $w = -\frac{\varepsilon_{k+1,l}}{\varepsilon_{k,l+1}}z$ in \Cref{EqTwist2}, we can factor the left-hand side by $z^{p+q}$, so that the solutions of the new equation in $\K^\times$ are the roots of the following degree 2 polynomial in $\K_\R [z]$:
\begin{align*}
& z^2 \left( \delta_{m,n} \left( \frac{-\varepsilon_{k+1,l}}{\varepsilon_{k,l+1}}\right)^n  t^{b} \right) \\ 
+ & z \left( \delta_{p+1,q}\left(  \frac{-\varepsilon_{k+1,l}}{\varepsilon_{k,l+1}} \right)^{q} + \delta_{p,q+1} \left(  \frac{-\varepsilon_{k+1,l}}{\varepsilon_{k,l+1}} \right)^{q+1} \right) \\
+ & \delta_{r,s} \left( \frac{-\varepsilon_{k+1,l}}{\varepsilon_{k,l+1}}\right)^s .
\end{align*}
By Tarski-Seidenberg's principle \cite[Theorem 1.4.2]{bochnak2013real} applied to the real closed field $\K_\R$, we can compute the discriminant in order to determine possible real roots.
The discriminant with respect to $z$ is $D = B^2 - 4 A t^{b}$, with 
\[ B:= \left( \delta_{p+1,q}\left(  \frac{-\varepsilon_{k+1,l}}{\varepsilon_{k,l+1}} \right)^{q} + \delta_{p,q+1} \left(  \frac{-\varepsilon_{k+1,l}}{\varepsilon_{k,l+1}} \right)^{q+1} \right) \]
and 
\[  A:= \delta_{m,n} \delta_{r,s} \left(  \frac{-\varepsilon_{k+1,l}}{\varepsilon_{k,l+1}} \right)^{n+s} . \]
Note that $\val (B) \geq 0$, and $\val (B) = 0$ unless the leading terms of the two summands cancel each other. 

We can now prove the two statements of the theorem.
\begin{enumerate}[label=(\Roman*)]
\item \label{ProofDisjoint} Assume that $\E_e \neq \E_{e'}'$ (first case of \Cref{Item4.3.9.1}).
Then the product \[ \delta_{p+1,q} \delta_{p,q+1}\varepsilon_{k+1,l}\varepsilon_{k,l+1} \] is negative, hence the leading coefficients of the two summands of $B$ do not cancel.
Thus $\val (B) = 0$, so the sign of the discriminant $D$ is determined by the sign of $B^2$, hence $D$ is positive.
Therefore, all realisations of $(C,\E)$ and $(C',\E ')$ intersect in two distinct real points with tropicalisation in $e$.

\item \label{ProofTwist} Assume that $\E_e = \E_{e'}'$ and $e$ is twisted (second case of \Cref{Item4.3.9.1}).
We have $n+s = 0 \mod 2$ if and only if $(m,n) = (r,s) \mod 2$.
Then by \Cref{PropTwistSign}, since $e$ is twisted, the discriminant $D$ is positive.
Therefore, all realisations of $(C,\E)$ and $(C',\E ')$ intersect in two distinct real points with tropicalisation in $e$.

\item \label{ProofNonTwist} Assume that $e$ is non-twisted and $\E_e = \E_{e'}'$.
Using \Cref{PropTwistSign}, we obtain the sign of the discriminant depending on the valuation of $B$ and the values of each summand of $D$ as follows.
The discriminant $D$ is positive if either $\val (B) < \frac{b}{2}$ or $\val (B) = \frac{b}{2}$ and $B^2 > 4 A t^{-b}$, and in that case all realisations of $(C,\E)$ and $(C',\E ')$ intersect in two distinct real points with tropicalisation in $e$.
Moreover, the condition $\val (B) < \frac{b}{2}$ is open, hence there exist infinitely many such realisations.
The discriminant $D$ is negative if either $\val (B) > \frac{b}{2}$ or $\val (B) = \frac{b}{2}$ and $B^2 < 4 A t^{-b}$, and in that case all realisations of $(C,\E)$ and $(C',\E ')$ intersect in two distinct complex conjugated points with tropicalisation in $e$.
Moreover, the condition $\val (B) > \frac{b}{2}$ is open, hence there exist infinitely many such realisations.
The discriminant $D$ is zero if $\val (B) = \frac{b}{2}$ and $B^2 = 4 A t^{-b}$, and in that case all realisations of $(C,\E)$ and $(C',\E ')$ intersect in a multiplicity 2 real point with tropicalisation in $e$.
Moreover, the condition $\val (B) = \frac{b}{2}$ is closed, hence the solution set is of dimension 0, and the discriminant $D$ is a degree 2 polynomial with coefficients in $\K_\R$, hence there exist exactly two pairs of realisations of $(C,\E)$ and $(C',\E ')$ intersecting in a multiplicity 2 real point with tropicalisation in $e$. \quad \quad \quad \quad $\blacksquare$ 
\end{enumerate} 
\end{proof}

We obtain from the proof of \Cref{Th4.3.9} the following corollaries.

\begin{corollary}
\label{CorDistinctPhaseTwist}
Let $(C,\E )$ and $(C' , \E ')$ be two non-singular real tropical curves in $\R^2$ such that there exists a closed bounded edge $e$ of $C$ contained in the interior of an edge $e'$ of $C'$.
Let $v_1$ and $v_2$ be the vertices of $e$.
Let $\mathcal{C}$ and $\mathcal{C'}$ be realisations in $(\K^\times )^2$ of $(C,\E )$ and $(C' , \E ')$.
If $\E_e \neq \E_{e'}'$, then for $p_1 , p_2 \in (\K_\R^\times )^2$ the two distinct real points of multiplicity 1 in $\mathcal{C} \cap \mathcal{C'}$ with tropicalisation in $e$, we have (up to renumbering) $\Trop (p_1) = v_1$ and $\Trop (p_2) = v_2$. 
\end{corollary}

\begin{proof}
Recall \Cref{ProofDisjoint} from the proof of \Cref{Th4.3.9}.
By \cite[Corollary 4.7]{brugalle2012inflection}, the two points $p_1 , p_2$ have tropicalisations of the form $\Trop (p_1) = (a,a)$ and $\Trop (p_2) = (b-a , b-a)$, for $a \in [0,b]$.
Let $z_1 ' = z_1 t^{-a}$ and $z_2 ' = z_2 t^{-b+a}$ be the roots of the degree 2 polynomial described in the proof of \Cref{Th4.3.9}, such that $\val (z_1) = \val (z_2) = 0$.
Evaluating the polynomial in $z_1 '$, we obtain the following equation
\begin{align*}
& z_1^2 t^{b-2a} \left( \delta_{m,n} \left( \frac{-\varepsilon_{k+1,l}}{\varepsilon_{k,l+1}}\right)^n \right) \\ 
+ & z_1 t^{-a} \left( \delta_{p+1,q}\left(  \frac{-\varepsilon_{k+1,l}}{\varepsilon_{k,l+1}} \right)^{q} + \delta_{p,q+1} \left(  \frac{-\varepsilon_{k+1,l}}{\varepsilon_{k,l+1}} \right)^{q+1} \right) \\
+ & \delta_{r,s} \left( \frac{-\varepsilon_{k+1,l}}{\varepsilon_{k,l+1}}\right)^s = 0 .
\end{align*}
Since $\val (B)= 0$ in \Cref{ProofDisjoint}, the equation implies that the minimum in the set $\{ b-2a , -a , 0 \}$ must be reached at least twice.
This is possible if and only if either $a= 0$ or $a = b$.
Therefore, up to renumbering, the intersection points $p_1$ and $p_2$ have tropicalisation $\Trop (p_1) = (0,0) = v_1$ and $\Trop (p_2) = (b,b) = v_2$, the two vertices of the edge $e$.
\end{proof}

The following corollary follows from \Cref{Item4.3.9.2} of \Cref{Th4.3.9}, and will be useful for proving conditions for hyperbolicity.

\begin{corollary}
\label{CorNonTwist}
Let $(C,\E )$ and $(C' , \E ')$ be two non-singular real tropical curves in $\R^2$ such that there exists a closed bounded edge $e$ of $C$ contained in the interior of an edge $e'$ of $C'$.
If $e$ is non-twisted and $\E_e = \E_{e '} '$, there exist distinct realisations $\mathcal{C}_+ , \mathcal{C}_- \subset (\K^\times )^2$ of $(C,\E )$ and distinct realisations $\mathcal{C}_+ ' , \mathcal{C}_- ' \subset (\K^\times )^2$ of $(C' , \E ')$ such that the intersection $\mathcal{C}_+ \cap \mathcal{C}_+ '$ has two distinct real points of multiplicity 1 with tropicalisation in $e$, and the intersection $\mathcal{C}_{-} \cap \mathcal{C}_{-} '$ has two distinct complex conjugated points of multiplicity 1 with tropicalisation in $e$.
\end{corollary}

We obtain an other corollary in the case of multiplicity 2 real intersection points, which follows from \Cref{ProofDisjoint}, \Cref{ProofTwist} and \Cref{ProofNonTwist} of the proof of \Cref{Th4.3.9}.

\begin{corollary}
\label{CorNonTwistTang}
Let $(C,\E )$ and $(C' , \E ')$ be two non-singular real tropical curves in $\R^2$ such that there exists a closed bounded edge $e$ of $C$ contained in the interior of an edge $e'$ of $C'$.
If there exist realisations $\mathcal{C}_0$ and $\mathcal{C}_0 '$ in $(\K^\times )^2$ of $(C,\E )$ and $(C' , \E ')$ intersecting in a multiplicity 2 real point with tropicalisation in $e$, then $e$ is non-twisted and $\E_e = \E_{e '} '$.
\end{corollary}

\begin{example}
\label{ExBitangent}

\begin{figure}
\centering
\includegraphics[width=0.5\textwidth]{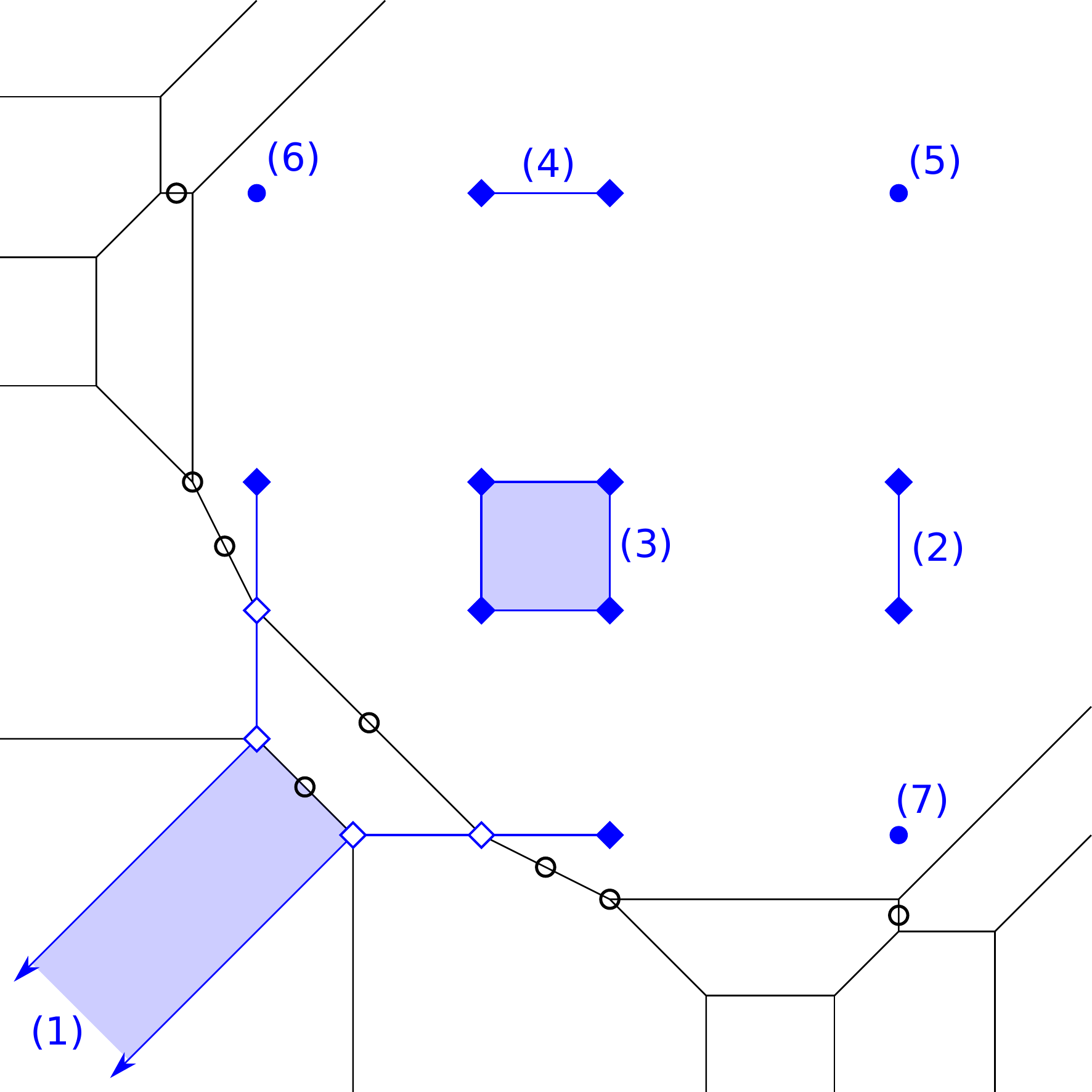}
\caption{Bitangent classes on the tropical quartic curve from \Cref{ExBitangent}}
\label{FigBitangentQuartic}
\end{figure}

Let $C$ be a non-singular tropical curve of degree 4 in $\T \pr^2$ as pictured in \Cref{FigBitangentQuartic}.
Let $\E$ be a real phase structure on $C$.
The tropical curve  $C$ has the same dual subdivision $\Delta_C$ as the tropical curve studied in \cite[Example 1.3]{cueto2020combinatorics}.
The blue zones correspond to the \emph{bitangent classes} of $C$ studied in \cite{baker2016bitangents}, \cite{len2020lifting},\cite{cueto2020combinatorics}.
In this example, any point lying in a bitangent class is the vertex of a tropical line $L$ such that the set-theoretic intersection $C\cap L$ consists of 2 connected components $E,E'$ with $(C\circ L)_E = (C\circ L)_{E'} = 2$.
Cueto and Markwig \cite[Theorem 1.2]{cueto2020combinatorics} proved that any bitangent class on a non-singular tropical quartic in $\T \pr^2$ is realised by either 0 or 4 real bitangents.
Moreover, they gave a classification of distribution of signs conditions in order for a bitangent class to lift to 4 (totally) real bitangents.

Throughout the example, whenever we mention an edge $e$ of $C$, we are referring to those containing a black dot on \Cref{FigBitangentQuartic}.
For any tropical line $L$ with vertex $v$ in the bitangent class $(3)$ such that $v$ is not a corner of the class (hence at least one of the intersection components $E,E'$ is a transverse intersection point), we obtain by \Cref{CorTransMult2} that at least one of the two intersection points is realised by two distinct points of multiplicity 1, hence $L$ cannot lift to a bitangent.
Then only the 4 corners of $(3)$ could give rise to some real bitangent. 
For each tropical line $L$ with vertex one of those corners, there exists exactly one real phase structure $\E '$ on $L$ so that the real tropical line $(L,\E ')$ can be realised by a real bitangent (one could see this by computing an extremal case of \Cref{CorTransMult2} for each isolated vertex of $C$ reached by $L$, also marked by black dots), and those are realised by exactly one totally real bitangent (ie. ~ bitangent in two real points) according to \cite[Example 1.3]{cueto2020combinatorics}.

Similarly, any tropical line $L$ with vertex $v$ in the bitangent class $(1)$, such that $v$ is not a filled blue point of $(1)$, cannot lift to a bitangent (again by \Cref{CorTransMult2}).
A real tropical line $(L,\E ')$ with vertex a filled blue point in $(1)$ can lift to some totally real bitangents only if the edge $e$ of $C$ contained in the intersection $C\cap L$ is not twisted and the real phase structures $\E$ and $\E '$ agree on $e$ by \Cref{CorNonTwistTang}. 
There are two possible real phase structures $\E'$ on $L$ agreeing with the real phase structure $\E$ on the edge $e$, hence each such real tropical line $(L,\E ')$ could be realised by a totally real bitangent, and those are realised by exactly one totally real bitangent according to \cite[Example 1.3]{cueto2020combinatorics}.

Using \Cref{CorTransMult2} and \Cref{CorNonTwistTang} in the same way as for the bitangent class $(1)$, we obtain for the bitangent classes $(2)$ and $(4)$ that only the tropical lines $L$ with vertex an endpoint of $(2), (4)$ can lift to a bitangent, and if the edge $e$ contained in the intersection $C\cap L$ is not twisted, there are two possible real phase structures $\E'$ for each $L$ so that the realisations of $(L,\E ')$ can be tangent in two real points.
According to \cite[Example 1.3]{cueto2020combinatorics}, each of these real tropical lines $(L,\E')$ is realised by exactly one totally real bitangent.

Finally, using \Cref{CorNonTwistTang}, the real tropical lines $(L,\E ' )$ with vertex $(5), (6)$ or $(7)$ can be realised by totally real bitangents only if the two edges of $C$ contained in $C\cap L$ are not twisted and the real phase structures $\E$ and $\E '$ agree on those edges. 
If those conditions are satisfied, the bitangent classes $(5), (6) , (7)$ lift to 4 totally real bitangents according to \cite[Example 1.3]{cueto2020combinatorics}.
\end{example}

\subsection{Relative twists and non-transverse real intersection}

\label{SecNonTransverseRel}

Let us finally consider non-singular real tropical curves $(C,\E) , (C' , \E ')$ with a non-transverse connected component $E \subset C\cap C'$, where $E$ is a segment (non-reduced to a point) strictly contained in both a closed edge of $C$ and a closed edge of $C'$ (see for instance the underlying local model in \Cref{FigDefRel}).
Since the tropical curves $C$ and $C'$ are non-singular, we obtain that $(C\circ C')_E = 2$, again using the balancing condition and the fact that all edges have weight 1. 
By \Cref{PropCpxMult} and \Cref{PropRealMult}, every realisation of $(C,\E)$ and $(C' , \E ')$ intersect in two points (counted with multiplicity) with tropicalisation in $E$, which are either two distinct real points, or two distinct complex conjugated points, or a multiplicity 2 real point with tropicalisation in $E$.
Using the following definition, we can compute the real intersection points in a similar way as in \Cref{Th4.3.9}.

\begin{definition}
\label{DefRelTwist}
Let $(C,\E )$ and $(C' , \E ')$ be two non-singular real tropical curves in $\R^2$ such that there exists a non-transverse connected component $E \subset C\cap C'$ which is a segment (non-reduced to a point) strictly contained in both a closed edge $e$ of $C$ and a closed edge $e'$ of $C'$ (see for example the underlying local model in \Cref{FigDefRel}).
If $\E_e = \E_{e'}'$, we say that the segment $E$ is \emph{relatively twisted} if the edges of $C\cup C'$ adjacent to $E$ and on distinct sides of the affine line containing $E$ share a common element of $\Z_2^2$ coming from $\E \cup \E '$ (see \Cref{FigRelTwist}).
\end{definition}

\begin{figure}
\centering
\begin{subfigure}[t]{0.435\textwidth}
	\centering
	\includegraphics[width=\textwidth]{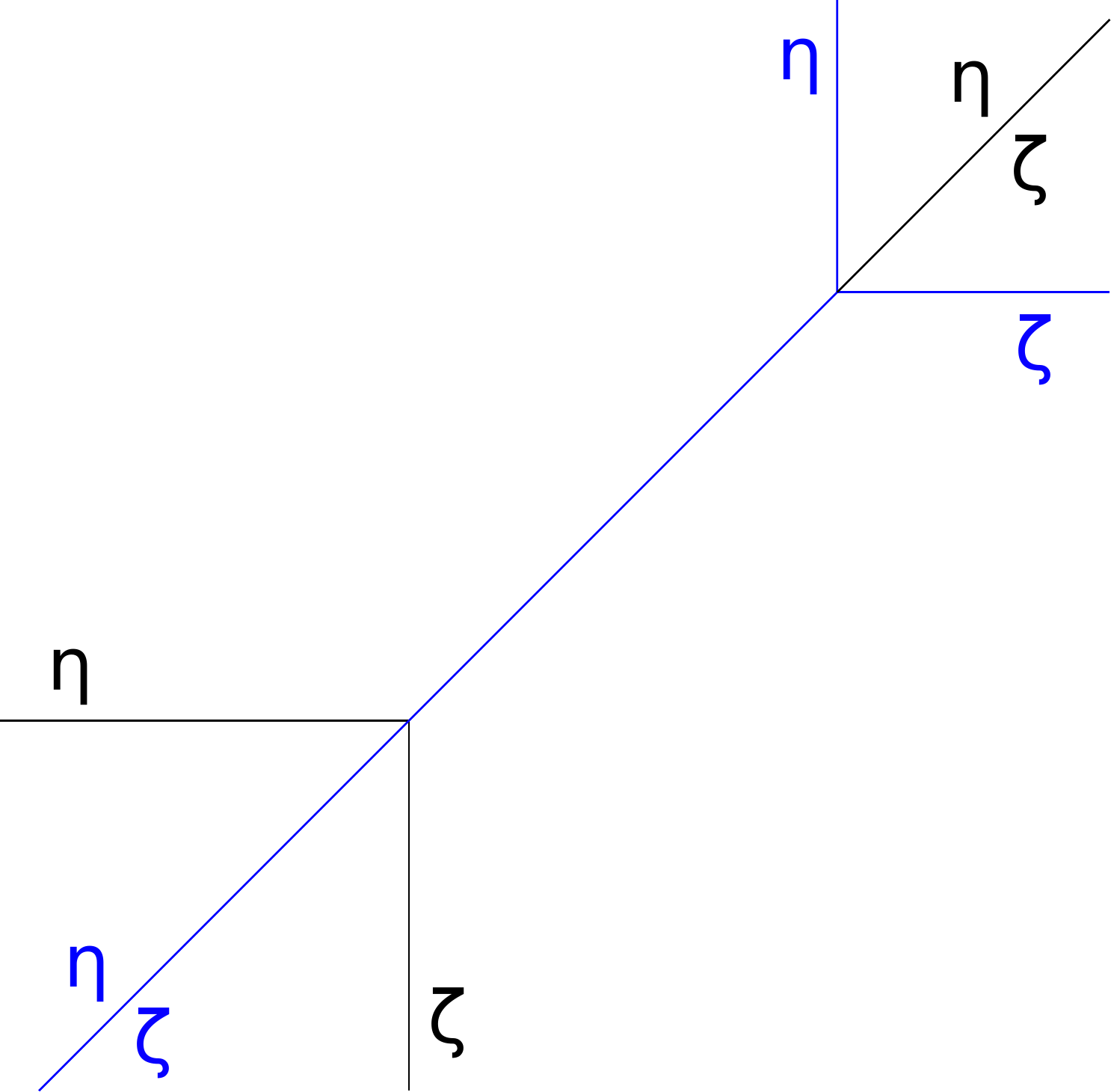}
	\caption{Relatively non-twisted intersection.}
	\label{FigRelNonTwist}
\end{subfigure}\hfill
\begin{subfigure}[t]{0.435\textwidth}
	\centering
	\includegraphics[width=\textwidth]{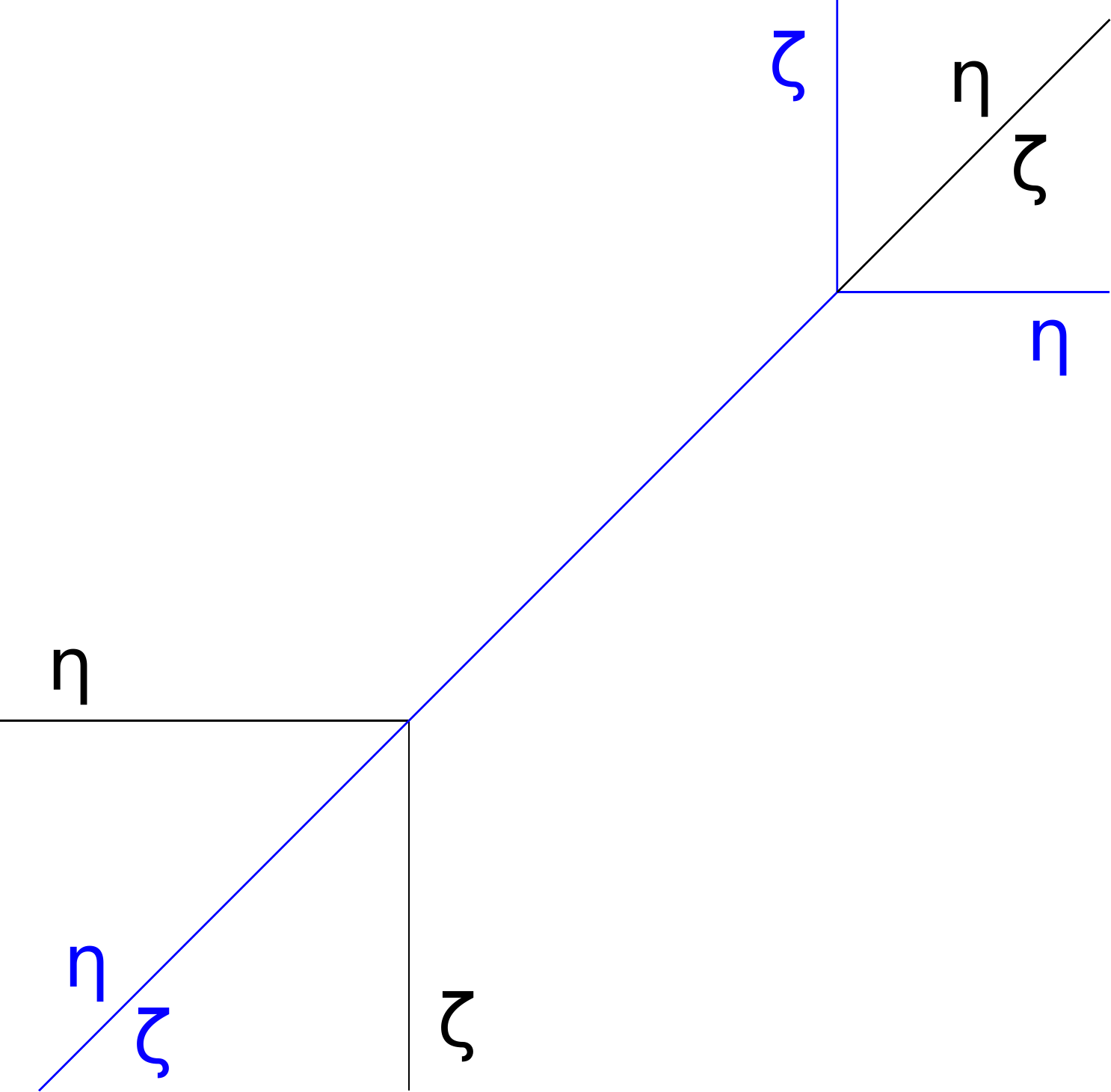}
	\caption{Relatively twisted intersection.}
	\label{FigRelTwist}
\end{subfigure}
\caption{\Cref{DefRelTwist}.}
\label{FigDefRel}
\end{figure}

Let $(C,\E )$ and $(C' , \E ')$ be two non-singular real tropical curves in $\R^2$, with associated distributions of signs $\delta , \delta '$, such that there exists a non-transverse connected component $E \subset C\cap C'$ which is a segment (non-reduced to a point) strictly contained in both a closed edge $e$ of $C$ and a closed edge $e'$ of $C'$.
Let $\sigma_e$ be the 2-dimensional face of $\Delta_C$ dual to the vertex of $e$ contained in $E$, and let $\sigma_{e'}$ be the 2-dimensional face of $\Delta_{C'}$ dual to the vertex of $e'$ contained in $E$. 
Let $v_i^e, i = 1,2,3$ be the vertices of $\sigma_e$ and let $v_i^{e'}, i = 1,2,3$ be the vertices of $\sigma_{e'}$, so that $v_1^\bullet , v_2^\bullet$ are the two vertices of the edge dual to $\bullet$.
Up to a translation of $\sigma_{e'}$, we can assume that the vertices $v_i^e , v_i^{e'} , i=1,2$ have the same integer coordinates in their corresponding dual subdivision.
We can reformulate \Cref{DefRelTwist} in terms of distributions of signs, in a similar way as for the twist case, as follows.

\begin{proposition}
\label{PropSignRelTwist}
Assume, up to a translation of $\sigma_{e'}$, that the vertices $v_i^e , v_i^{e'} , i=1,2$ have the same integer coordinates in their corresponding dual subdivision.
Assume that $ \delta (v_1^e) \delta (v_2^e) \delta ' ( v_1^{e '}) \delta ' ( v_2^{e '}) = 1$. 
\begin{enumerate}
\item \label{DisCoord} If the coordinates modulo 2 of $v_3^e$ and $v_3^{e'}$ are distinct in the dual subdivisions $\Delta_{C} , \Delta_{C'}$, then the segment $E$ is relatively twisted if and only if $\delta (v_i^e) \delta (v_3^e) \delta ' (v_j^{e'}) \delta ' (v_3^{e'}) = +1$, for both $(i,j) = (1,2) , (2,1)$.
\item \label{ComCoord} If the coordinates modulo 2 of $v_3^e$ and $v_3^{e'}$ coincide in the dual subdivisions $\Delta_{C} , \Delta_{C'}$, then the segment $E$ is relatively twisted if and only if we have $\delta (v_3^e) \delta (v_i^e) \delta ' (v_3^{e'}) \delta ' ( v_i^{e '}) = -1$ for both $i=1,2$.
\end{enumerate} 
\end{proposition}

\begin{proof}
Consider the 2-dimensional simplicial complex $\Delta '$ given by gluing $\sigma_e$ and $\sigma_{e'}$ along the 1-dimensional faces dual to $e,e'$. 
Since $\delta (v_1^e) \delta (v_2^e) \delta ' ( v_1^{e '}) \delta ' ( v_2^{e '}) = 1$ by assumption, we have either $\delta (v_i^e) = \delta ' ( v_i^{e '})$ for both $i=1,2$, or $\delta (v_i^e) \neq \delta ' ( v_i^{e '})$ for both $i=1,2$. 
The segment $E$ is relatively twisted if and only if the edge dual to the interior edge of $\Delta '$ is twisted.
Therefore, we obtain the sign conditions by using \Cref{PropTwistSign} on $\Delta '$ with distribution of signs induced by $\delta \cup \delta ' $ if $\delta (v_i^e) = \delta ' ( v_i^{e '})$ for both $i=1,2$, and distribution of signs induced by $\delta \cup ( - \delta ') $ if $\delta (v_i^e) \neq \delta ' ( v_i^{e '})$ for both $i=1,2$.
\end{proof}

 
\begin{theorem}[\Cref{IntroTh4.3.10}]
\label{Th4.3.10}
Let $(C,\E )$ and $(C' , \E ')$ be two non-singular real tropical curves in $\R^2$ such that there exists a non-transverse connected component $E \subset C\cap C'$ which is a segment, non-reduced to a point, strictly contained in both a closed edge $e$ of $C$ and a closed edge $e'$ of $C'$ (see \Cref{FigDefRel}).
\begin{enumerate}
\item \label{Item4.3.10.1} If $\E_e \neq \E_{e '} '$, or if $\E_e = \E_{e'}'$ and $E$ is relatively non-twisted, then all realisations of $(C,\E )$ and $(C' , \E ')$ intersect in two distinct real points of multiplicity 1 with tropicalisation in $E$.
\item \label{Item4.3.10.2} If $\E_e = \E_{e '} '$ and $E$ is relatively twisted, then a realisation $(C,\E )$ can intersect a realisation of $(C' , \E ')$ in either two distinct real points of multiplicity 1 with tropicalisation in $E$, two distinct complex conjugated points of multiplicity 1 with tropicalisation in $E$, or a multiplicity 2 real point with tropicalisation in $E$.
Moreover, there exist infinitely many realisations satisfying the first two intersection types, and there exist exactly two pairs of realisations satisfying the third intersection type.
\end{enumerate}
%
\end{theorem}

Note that twisted edges and relatively twisted intersection components have opposite roles in \Cref{Th4.3.9} and \Cref{Th4.3.10}.
The names in \Cref{DefRelTwist} have been chosen in order to have topological properties similar to those for twisted edges listed in \cite[Section 3.2]{brugalle2015brief}.

\begin{proof}[Proof of \Cref{Th4.3.10}]\renewcommand{\qedsymbol}{}
After an affine linear transformation, we can assume without loss of generality that the edges $e , e'$ have primitive integer direction $(1,1)$ and $E$ is the segment with endpoints $v_1 = (0,0)$ and $v_2 = (b,b)$, such that $b>0$. 
We can moreover assume that the two edges of $C$ adjacent to $e$ to have primitive integer direction $(0,1)$ and $(1,0)$, respectively (see for instance the black tropical curve in \Cref{FigRelTwist}).

Then the local equations of some realisations in $(\K^\times )^2$ of $(C,\E)$ and $(C' ,\E ')$ are of the form
\begin{align}
\label{EqRelTwist1} z^{k} w^l \left( \varepsilon_{k+1,l}z + \varepsilon_{k,l+1} w + \varepsilon_{k,l} \right) & = 0 \\
\label{EqRelTwist2} \delta_{m,n} t^{b} z^{m} w^{n} + \delta_{p+1,q} z^{p+1} w^q + \delta_{p,q+1} z^p w^{q+1}  & = 0
\end{align} 
with $m+n = p+q+2$ and such that the coefficients $\varepsilon_{ij} , \delta_{ij} \in \mathbb{K}_\R^\times$ have valuation $0$.
In $(\K^\times )^2$, \Cref{EqRelTwist1} is equivalent to $w = \frac{-1}{\varepsilon_{k,l+1}}(\varepsilon_{k+1,l} z + \varepsilon_{k,l})$.
Then after substitution in \Cref{EqRelTwist2}, we obtain the following equation
\begin{align*}
\label{EqPolRelTwist} & \delta_{m,n} t^{b} z^m \left( \frac{-\varepsilon_{k+1,l} z - \varepsilon_{k,l}}{\varepsilon_{k,l+1}} \right)^{n} \\
+ & \delta_{p+1,q} z^{p+1} \left( \frac{-\varepsilon_{k+1,l} z - \varepsilon_{k,l}}{\varepsilon_{k,l+1}} \right)^{q} \\
+ & \delta_{p,q+1} z^{p} \left( \frac{-\varepsilon_{k+1,l} z - \varepsilon_{k,l}}{\varepsilon_{k,l+1}} \right)^{q+1} = 0.
\end{align*}
%

If $p-m>0$, or $q-n >0$, the local tropical model contains a transverse intersection point outside $E$ of tropical intersection multiplicity $p-m$, or $q-n$.
In those cases, by \Cref{CorTransMult2}, the equation has $p-m$ (or $q-n$) distinct solutions corresponding to points with tropicalisation the point $(- \frac{b}{p-m},0) \not\in E$ (or $(0,- \frac{b}{q-n}) \not\in E$).
In all cases, the two remaining solutions $z_1 ' , z_2 ' \in \K^\times$ are those corresponding to the points with tropicalisation in $E$.
By \cite[Corollary 4.7]{brugalle2012inflection}, the two solutions are of the form $z_1 ' = z_1 t^{-a}$ and $z_2 ' = z_2 t^{-b+a}$, with $\val (z_1) = \val (z_2) = 0$ and $a \in [0,b]$.
In particular, the equation 
\begin{equation}
\label{EqAi}
A_1 + A_2 + A_3 = 0
\end{equation} 
is satisfied, for 
 \[ A_1 := \delta_{m,n} z_1^m \left( \frac{-\varepsilon_{k+1,l} z_1 t^{-a} - \varepsilon_{k,l}}{\varepsilon_{k,l+1}} \right)^{n} t^{b-ma}, \]
 \[ A_2 := \left( \delta_{p+1,q} - \frac{\delta_{p,q+1} \varepsilon_{k+1,l}}{\varepsilon_{k,l+1}} \right) \left( \frac{-\varepsilon_{k+1,l} z_1 t^{-a} - \varepsilon_{k,l}}{\varepsilon_{k,l+1}} \right)^{q} z_1^{p+1} t^{-a(p+1)}, \] and \[ A_3 := \frac{\delta_{p,q+1} \varepsilon_{k,l}}{\varepsilon_{k,l+1}} \left( \frac{-\varepsilon_{k+1,l} z_1 t^{-a} - \varepsilon_{k,l}}{\varepsilon_{k,l+1}} \right)^{q} z_1^p t^{-ap}.  \]
In order for \Cref{EqAi} to be satisfied, the minimum must be reached at least twice in the set $\{ \val (A_i) ~ ; ~ i=1,2,3 \}$.
%
%
We will now prove the statements of the theorem.
\begin{enumerate}[label=(\Roman*)]
\item \label{ProofDisRel} Assume first that $\E_e \neq \E_{e'}'$ (first case of \Cref{Item4.3.10.1}).
Then the product $\delta_{p+1,q} \delta_{p,q+1}\varepsilon_{k+1,l}\varepsilon_{k,l+1}$ is negative.
In particular, we have \[ \val \left( \delta_{p+1,q} - \frac{\delta_{p,q+1} \varepsilon_{k+1,l}}{\varepsilon_{k,l+1}} \right) = 0 \] as there can be no cancellation of the leading coefficients.
Then \[ \quad \quad \quad \min \{ \val (A_i) ~ ; ~ i=1,2,3 \} = \min \{ b-(m+n)a , -(p+q+1)a , -(p+q)a \} .  \]
Since $m+n = p+q+2$, the minimum is reached at least twice if and only if either $a=0$ or $a=b$.
Then $ \val (z_1 ') \neq \val (z_2 ')$ and $\val (z_i ') \in \{ 0,b \}$ for $i=1,2$, hence $z_1 '$ and $z_2 '$ are two distinct real solutions of \Cref{EqAi}.
Therefore, all realisations of $(C,\E)$ and $(C',\E ')$ intersect in two distinct real points $p_1 , p_2$ with tropicalisation the vertices $v_1 = (0,0)$ and $v_2 = (b,b)$ of $E$.

\item \label{ProofRelNonTwist} Assume now that $\E_e = \E_{e'}'$ and the segment $E$ is relatively non-twisted (second case in \Cref{Item4.3.10.1}). 
The real phase structure assumptions implies that the product $\delta_{p+1,q} \delta_{p,q+1}\varepsilon_{k+1,l}\varepsilon_{k,l+1}$ is positive, hence there can be cancellations of the leading coefficients in $\left( \delta_{p+1,q} - \frac{\delta_{p,q+1} \varepsilon_{k+1,l}}{\varepsilon_{k,l+1}} \right)$.
Let $c:= \val \left( \delta_{p+1,q} - \frac{\delta_{p,q+1} \varepsilon_{k+1,l}}{\varepsilon_{k,l+1}} \right)$.
The minimum of $\{ \val (A_i) ~ ; ~ i=1,2,3 \}$ is reached at least twice if and only if either $a=c$ and $c\leq \frac{b}{2}$, or $a=b-c$ and $c\leq \frac{b}{2}$, or $a= \frac{b}{2}$ and $c\geq \frac{b}{2}$.
If $c < \frac{b}{2}$, we have $\val (z_1 ') \neq \val (z_2 ')$, hence $z_1 '$ and $z_2 '$ are two distinct real solutions of \Cref{EqAi}.
If $c \geq \frac{b}{2}$, then the following equation is satisfied,
\begin{equation*}
\label{EqVal}
\quad \quad \quad \delta_{m,n} \left( -\frac{\varepsilon_{k+1,l}}{\varepsilon_{k,l+1}} \right)^{n-q} z_1^2 + \gamma z_1 - \frac{\delta_{p,q+1} \varepsilon_{k,l}}{\varepsilon_{k,l+1}} = 0 ,
\end{equation*}
where $\gamma = 0$ if $c> \frac{b}{2}$ and $\gamma \neq 0 ,\val (\gamma) = 0$ is induced by the non-zero terms of $\left( \delta_{p+1,q} - \frac{\delta_{p,q+1} \varepsilon_{k+1,l}}{\varepsilon_{k,l+1}} \right)$ if $c = \frac{b}{2}$.
Let 
\begin{equation}
\label{EqDiscrimRel}
D:= \gamma^2 + 4 \frac{\delta_{m,n} \delta_{p,q+1} \varepsilon_{k,l}}{\varepsilon_{k,l+1}} \left( -\frac{\varepsilon_{k+1,l}}{\varepsilon_{k,l+1}} \right)^{n-q}
\end{equation}
be the discriminant associated to the degree 2 equation above.
Since $E$ is relatively non-twisted and $n-q = 0 \mod 2$ if and only if $(m,n) = (p,q) \mod 2$, by \Cref{PropSignRelTwist}, the discriminant $D$ is positive, hence $z_1 '$ and $z_2 '$ are two distinct real solutions of \Cref{EqAi}.
Therefore, all realisations of $(C,\E)$ and $(C',\E ')$ intersect in two distinct real points $p_1 , p_2$ with tropicalisation in $E$.

\item \label{ProofRelTwist} Assume finally that $\E_e = \E_{e'}$ and $E$ is relatively twisted.
From \Cref{ProofRelNonTwist}, we can recover the facts that if $c:= \val \left( \delta_{p+1,q} - \frac{\delta_{p,q+1} \varepsilon_{k+1,l}}{\varepsilon_{k,l+1}} \right) < \frac{b}{2}$, the two solutions $z_1 ' , z_2 '$ of \Cref{EqAi} are distinct and real, and if $c\geq \frac{b}{2}$, it depends on the sign of the discriminant $D$ from \Cref{EqDiscrimRel}. 
In the first case, we obtain that all realisations of $(C,\E )$ and $(C', \E ')$ intersect in two distinct real points with tropicalisation in $E$. 
Moreover, the condition $c< \frac{b}{2}$ is open, hence there exist infinitely many such realisations.
In the second case, since $E$ is relatively twisted and $n-q = 0 \mod 2$ if and only if $(m,n) = (p,q) \mod 2$, we can use \Cref{PropSignRelTwist} to determine the sign of the discriminant $D$.
We have that $D$ is positive if \[\gamma^2 > \left| 4 \frac{\delta_{m,n} \delta_{p,q+1} \varepsilon_{k,l}}{\varepsilon_{k,l+1}} \left( -\frac{\varepsilon_{k+1,l}}{\varepsilon_{k,l+1}} \right)^{n-q} \right| , \]
and then all realisations of $(C,\E )$ and $(C', \E ')$ intersect in two distinct real points with tropicalisation in $E$.
We have that $D$ negative if \[\gamma^2 < \left| 4 \frac{\delta_{m,n} \delta_{p,q+1} \varepsilon_{k,l}}{\varepsilon_{k,l+1}} \left( -\frac{\varepsilon_{k+1,l}}{\varepsilon_{k,l+1}} \right)^{n-q} \right| , \]
and then all realisations of $(C,\E )$ and $(C', \E ')$ intersect in two distinct complex conjugated points with tropicalisation in $E$.
Moreover, this case is satisfied when $c> \frac{b}{2}$, which is an open condition, hence there exist infinitely many such realisations.
We have that $D$ is zero if \[\gamma^2 = \left| 4 \frac{\delta_{m,n} \delta_{p,q+1} \varepsilon_{k,l}}{\varepsilon_{k,l+1}} \left( -\frac{\varepsilon_{k+1,l}}{\varepsilon_{k,l+1}} \right)^{n-q} \right| , \]
and then all realisations of $(C,\E )$ $(C', \E ')$ intersect in a multiplicity 2 real point with tropicalisation in $E$.
Moreover, the condition on $\gamma$ implies that $\c = \frac{b}{2}$, which is a closed condition, hence the solution set is of dimension 0, and the discriminant $D$ is a degree 2 polynomial with coefficients in $\K_\R$, hence there exist exactly two pairs of realisations of $(C,\E)$ and $(C',\E ')$ intersecting in a multiplicity 2 real point with tropicalisation in $e$.
\quad \quad \quad $\blacksquare$
\end{enumerate}
\end{proof}

We obtain from \Cref{ProofDisRel} of the proof of \Cref{Th4.3.10} the following corollary.

\begin{corollary}
\label{CorRelDistRealPhase}
Let $(C,\E )$ and $(C' , \E ')$ be two non-singular real tropical curves in $\R^2$ such that there exists a non-transverse connected component $E \subset C\cap C'$ which is a segment (non-reduced to a point) strictly contained in both a closed edge $e$ of $C$ and a closed edge $e'$ of $C'$.
Let $v_1$ and $v_2$ be the vertices of $E$.
Let $\mathcal{C}$ and $\mathcal{C'}$ be realisations of $(C,\E )$ and $(C' , \E ')$.
If $\E_e \neq \E_{e'}'$, then for $p_1 , p_2 \in (\K_\R^\times )^2$ the two distinct real points of multiplicity 1 in $\mathcal{C} \cap \mathcal{C'}$ with tropicalisation in $E$, we have (up to renumbering) $\Trop (p_1) = v_1$ and $\Trop (p_2) = v_2$. 
\end{corollary}

The following corollary follows from \Cref{ProofRelTwist} of the proof of \Cref{Th4.3.10} and will be useful for proving conditions for hyperbolicity.

\begin{corollary}
\label{CorNonRelTwist}
Let $(C,\E )$ and $(C' , \E ')$ be two non-singular real tropical curves in $\R^2$ such that there exists a non-transverse connected component $E \subset C\cap C'$ which is a segment (non-reduced to a point) strictly contained in both a closed edge $e$ of $C$ and a closed edge $e'$ of $C'$.
If $E$ is relatively twisted and $\E_e = \E_{e'}'$, there exist distinct realisations $\mathcal{C}_+ , \mathcal{C}_- \subset (\K^\times )^2$ of $(C,\E )$ and distinct realisations $\mathcal{C}_+ ' , \mathcal{C}_- ' \subset (\K^\times )^2$ of $(C' , \E ')$ such that the intersection $\mathcal{C}_+ \cap \mathcal{C}_+ '$ contains two distinct real points of multiplicity 1 with tropicalisation in $E$, and the intersection $\mathcal{C}_{-} \cap \mathcal{C}_{-} '$ contains two distinct complex conjugated points of multiplicity 1 with tropicalisation in $E$.
\end{corollary}

We obtain an other corollary in the case of multiplicity 2 real intersection points following from \Cref{ProofDisRel}, \Cref{ProofRelNonTwist} and \Cref{ProofRelTwist} of the proof of \Cref{Th4.3.10}.

\begin{corollary}
\label{CorNonRelTwistTang}
Let $(C,\E )$ and $(C' , \E ')$ be two non-singular real tropical curves in $\R^2$ such that there exists a non-transverse connected component $E \subset C\cap C'$ which is a segment (non-reduced to a point) strictly contained in both a closed edge $e$ of $C$ and a closed edge $e'$ of $C'$.
If $\mathcal{C}_0$ and $\mathcal{C}_0 '$ are realisations of $(C,\E )$ and $(C' , \E ')$ intersecting in a multiplicity 2 real point with tropicalisation in $E$, then $e$ is relatively twisted and $\E_e = \E_{e '} '$.
\end{corollary}

\section{Hyperbolic real tropical curves}
\label{SecHyperbolic}

\subsection{Topological criterion for hyperbolicity}
%
%
%
We start our study of hyperbolic curves near the non-singular tropical limit by translating an equivalent topological condition for hyperbolicity into combinatorial conditions on tropical curves.
For $\mathcal{C}$ a non-singular real algebraic curve in $\pr^2$, recall that an \emph{oval} is a connected component of $\mathcal{C}(\R )$ homeomorphic to $S^1$ disconnecting $\pr^2 (\R)$, and a \emph{pseudo-line} is a connected component of $\mathcal{C}(\R )$ homeomorphic to $S^1$ and which does not disconnect $\pr^2 (\R)$.
The following proposition is given in \cite[Proposition 1.1]{orevkov2007arrangements}, and can be recovered easily from Rokhlin formulas \cite{rokhlin1978complex}. 

\begin{proposition}[\cite{orevkov2007arrangements}]
\label{PropHypRok}
Let $\mathcal{C}$ be a non-singular dividing real algebraic curve of degree $2k$ or $2k+1$ in $\pr^2$.
Then $\mathcal{C}(\R )$ has $l\geq k$ ovals and if $l=k$, then $\mathcal{C}$ is hyperbolic.
\end{proposition}

Using in addition either \cite[Theorem 5.2]{helton2007linear} or \cite[Statement 3.6]{rokhlin1978complex}, we obtain the following corollary
\begin{corollary}
\label{CorHypRok}
Let $\mathcal{C}$ be a non-singular real algebraic curve of degree $d$ in $\pr^2$.
Then $\mathcal{C}$ is hyperbolic if and only if $\mathcal{C}$ is a dividing curve and $\mathcal{C}(\R )$ consists of exactly $\left\lfloor \frac{d}{2} \right\rfloor$ ovals (and a pseudo-line if $d$ is odd).
\end{corollary}

The two conditions from \Cref{PropHypRok} have a combinatorial counterpart whenever the curve considered is near the non-singular tropical limit.
We start by recalling Haas' dividing criterion \cite{haas1997real}.

\begin{definition}
Let $C \subset \T \pr^2$ be a non-singular tropical curve. 
A cycle $\gamma$ of $C$ (seen as a graph) is called \emph{primitive} if it bounds a connected component of $\T \pr^2 \backslash C$.
\end{definition}

In the following, for $\gamma$ a cycle on a non-singular tropical curve $C$ (seen as a graph), we identify $\gamma$ with its underlying subset of bounded edges of $C$ (in particular, the cardinality $| \gamma |$ is the number of edges in $\gamma$).

\begin{theorem}[\cite{haas1997real}]
\label{ThDividing}
Let $\mathcal{C}:= (\mathcal{C}_t)_t$ be a real algebraic curve in $\pr_\K^2$ with tropicalisation the non-singular tropical curve $C \subset \T \pr^2$, and let $T$ be the admissible set of twisted edges of $C$ induced by $\mathcal{C}$.
Then for $t>0$ small enough, the curve $\mathcal{C}_t \subset \pr^2$ is dividing if and only if for any primitive cycle $\gamma$ in $C$, we have 
\begin{equation}
\label{EqDiv}
|\gamma  \cap T | = 0 \mod 2. 
\end{equation}
\end{theorem}

\begin{definition}
We say that an admissible set of twisted edges $T$ on a non-singular tropical curve $C$ is \emph{dividing} if $T$ satisfies \Cref{EqDiv} for every primitive cycle in $C$. 
\end{definition}
%
%
Renaudineau and Shaw \cite[Theorem 7.2]{renaudineau2018bounding} gave a combinatorial way to compute the number of connected components of a real curve near the non-singular tropical limit, which can be reformulated as follows.

\begin{theorem}[\cite{renaudineau2018bounding}]
\label{ThConnectHom}
Let $(C,\E )$ be a non-singular real tropical curve of degree $d$ in $\T \pr^2$ with induced admissible set of twisted edges $T$.
For $\gamma_i , i= 1,\ldots ,g$ the primitive cycles on $C$, let $A_T$ be a symmetric $g\times g$-matrix with coefficients in $\Z_2$ defined as
\[ A_T := \left( |\gamma_i \cap \gamma_j \cap T | \mod 2 \right)_{i,j = 1,\ldots ,g} .   \] 
Then the number of connected components of the real part $\R C_\E$ is $1 + \dim \ker A_T$. 
\end{theorem}

Applying \Cref{CorHypRok}, \Cref{ThDividing} and \Cref{ThConnectHom}, we obtain the following proposition.

\begin{proposition}[\Cref{IntroPropHypTwist}]
\label{PropHypTwist}
Let $\mathcal{C}:= (\mathcal{C}_t)_t$ be a non-singular real algebraic curve in $\pr_\K^2$ with Newton polygon $\Delta_d$, let $C \subset \T \pr^2$ be the non-singular tropical curve of degree $d$ given as $C := \Trop (\mathcal{C})$, and let $T$ be the admissible set of twisted edges on $C$ induced by $\mathcal{C}$.
Then for $t>0$ small enough, the curve $\mathcal{C}_t$ is hyperbolic if and only if $T$ is dividing and $\dim \ker A_T = \left\lceil \frac{d}{2} \right\rceil - 1$.
\end{proposition}

\begin{definition}
\label{DefTropHyperbo}
We say that a non-singular real tropical curve $(C,\E )$ of degree $d$ in $\T \pr^2$ is \emph{hyperbolic} if the admissible set of twisted edges $T$ on $C$ induced by $\E$ satisfies the conditions of \Cref{PropHypTwist}.
In that case, the pair $(\R \pr^2 , \R C_\E)$ is homeomorphic to a pair $(\pr^2 (\R), \mathcal{C} (\R ))$, for $\mathcal{C}$ a non-singular hyperbolic curve of degree $d$ in $\pr^2$.
\end{definition}

\begin{figure}
\centering
\begin{subfigure}[t]{0.45\textwidth}
	\centering
	\includegraphics[width=\textwidth]{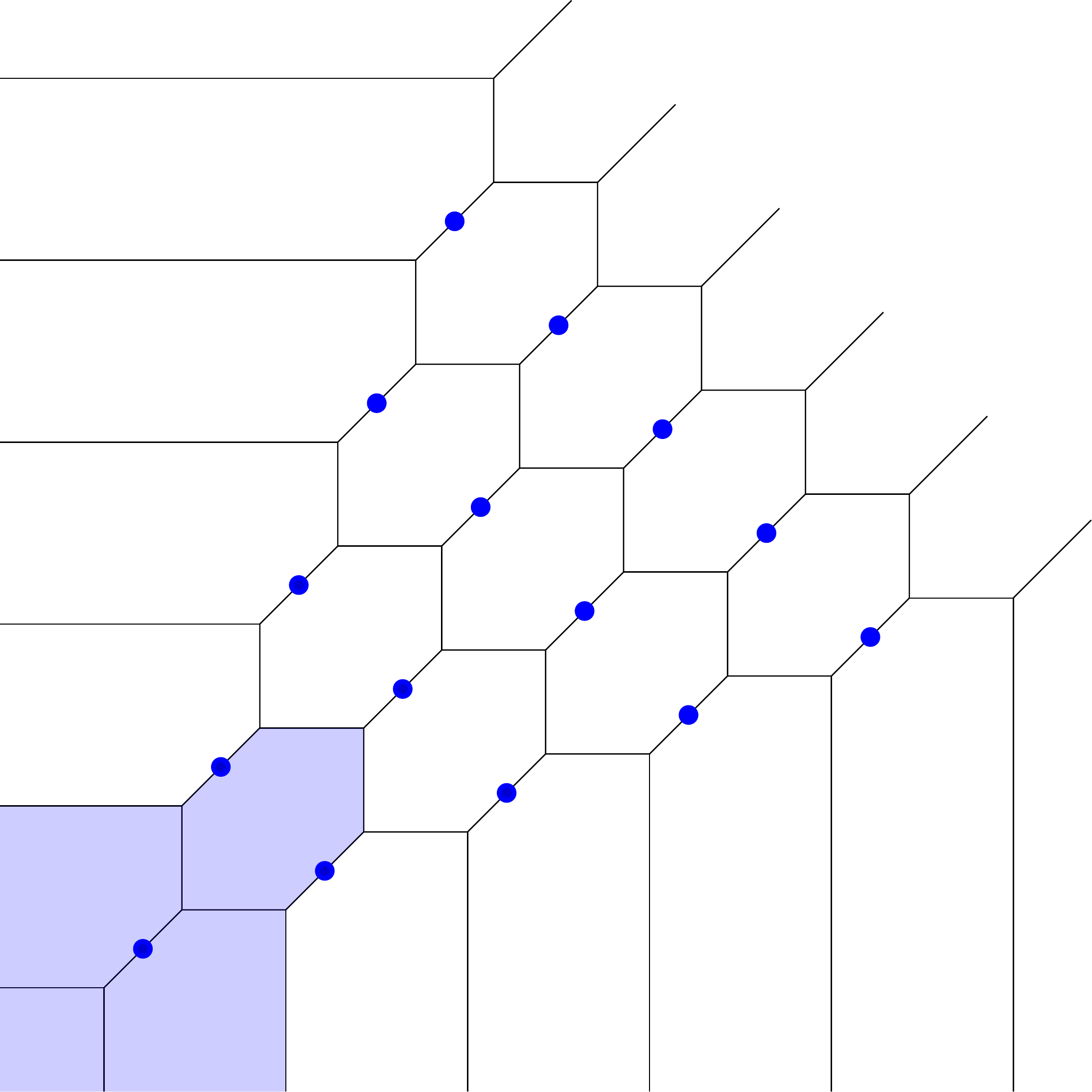}
	\caption{A tropical curve with hyperbolic set of twisted edges and its tropical hyperbolicity locus.}
	\label{FigHyperboHoney}
\end{subfigure}\hfill
\begin{subfigure}[t]{0.45\textwidth}
	\centering
	\includegraphics[width=\textwidth]{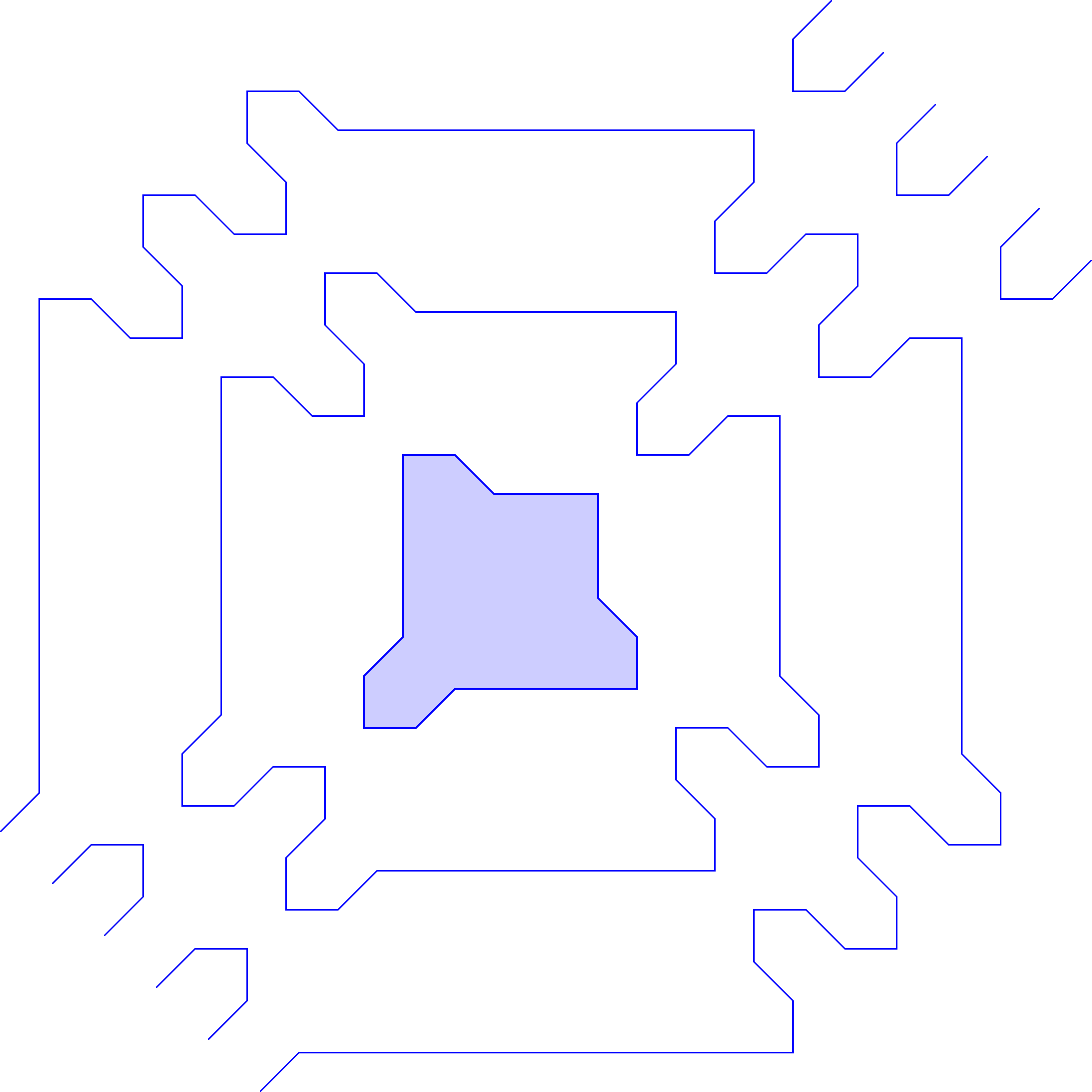}
	\caption{The real part of the tropical curve and its signed tropical hyperbolicity locus.}
	\label{FigHypLocus}
\end{subfigure}
\caption{\Cref{ExHyperbo1}.}
\label{Fighyperbo}
\end{figure}

In the following example and beyond, for $(C,\E )$ a non-singular real tropical curve of degree $d$ in $\T \pr^2$, an \emph{oval} of $(C,\E )$ is a connected component of the real part $\R C_\E$ disconnecting $\R \pr^2 = (\T \pr^2)^* /\sim$, and a \emph{pseudo-line} of $(C,\E)$ is a connected component of $\R C_\E$ which does not disconnect $\R \pr^2$.

\begin{example}
\label{ExHyperbo1}
Let $C$ be the non-singular tropical curve of degree $d=6$ in $\T \pr^2$ pictured in \Cref{Fighyperbo}.
Let $T$ be the (admissible) dividing set of twisted edges on $C$ marked by blue points in \Cref{FigHyperboHoney}.
The tropical curve $C$ has 10 primitive cycles $\{ \gamma_1 , \gamma_{10} \}$, ordered by the ``diagonals" of consecutive primitive cycles, starting from the bottom diagonal.
The corresponding size 10 matrix $A_T$ defined in \Cref{ThConnectHom} is a bloc diagonal matrix with each non-zero bloc associated to a ``diagonal" of consecutive primitive cycles meeting at a twisted edge.
The rank of $A_T$ is given as the sum of the rank of the non-zero blocs $A_i, i=1,\ldots ,4$.
We can compute that $\rank A_1 = 0, \rank A_2 = \rank A_3 = 2$ and $\rank A_4 = 4$.
Let $\E$ be a real phase structure on $C$ inducing the set of twisted edges $T$.
Then by \Cref{ThConnectHom}, the number of connected components of the real part $\R C_\E$ is
\begin{align*}
b_0 (\R C_\E ) & = 1 + \dim \ker A_T \\
& = 1+ 10 - \sum_i \rank A_i = 3 = \left\lceil \frac{d}{2} \right\rceil .
\end{align*}  
Then by \Cref{PropHypTwist}, the real tropical curve $(C,\E )$ is hyperbolic.
The real part $\R C_\E$ of the real tropical curve $(C,\E )$, consisting of 3 nested ovals, is pictured in \Cref{FigHypLocus}.
\end{example}

\subsection{The tropical hyperbolicity locus, spectrahedra and convexity}

\begin{definition}
\label{DefTropHypLocus}
Let $(C,\E )$ be a non-singular hyperbolic real tropical curve of degree $d$ in $\T \pr^2$.
The \emph{signed tropical hyperbolicity locus} $\R H \subset \R \pr^2$ of $(C,\E)$ is the interior of the innermost oval $\rho$ of $\R C_\E$ if $d\geq 2$, and is given as $\R \pr^2 \backslash \R C_\E$ if $d=1$.
The \emph{tropical hyperbolicity locus} $H \subset (\T \pr^2 \backslash C)$ of $(C,\E)$ is the set of connected components $D$ of $\T \pr^2 \backslash C$ such that there exists a symmetric copy $\varepsilon (D)$ in $\R H$.
\end{definition}

\begin{example}
Let $(C,\E )$ be the non-singular real tropical curve in \Cref{Fighyperbo}, with underlying tropical curve and set of twisted edges represented in \Cref{FigHyperboHoney} and with real part represented in \Cref{FigHypLocus}. 
The tropical hyperbolicity locus $H$ of $(C,\E)$ is represented in \Cref{FigHyperboHoney} in light blue, and the signed tropical hyperbolicity locus $\R H$ is pictured in \Cref{FigHypLocus} in light blue.
The connected component of $\T \pr^2 \backslash C$ bounded by a primitive cycle and belonging to $H$ has exactly one symmetric copy in $\R H$.
The two connected components of $\T \pr^2 \backslash C$ meeting exactly one 1-dimensional face of the stratification of $\T \pr^2$ have exactly two symmetric copies in $\R H$.
Finally, the connected component of $\T \pr^2 \backslash C$ containing a 0-dimensional face of the stratification of $\T \pr^2$ has all four of its symmetric copies lying in $\R H$.
\end{example}

\begin{proposition}
\label{PropHomeoHypLocus}
Let $(C,\E )$ be a non-singular hyperbolic real tropical curve of degree $d$ in $\T \pr^2$, and let $\mathcal{C} := (\mathcal{C}_t)_t$ be a realisation in $\pr_\K^2$ of $(C,\E )$.
The homeomorphism from \Cref{ThViro} sends the hyperbolicity locus of $\mathcal{C}_t$ to the signed tropical hyperbolicity locus of $(C,\E)$, for $t>0$ small enough.
\end{proposition}

\begin{proof}
By \cite[Section 5]{helton2007linear}, the hyperbolicity locus of the curve $\mathcal{C}_t$ is given as the interior of the innermost oval $\rho_t$ of $\mathcal{C}_t (\R )$ if $d>1$, and as $\pr^2 (\R ) \backslash \mathcal{C}_t (\R )$ if $d=1$.
We obtain the result by seeing that the homeomorphism of pairs from \Cref{ThViro} sends the interior of $\rho_t$ to the interior of the innermost oval of $\R C_\E$ if $d>1$, and it sends $\pr^2 (\R ) \backslash \mathcal{C}_t (\R )$ to $\R \pr^2 \backslash \R C_\E$ if $d=1$. 
\end{proof}

\begin{remark}
\label{RkTropConv}
Recall that the hyperbolicity locus $\mathcal{H}$ of a (non-singular) hyperbolic curve in $\pr^2$ is a convex set \cite[Property 5.3.(3)]{helton2007linear}.
We would like to see if the signed tropical hyperbolicity locus $\R H$ of $(C,\E )$ satisfies some tropical counterpart of the notion of convexity.
Given any two points $x,y \in \R H \subset \R \pr^2$, there exist a (not necessarily unique) real tropical line $(L,\E ')$ with real part $\R L_{\E '}$ going through $x$ and $y$.
By \Cref{PropHomeoHypLocus}, for every real tropical line $(L,\E ')$ with real part going through $x$ and $y$, the intersection $\R L_{\E '} \cap \R H$ is a connected 1-dimensional chain.
If $x \neq y$, the union 
\[ S_{x,y}:= \bigcup\limits_{ (L,\E')} \R L_{\E '} \cap \R H , \]
for $(L,\E ')$ running through the real tropical lines with real part containing $x$ and $y$, is a \emph{signed tropical line segment}, in the sense of \cite[Example 3.11]{loho2019signed}.
Then for every couple of points $(x,y) \in (\R H)^2$, the signed tropical line segment $S_{x,y}$ defined as above is contained in $\R H$.
Hence the set $\R H$ is \emph{signed tropically convex}, in the sense of \cite{loho2019signed}.
However, the closure $\overline{\R H}$ is not necessarily signed tropically convex. 
Indeed, if there exist two points $x,y \in \partial (\R H)$ in non-general position (ie.~ there exist several real tropical lines through $x$ and $y$), then for $(L,\E ')$ a real tropical line with real part going through $x$ and $y$, the intersection $\R L_{\E '} \cap \overline{\R H}$ is not necessarily connected. 
For instance, in \Cref{Fighyperbo}, take two points $x',y' \in \T \pr^2$ on the twisted edge lying in the interior of $\overline{H}$, and consider the symmetric copy $x$ of $x'$ in the top-left orthant and the symmetric copy $y$ of $y'$ in the bottom-right orthant. 
The two points $x$ and $y$ lie on the boundary of $\overline{\R H}$, and we have the intersection $\R L_{\E '} \cap \overline{\R H}$ disconnected for some $\R L_{\E '}$ going through $x, y$ and the positive orthant of $\R \pr^2$.
Loho and Skomra, in a work in preparation \cite{loho2021signed}, define a relaxed version of signed tropical convexity, where the signed tropical line segments $S_{x,y}$ are given as intersection over all real tropical lines through $x,y$ instead of the union over all real tropical lines through $x,y$.
In that case, the closure $\overline{\R H}$ should satisfy the conditions in order to be relaxed signed tropically convex.
\end{remark}

We want to define the notion of non-singular real tropical curve hyperbolic with respect to a \emph{real tropical point}.
A real tropical point in $\T \pr^2$ is a couple $(v,\varepsilon ) \in \T \pr^2 \times \mathcal{R}^2$ with real part the point $\varepsilon (v) \in \R \pr^2$.
A \emph{realisation} of a real tropical point $(v,\varepsilon )$ is a point $p \in \pr^2_\K (\R )$ with tropicalisation $v$ lying in (the closure of) the orthant $\varepsilon$ of $\pr^2_\K (\R )$.

\begin{definition}
\label{DefTropHyperboPoint}
Let $(C,\E)$ be a non-singular real tropical curve of degree $d$ in $\T \pr^2$.
We say that $(C,\E )$ is hyperbolic with respect to the real tropical point $(v,\varepsilon)$, with $v\in \T \pr^2  \backslash C$ and $\varepsilon \in \mathcal{R}^2$, if for every realisation $\mathcal{C} \subset \pr^2_\K$ of $(C,\E )$, the curve $\mathcal{C}$ is hyperbolic with respect to all realisations of $(v,\varepsilon)$.
\end{definition}

\begin{proposition}
\label{PropEquivHyp}
A non-singular real tropical curve $(C,\E)$ of degree $d$ in $\T \pr^2$ is hyperbolic if and only if there exists a real tropical point $(v,\varepsilon)$ such that $(C,\E)$ is hyperbolic with respect to $(v,\varepsilon)$.
\end{proposition}

\begin{proof}
Recall that a non-singular real algebraic curve $\mathcal{C}$ of degree $d$ in $\pr_\K^2$ is hyperbolic with respect to a point $p \in \pr_\K^2 (\R ) \backslash \mathcal{C} (\R )$ if every real line $\mathcal{L}$ through $p$ in $\pr_\K^2$ intersects $\mathcal{C}$ in $d$ distinct real points $p_1 , \ldots , p_d$. 
Then by \Cref{PropHomeoHypLocus}, for every realisation $\mathcal{C}\subset \pr^2_\K$ of a non-singular hyperbolic tropical curve $(C,\E)$ of degree $d$ in $\pr^2$, the curve $\mathcal{C}$ is hyperbolic with respect to every realisation $p$ of each real tropical point with real part in the signed tropical hyperbolicity locus $\R H$ of $(C,\E)$.
\end{proof}

\begin{corollary}
\label{CorTropicaliseHypLoc}
For $\mathcal{C}$ a realisation of a non-singular hyperbolic tropical curve $(C,\E)$ of degree $d$ in $\T \pr^2$, we have an inclusion $H \subset \Trop (\mathcal{H})$ of the tropical hyperbolicity locus $H$ inside the tropicalisation of the hyperbolicity locus $\mathcal{H}$ of $\mathcal{C}$, and we have equality of closures $\overline{H} = \overline{\Trop (\mathcal{H})} = \Trop (\overline{\mathcal{H}})$.
\end{corollary}

The first inclusion in \Cref{CorTropicaliseHypLoc} may be strict, for instance in the case where $(C,\E)$ is a real tropical line in $\T \pr^2$.

\begin{remark}
\label{RkTropSpec}
Let $\mathcal{H} \subset \pr^2 (\R)$ (or $\mathcal{H} \subset \pr_\K^2 (\R)$) be the hyperbolicity locus of a hyperbolic curve $\mathcal{C}$ in $\pr^2$ (or $\mathcal{C}$ in $\pr_\K^2$).
Recall that the set $\mathcal{H}$ is a \emph{spectrahedron}, as $\mathcal{C}$ admits a symmetric determinantal representation $D$ such that $D(x)$ is positive semi-definite for every point $x\in \mathcal{H}$ \cite[Theorem 2.2]{helton2007linear}.
Then by \Cref{PropEquivHyp}, for $(C,\E)$ a hyperbolic non-singular real tropical curve with tropical hyperbolicity locus $\R H$, the closure $\overline{H}$ of the tropical hyperbolicity locus is a \emph{tropical spectrahedron}, ie.~ the tropicalisation of a spectrahedron, as defined in \cite[Definition 5.1]{allamigeon2020tropical}.
\end{remark}

\subsection{Intersection criterion for the hyperbolicity locus}
\label{SecHypIntersection}

We want to give a criterion, based on the intersection tools of \Cref{SectionIntersection}, in order to determine if a non-singular real tropical curve is hyperbolic with respect to a real tropical point.  
We define a \emph{polyhedral subdivision} of $\T \pr^2$ with respect to a point $v$, which will allow to parameterise the pencil of tropical lines going through $v$.

\begin{definition}
\label{DefFanPt}
Let $v$ be a point in $\T \pr^2 \backslash \partial \T \pr^2$.
We construct a \emph{polyhedral subdivision} $\Sigma_v$ of $\T \pr^2$ \emph{with respect to} $v$ (see \Cref{SubdivisionPoint}) by subdividing $\T \pr^2$ into three 2-dimensional faces $\sigma_\eta$, with $\eta \in \{ (1,0) , (0,1) , (1,1) \} \subset \Z^2$, such that
\begin{itemize}
\item the faces $\tau_{1,1} , \tau_{1,0} , \tau_{0,1}$ are half-lines of origin $v$ and outward primitive integer direction $(-1,-1),(1,0),(0,1)$, respectively;
\item two faces $\sigma_\eta , \sigma_\zeta$ intersect in a 1-dimensional face $\tau_\xi$, for \[ \{ \eta , \zeta , \xi \} = \{ (1,0) , (0,1) , (1,1) \} . \]
\end{itemize} 
\end{definition}

To each point $u$ lying on a 1-dimensional face $\tau_\eta$ of $\Sigma_v$, we can associate a tropical line $L_u$ with 3-valent vertex $u$.
Then the pencil of tropical lines through the point $v$ is given as the set of all tropical lines $L_u$ defined this way. 
\begin{figure}
\centering
\begin{subfigure}[t]{0.45\textwidth}
	\includegraphics[width=\textwidth]{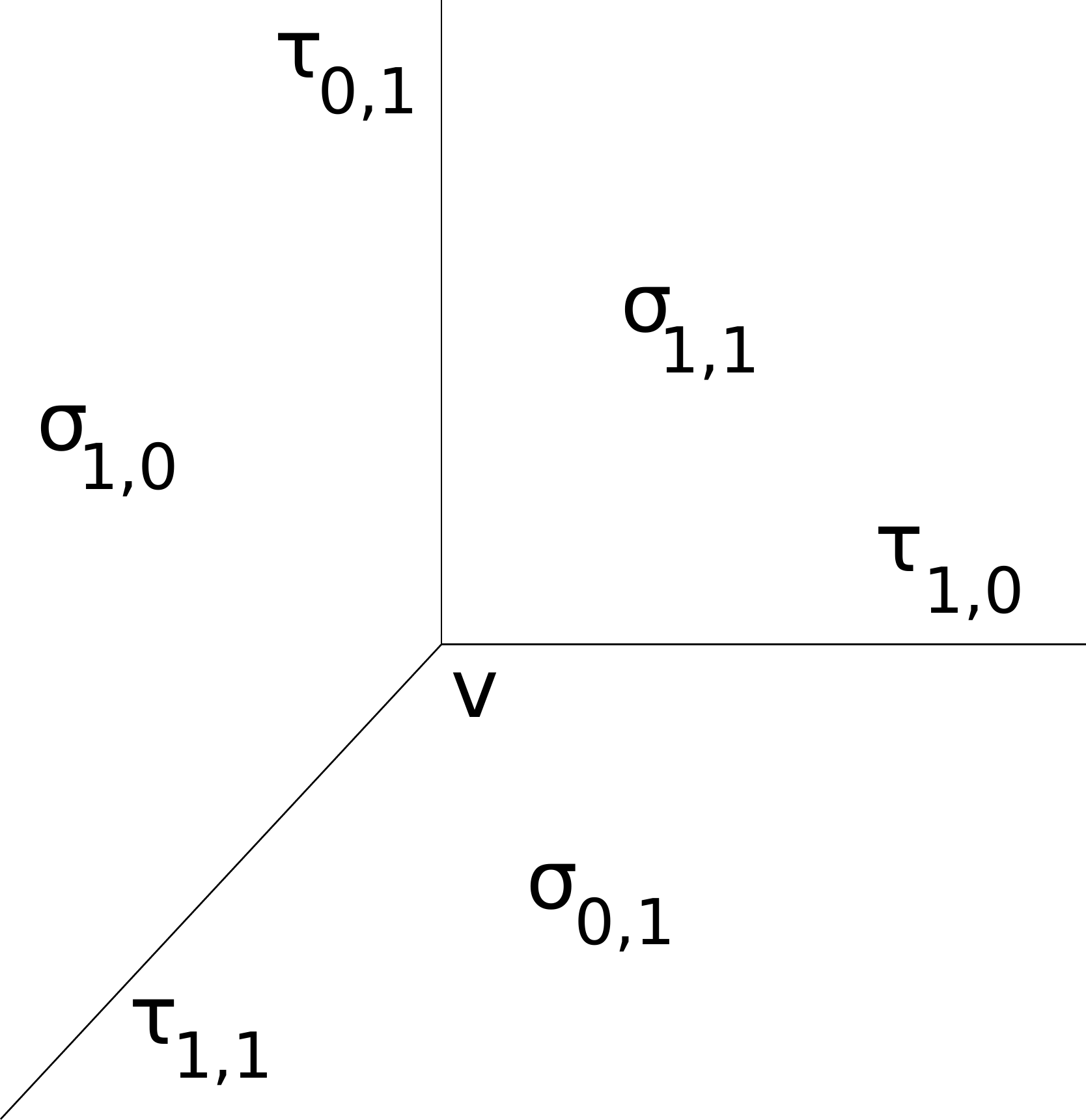}
	\caption{Polyhedral subdivision $\Sigma_v$ of $\T \pr^2$ from \Cref{DefFanPt}.}
	\label{SubdivisionPoint}
\end{subfigure}\hfill
\begin{subfigure}[t]{0.45\textwidth}
	\centering
	\includegraphics[width=\textwidth]{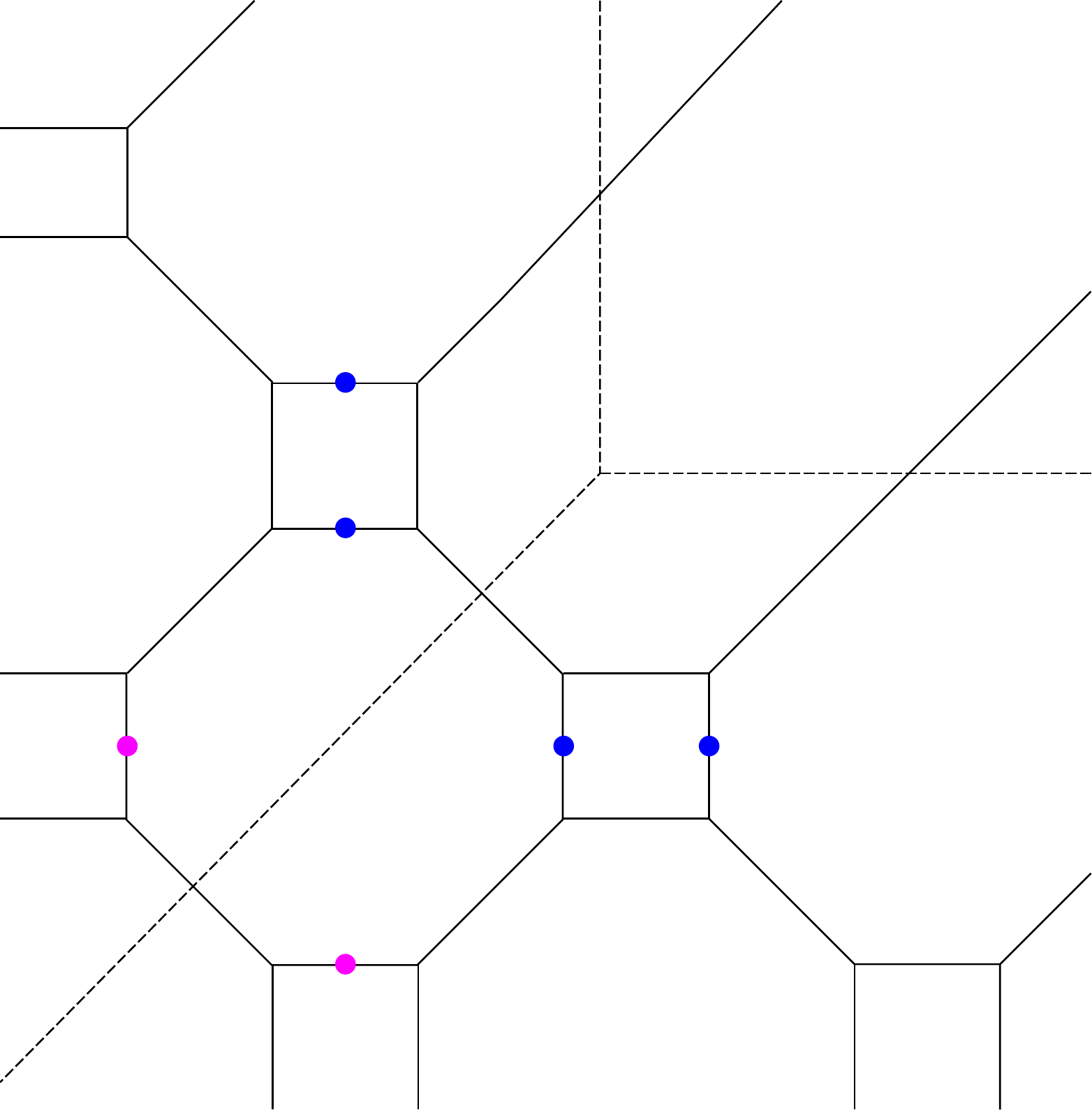}
	\caption{Non-singular real tropical curve hyperbolic with respect to a real tropical point from \Cref{ExHyperboQuart}.}
	\label{FigHyperboQuart}
\end{subfigure}
\caption{}
\label{FigPolyhedralSubdivision}
\end{figure}
A point $v$ is \emph{generic with respect to} a non-singular tropical curve $C$ if the half-lines $\tau_\eta$ of the polyhedral subdivision $\Sigma_v$ of $\T \pr^2$ with respect to $v$ intersect $C$ transversely in interior of edges. 
We can now characterise hyperbolicity with respect to a real tropical point thanks to the results of \Cref{SectionIntersection}.

\begin{theorem}[\Cref{IntroTh4.4.4}]
\label{Th4.4.4}
Let $(C,\E )$ be a non-singular real tropical curve of degree $d$ in $\T \pr^2$, let $v'$ be a point of $\T \pr^2 \backslash C$ and let $\varepsilon \in \mathcal{R}^2$.
Then $(C,\E)$ is hyperbolic with respect to the real tropical point $(v',\varepsilon)$ if and only if for $\Sigma_v$ the subdivision of $\T \pr^2$ with respect to a point $v$ generic with respect to $C$ in the same connected component of $\T \pr^2 \backslash C$ as $v'$, we have:
\begin{enumerate}
\item \label{ItemHyp1} Every vertex of $C$ lying in the interior of a face $\sigma_\eta$ of $\Sigma_v$ is incident to an edge of primitive integer direction $\eta$.

\item \label{ItemHyp2} Every edge $e$ of $C$ intersecting a face $\tau_\zeta$ of $\Sigma_v$ and such that $|\det (\overrightarrow{e} | \zeta )| = 2$, for $\overrightarrow{e}$ the primitive integer direction of $e$, satisfies $\varepsilon \in \E_e$.

\item \label{ItemHyp3} For every bounded edge $e$ of $C$ of primitive integer direction $\eta$ intersecting the face $\sigma_\eta$ of $\Sigma_v$, the edge $e$ is twisted if $e \subset \sigma_\eta$, and otherwise the segment $e\cap \sigma_\eta$ is relatively non-twisted with respect to the unique real tropical line $(L,\E ')$ through $(v,\varepsilon)$ with $(e\cap \sigma_\eta ) \subset e' \subset L$ and $\E_e = \E_{e'}'$. 
%
\end{enumerate}
\end{theorem}
%
%

\begin{example}
\label{ExHyperboQuart}
Let $C$ be the non-singular tropical curve of degree 4 in $\T \pr^2$ represented in \Cref{FigHyperboQuart} in plain black.
Let $\E$ be a real phase structure on $C$ inducing the set of twisted edges marked by blue and purple points.
Let $\Sigma_v$ be the polyhedral subdivision of $\T \pr^2$ with 1-dimensional faces represented in dotted black.
Then each 3-valent vertex of $C$ satisfies \Cref{ItemHyp1} of \Cref{Th4.4.4} with respect to the point $v$. 

Let $e_1$ and $e_2$ be the two edges of direction $(1,-1)$ in the biggest primitive cycle of $C$.
According to the twisted edges on that cycle, we can see that $\E_{e_1} = \E_{e_2}$.
Let $\varepsilon \in \mathcal{R}^2$ be an element contained in $\E_{e_1}$ and $\E_{e_2}$.
Then $(C,\E )$ satisfies \Cref{ItemHyp2} of \Cref{Th4.4.4} with respect to the real tropical point $(v,\varepsilon )$.

Finally, we can see that every bounded edge of direction $\eta$ lying in the interior of $\sigma_\eta$ is twisted.
Those edges are the ones marked by blue points on \Cref{FigHyperboQuart}, and the edges marked by purple points are twisted in order to satisfy twist-admissibility (\Cref{EqAdm}) and dividing (\Cref{EqDiv}) conditions.
Moreover, there is no bounded edge of direction $\eta$ intersecting $\sigma_\eta$ and such that $e \not\subset \sigma_\eta$.
Therefore $(C,\E )$ satisfies \Cref{ItemHyp3} of \Cref{Th4.4.4} with respect to the real tropical point $(v,\varepsilon )$, hence $(C,\E )$ is hyperbolic with respect to $(v,\varepsilon )$.  
\end{example} 

\begin{lemma}
\label{LemDirEdge}
Let $u$ be a 3-valent vertex of a non-singular tropical curve $C$.
Let $\eta \in \Z^2 \backslash \{ (0,0) \}$ be a primitive integer vector.
Then $u$ is incident to an edge of primitive integer direction $\eta$ if and only if every edge $e$ adjacent to $u$ satisfies 
\[ |\det (\overrightarrow{e} | \eta )| \leq 1 , \]
for $\overrightarrow{e} \in \Z^2$ the primitive integer direction of $e$. 
\end{lemma}

\begin{proof}
Since $C$ is non-singular, the self-intersection of $C$ locally around $u$ has tropical intersection multiplicity $(C\circ C)_u = 1$. 
In particular, for $e_0 , e_1 , e_2$ the three edges adjacent to $u$, we have 
$ | \det (\overrightarrow{e_i} | \overrightarrow{e_j} )| \leq 1  $
for $i,j \in \{ 0,1,2 \}$.
We obtain directly that if the edge $e_0$ has primitive integer direction $\eta$, then $ | \det (\overrightarrow{e_i} | \eta )| \leq 1  $
for $i \in \{ 0,1,2 \}$.

Conversely, assume that $ | \det (\overrightarrow{e_i} | \eta )| \leq 1 $
for $i \in \{ 0,1,2 \}$.
Choose the primitive integer directions $\overrightarrow{e_i}$ of the edges $e_i$ so that they all point inward the vertex $u$.
By non-singularity and balancing condition around $u$, we get $\sum_i \det (  \overrightarrow{e_i} | \eta ) = 0 $.
By assumption, exactly one of the edges $e_i$ (say $e_0$) must satisfy $\det (  \overrightarrow{e_0} | \eta ) = 0$.
Then $e_0$ has primitive integer direction $\eta$.
\end{proof}

\begin{proof}[Proof of \Cref{Th4.4.4}]
Let $D$ be the connected component of $\T \pr^2 \backslash C$ containing $v'$.
By \Cref{PropEquivHyp}, it is enough to check the hyperbolicity of $(C,\E )$ with respect to a real tropical point $(v,\varepsilon )$, with $v$ generic with respect to $C$ and contained in $D$.

Let $L$ be a tropical line going through $v$.
Since $v$ is generic with respect to $C$, an connected component $E\subset C\cap L$ is either a transverse intersection point, an isolated vertex of $C$ lying in the interior of a face $\sigma_\eta$ of $\Sigma_v$, a closed edge of $C$ lying in the interior of a face $\sigma_\eta$ of $\Sigma_v$, or a segment (non-reduced to a point) and strictly contained in a closed edge of $C$ intersecting a face $\tau_\zeta$ of $\Sigma_v$. 
In the following, we will check for each four possible intersection types that, for every real phase structure $\E '$ on $L$ such that the real tropical line $(L,\E ') $ goes through the real tropical point $(v,\varepsilon )$ (ie. ~ the real part $\R L_{\E '}$ goes through the real part $\R v_\varepsilon$ of $v$ given by the symmetric copy $\varepsilon (v)$), any realisation of $(C,\E )$ and $(L,\E ')$ intersect in distinct real points with tropicalisation in the considered connected component of the intersection if and only if \Cref{ItemHyp1}, \Cref{ItemHyp2} and \Cref{ItemHyp3} of the statement of \Cref{Th4.4.4} are satisfied. 
If $L$ goes through $v$ and has vertex distinct from $v$, the condition $(L,\E ')$ going through $(v,\varepsilon )$ imposes that $\varepsilon$ is contained in the real phase structure of the edge of $L$ containing $v$.  
Therefore, out of the 4 possible real phase structures $\E '$ on $L$, two of them satisfy the condition that $(L,\E ')$ goes through $(v,\varepsilon )$.

\begin{enumerate}[label=(\Roman*)]
\item \label{ItemProofHyp1} Assume first that $E$ is a transverse intersection point $u$ in $C\cap L$.
Let $e$ and $e'$ be the edges of $C$ and $L$ containing $u$, respectively.
By \Cref{CorTransMult2}, for any real phase structure $\E '$ on $L$ such that $\R L_{\E '}$ contains $\varepsilon (v)$, the intersection points with tropicalisation $u$ are all real for any realisations of $(C,\E )$ and $(L,\E ')$ if and only if $|\det (\overrightarrow{e} | \overrightarrow{e}')| \leq 2$, and $\E_e = \E_{e'}'$ if $|\det (\overrightarrow{e} | \overrightarrow{e}')| = 2$.

If $u$ lies in the interior of a face $\sigma_\eta$ of $\Sigma_v$, then the edge $e'$ is distinct from the edge of $L$ containing $v$, hence if $|\det (\overrightarrow{e} | \overrightarrow{e}')| = 2$, one of the two possible choices of real phase structure $\E '$ on $L$ satisfies $\E_e \cap \E_{e'}' = \emptyset$.
In that case, any realisations of $(C,\E)$ and $(L,\E' )$ intersect in two distinct complex conjugated points with tropicalisation $u$.
Then if $u$ lies in the interior of a face $\sigma_\eta$ of $\Sigma_v$, the intersection points with tropicalisation $u$ are all real for any choice of real phase structure $\E '$ and any realisations of $(C,\E )$ and $(L,\E ')$ if and only if $|\det (\overrightarrow{e} | \overrightarrow{e}')| \leq 1$.
Since in that case, the edge $e'$ must have primitive integer $\eta$, by \Cref{LemDirEdge} the intersection points with tropicalisation $u$ are all real for any choice of real phase structure $\E '$ and any realisations of $(C,\E )$ and $(L,\E ')$ if and only if for any vertex of $C$ incident to $e$, that vertex is incident to an edge of primitive integer direction $\eta$.

If $u$ lies on a face $\tau_\zeta$ of $\Sigma_v$, then $e'$ is the edge of $L$ containing $v$ and so has primitive integer direction $\zeta$.
For $\{ \zeta , \eta , \xi \} = \{ (1,0) , (0,1), (-1,-1) \}$, if the vertex of $e$ contained in $\sigma_\eta$ is incident to an edge of direction $\eta$, and the  the vertex of $e$ contained in $\sigma_\xi$ is incident to an edge of direction $\xi$, then by non-singularity of $C$ we obtain $|\det (\overrightarrow{e} | \zeta )| \leq 2$.
If moreover we have $|\det (\overrightarrow{e} | \zeta )| = 2$, then the condition $\E_e = \E_{e'}'$ is equivalent to the condition $\varepsilon \in \E_e$, since $\varepsilon \in \E_{e'}'$ by assumption and the edges $e$ and $e'$ have common primitive integer direction modulo 2.

Therefore, if $(C,\E)$ is hyperbolic with respect to $(v,\varepsilon )$, all vertices of $C$ lying in the interior of the face $\sigma_\eta$ of $\Sigma_v$ are incident to an edge of primitive integer direction $\eta$ (\Cref{ItemHyp1}), and the edges $e$ of $C$ intersecting $\tau_\zeta$ with $|\det (\overrightarrow{e} | \zeta )| = 2$ satisfy $\varepsilon \in \E_e$ (\Cref{ItemHyp2}).

\item \label{ItemProofHyp2} Assume that $E$ is an isolated vertex $u$ of $C$ contained in the interior of a face $\sigma_\eta$ of $\Sigma_v$.
Then every edge of $C$ adjacent to $u$ has primitive integer direction distinct from $\eta$.
In particular, by \Cref{LemDirEdge}, there exists an edge $e_u$ of $C$ adjacent to $u$ with $|\det (\overrightarrow{e_u} | \eta)| \geq 2$.
For $L'$ a tropical line going through $v$ given by a small translation of $L$ and intersecting $e_u$ in a transverse intersection point $u'$, we have $(C\circ L')_{u'} = |\det (\overrightarrow{e_u} | \eta)| \geq 2$.
Using the same arguments as in \Cref{ItemProofHyp1}, we obtain that if $C$ and $L$ intersect in an isolated vertex of $C$ lying in the interior of $\sigma_\eta$, there exists a real phase structure $\E '$ on $L'$ such that $(L' , \E ')$ goes through $(v, \varepsilon )$ and the realisations of $(C,\E )$ and $(L' , \E ')$ intersect in non-real points. 
In that case, the real tropical curve $(C,\E)$ cannot be hyperbolic with respect to $(v,\varepsilon)$.

Therefore if $(C,\E)$ is hyperbolic with respect to $(v,\varepsilon )$, there exist no tropical line going through $v$ intersecting $C$ in an isolated vertex lying in the interior of $\sigma_\eta$.   

\item \label{ItemProofHyp3} Assume that $E$ is a closed edge $e$ of $C$ lying in the interior of a face $\sigma_\eta$ of $\Sigma_v$.
Let $e'$ be the edge of $L$ containing $e$.
Since the edge $e$ lies in the interior of $\sigma_\eta$, the edge $e$ has primitive integer direction $\eta$ and the edge $e'$ does not contain the point $v$.

If $e$ is unbounded, then the edge $e$ has tropical intersection multiplicity 1, hence by \Cref{PropRealMult}, for any real phase structure $\E '$ on $L$ such that $\R L_{\E '}$ contains $\varepsilon (v)$, the unique intersection point tropicalising in $e$ is real for any realisations of $(C,\E )$ and $(L,\E ')$.

If $e$ is bounded, then by \Cref{Th4.3.9}, for any real phase structure $\E '$ on $L$ such that $\R L_{\E '}$ contains $\varepsilon (v)$, the intersection points tropicalising in $e$ are all real for any realisations of $(C,\E )$ and $(L,\E ')$ if and only if the edge $e$ is twisted.

Therefore if $(C,\E)$ is hyperbolic with respect to $(v,\varepsilon )$, every closed bounded edge $e$ of $C$ of primitive integer direction $\eta$ and contained in the interior of the face $\sigma_\eta$ is twisted (first case of \Cref{ItemHyp3}).
 
\item \label{ItemProofHyp4} Assume that $E$ is a segment (non-reduced to a point) strictly contained in a closed edge $e$ of $C$ intersecting a face $\tau_\zeta$ of $\Sigma_v$.
Since $v$ is generic with respect to $C$, the edge $e$ intersects $\tau_\zeta$ in a transverse point in the interior of $e$ and $\tau_\zeta$.
Let $e'$ be the closed edge of $L$ containing $E$.
Note that $e'$ does not contain the point $v$.
Since the edge $e$ intersects $\tau_\zeta$ and is not contained in $\tau_\zeta$, it has primitive integer direction in $\{ \eta , \xi \}$, for $\{ \zeta , \eta , \xi \} = \{ (1,0) , (0,1), (-1,-1) \}$. 

If $e$ is unbounded, then the segment $E$ has tropical intersection multiplicity 1.
By \Cref{PropRealMult}, for any real phase structure $\E '$ on $L$ such that $(L,\E ')$ goes through $(v,\varepsilon)$, every realisation of $(C,\E )$ and $(L,\E ')$ intersect in a single real point tropicalising in $E$.

If $e$ is bounded, then by \Cref{Th4.3.10}, for any real phase structure $\E '$ on $L$ such that $(L,\E ')$ goes through $(v,\varepsilon )$, every realisations of $(C,\E )$ and $(L,\E ')$ intersect in distinct real points with tropicalisation in $E$ if and only if either $\E_e \neq \E_{e'} '$ or $\E_e = \E_{e'} '$ and $E$ is relatively non-twisted.
One of the two possible choices for $\E '$ satisfies $\E_e \neq \E_{e'} '$ since $e'$ does not contain $v$, and the other choice satisfies $\E_e = \E_{e'} '$.

Therefore if $(C,\E)$ is hyperbolic with respect to $(v,\varepsilon )$, for every closed bounded edge $e$ of $C$ of primitive integer direction $\eta$ intersecting $\tau_\zeta$, the segment $E := e \cap \sigma_\eta$ is relatively non-twisted with respect to the unique real tropical line through $(v,\varepsilon )$ containing $E$ and with $\E_e = \E_{e'} '$ (second case of \Cref{ItemHyp3}). 
\end{enumerate}

Therefore, if $(C,\E)$ is hyperbolic with respect to $(v,\varepsilon )$, the conditions \Cref{ItemHyp1}, \Cref{ItemHyp2} and \Cref{ItemHyp3} from the statement of \Cref{Th4.4.4} are satisfied.
Conversely, if \Cref{ItemHyp1}, \Cref{ItemHyp2} and \Cref{ItemHyp3} are satisfied, then all possible types of intersection components give rise to only distinct real points by \Cref{ItemProofHyp1}, \Cref{ItemProofHyp2}, \Cref{ItemProofHyp3} and \Cref{ItemProofHyp4}, hence $(C,\E)$ is hyperbolic with respect to $(v,\varepsilon )$. 
\end{proof}

\begin{definition}
\label{DefHoneycomb}
Let $C$ be a non-singular tropical curve.
We say that $C$ is a \emph{honeycomb} if every edge $e$ of $C$ has primitive integer direction contained in $\{ (1,0 ) , (0,1) , (1,1) \}$.
\end{definition}

For instance, the non-singular tropical curve represented in \Cref{FigHyperboHoney} is a honeycomb, and a non-singular honeycomb of degree $d$ in $\T \pr^2$ has dual subdivision of the form pictured in \Cref{FigHyperboHoney2Dual}.
If $C$ is a non-singular honeycomb of degree $d$ in $\T \pr^2$, then for $\E$ a real phase structure on $C$ and $(v',\varepsilon )$ a real tropical point, \Cref{ItemHyp1} and \Cref{ItemHyp2} of \Cref{Th4.4.4} are automatically satisfied, hence the real tropical curve $(C,\E )$ is hyperbolic with respect to $(v',\varepsilon )$ if and only if \Cref{ItemHyp3} of \Cref{Th4.4.4} is satisfied. 

Recall that a real algebraic curve $\mathcal{C}$ in $\pr^2$, or in this case $\pr_\K^2$, is said to be \emph{stable} if $\mathcal{C}$ is hyperbolic with respect to every point in the positive orthant $(\R_{>0})^2$ of the algebraic torus $(\R^\times )^2 \subset \pr^2 (\R)$, or in this case in the positive orthant $((\K_\R)_{>0})^2$ of the algebraic torus $(\K_\R^\times )^2 \subset \pr_\K^2 (\R)$.
We give a new approach to Speyer's classification of stable curves near the tropical limit in the case when the tropical limit is non-singular.

\begin{theorem}[\cite{speyer2005horn}]
\label{ThStable}
Let $(C,\E )$ be a non-singular real tropical curve of degree $d$ in $\T \pr^2$, and let $\mathcal{C} \subset \pr_\K^2$ be a realisation of $(C,\E)$.
Then $(C,\E )$ is hyperbolic with respect to all real tropical points $(v,\varepsilon) \in \T \pr^2 \times \mathcal{R}^2$ for some fixed $\varepsilon \in \mathcal{R}^2$ if and only if $C$ is a honeycomb with every bounded edges twisted.
Moreover, the curve $\mathcal{C}$ is stable if and only if, in addition, every distribution of signs $\delta$ determining $\E$ is constant on $\Delta_C \cap \Z^2$.
\end{theorem}

\begin{proof}[Proof of \Cref{ThStable}]
Assume that $(C,\E )$ is hyperbolic with respect to all real tropical points $(v,\varepsilon) \in \T \pr^2 \times \mathcal{R}^2$ for some fixed $\varepsilon \in \mathcal{R}^2$.
By \Cref{ItemHyp1} of \Cref{Th4.4.4} applied to every point $v$ in $\T \pr^2$, every vertex of $C$ not lying on a 1-dimensional stratum of $\T \pr^2$ must be incident to an edge of direction $(1,0)$, an edge of direction $(0,1)$ and an edge of direction $(1,1)$.
Since $C$ is non-singular, every such vertex is 3-valent, hence $C$ is a honeycomb.
By \Cref{ItemHyp3} of \Cref{Th4.4.4} applied to every point $v$ in $\T \pr^2$, every bounded edge of $C$ of direction $(1,0) , (0,1)$ or $(1,1)$ must be twisted, hence every bounded edge of $C$ must be twisted since $C$ is honeycomb.

Conversely, assume that $(C,\E )$ is a honeycomb with every bounded edge twisted.
Then \Cref{ItemHyp1}, \Cref{ItemHyp2} and the first part of \Cref{ItemHyp3} in \Cref{Th4.4.4} are satisfied for every point $v$ in $\T \pr^2$.
By the first item of \Cref{PropTwistSign}, since every bounded edge is twisted, there exists an orthant $\varepsilon \in \mathcal{R}^2$ of $\R \pr^2$ such that (the extension to all symmetric copies of) every distribution of signs $\delta$ determining $\E$ is constant on $\varepsilon (\Delta_C) \cap \Z^2$.
Then for every tropical line $L$ intersecting $C$ in a segment $E$ strictly contained in both an edge $e$ of $C$ and an edge $e'$ of $L$, if $\E '$ is a real phase structure on $L$ such that $(L,\E' )$ goes through a real tropical point $(v,\varepsilon)$, then either $\E_e \neq \E_{e'}'$ or $E$ is relatively non-twisted.
Therefore \Cref{ItemHyp3} in \Cref{Th4.4.4} is satisfied, hence $(C,\E )$ is hyperbolic with respect to every real tropical point $(v,\varepsilon)$.
We obtain the additional statement whenever $\varepsilon = (0,0)$ (seen in $\Z_2^2$). 
\end{proof} 

Speyer's statement is stronger than \Cref{ThStable}, as it considers any stable curve near the tropical limit, without non-singularity assumption.
Using a generalisation of \Cref{CorTransMult2} to weighted edges (in terms of distribution of signs instead of real phase structure), we can recover the fact that if the curve $\mathcal{C}$ is non-singular and stable, then the tropical limit $C$ is a honeycomb with possibly some contracted bounded edges, and the associated distribution of signs is constant on $\Delta_C \cap \Z^2$.
Speyer completed the classification by proving that the subset of non-singular stable curves is dense in the set of stable curves \cite[Lemma 3]{speyer2005horn}.   

\subsection{Vector space structure on twist-admissible sets}

The following will allow us to consider the set of dividing sets of twisted edges on a non-singular tropical curve as a $\Z_2$-vector space.
We will need this point of view in order to treat the case of hyperbolic honeycombs in \Cref{SecHoneycomb} and prove \Cref{propHypHoney}.
 
Let $C$ be a non-singular tropical curve in $\R^2$ or $\T \pr^2$, and let $g$ be the number of primitive cycles of $C$.
We can equip the set $\Edge^0 (C)$ of bounded edges of $C$ with a $\Z_2$-vector space structure by identification with $\Z_2^{| \Edge^0 (C) |}$.

\begin{definition}
\label{DefAdm}
Let $C$ be a non-singular tropical curve in $\R^2$ or $\T \pr^2$, and let $g$ be the number of primitive cycles of $C$.
The space of \emph{admissible configurations of twists} $\Adm (C)$ on $C$ is the $\Z_2$-vector subspace of $\Edge^0 (C)$ induced by the $2g$ linear conditions from \Cref{EqAdm} (saying that the sum of directions modulo 2 of twisted edges on a primitive cycle is the zero vector).
The space of \emph{dividing configurations of twists} $\Div (C)$ is the $\Z_2$-vector subspace of $\Adm (C)$ induced by the linear conditions from \Cref{EqDiv} (saying that each primitive cycle has an even number of twisted edges).
\end{definition}

\begin{remark}
\label{RkDividing}
Let $g$ be the number of primitive cycles of a non-singular tropical curve $C$ of degree $d$ in $\T \pr^2$.
The space $\Div (C)$ has codimension $l \leq g$ in $\Adm (C)$, where $l$ is the number of primitive cycles $\gamma$ in $C$ such that $\gamma$ contains at least one edge of direction $(1,0)$ modulo 2, one edge of direction $(0,1)$ modulo 2 and one edge of direction $(1,1)$ modulo 2.

Let $\Hyp (C)$ be the set of configuration of twists $T$ such that a non-singular real tropical curve $(C,\E)$ inducing $T$ is hyperbolic.
For $d=1,2,3$, we have equality of $\Hyp (C)$ and $\Div (C)$, hence in that case the set $\Hyp (C)$ inherits the $\Z_2$-vector space structure from $\Div (C)$.
For $d\geq 4$, any configuration of twists $T \in \Hyp (C)$ satisfies 
\[   \dim \ker A_T = \left\lfloor \frac{d}{2} \right\rfloor < g = \binom{d-1}{2} ,  \]
hence $T \neq \overrightarrow{0}$, for $\overrightarrow{0}$ the zero vector of $\Edge^0 (C)$ corresponding to the empty set of edges.
In that case, the set $\Hyp (C)$ is strictly contained in the set $\Div (C)$, and does not inherit the $\Z_2$-vector space structure from $\Div (C)$.
However, the following property is still satisfied. 
For $T \in \Hyp (C)$ and $T' \in \Div (C)$ such that the edge supports of $T$ and $T'$ are disjoint, the configuration of twists $T+T'$ belongs to $\Hyp (C)$.
This follows from the fact that hyperbolic curves are dividing curves with minimal number of connected components by \Cref{PropHypRok}, hence $\dim \ker A_{T + T'} \geq \dim \ker A_T$, and since the edge supports are disjoint, we have $\dim \ker A_{T + T'} \leq \dim \ker A_T$. 
\end{remark}

\subsection{Going further in the honeycomb case}
\label{SecHoneycomb}


Honeycombs can be decomposed into smaller honeycombs glued along multi-bridges.

\begin{definition}
\label{DefMultiBridge}
Let $C$ be a non-singular honeycomb in $\T \pr^2$.
A \emph{multi-bridge} $B$ on $C$ is a twist-admissible dividing set of bounded edges of $C$, with common primitive integer direction, such that the graph $C\backslash B$ has two connected components, and the edges of the dual subdivision $\Delta_C$ dual to the edges of $B$ lie on a common line in $\Delta_C$.
\end{definition}

From \Cref{DefHoneycomb} and \Cref{DefMultiBridge}, we obtain that any two distinct multi-bridges $B,B'$ on a non-singular honeycomb $C$ have disjoint edge supports. 

\begin{figure}
\centering
\begin{subfigure}[t]{0.45\textwidth}
	\centering
	\includegraphics[width=\textwidth]{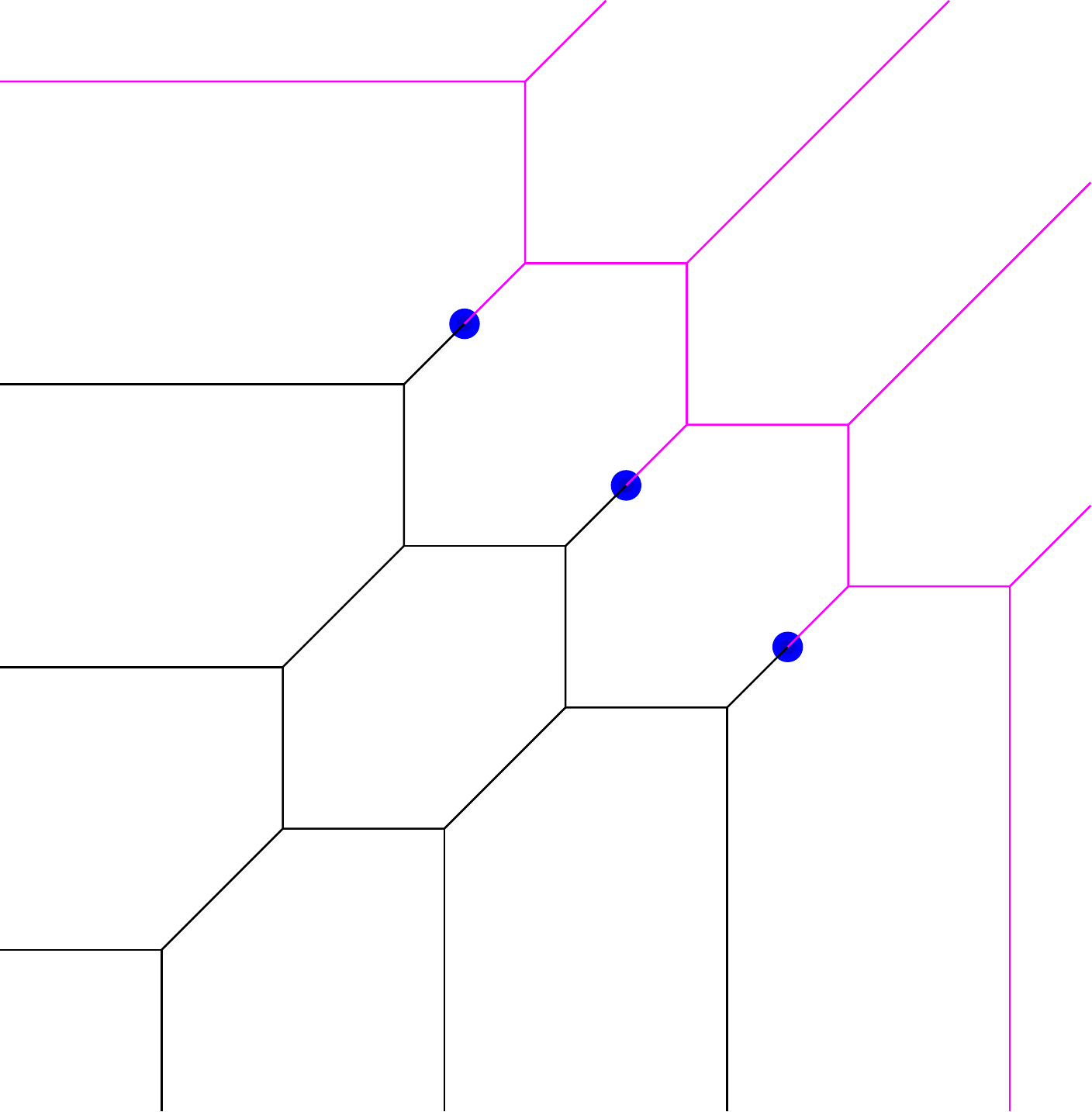}
	\caption{Honeycomb with a twisted multi-bridge.}
	\label{FigHyperboHoney2}
\end{subfigure}\hfill
\begin{subfigure}[t]{0.45\textwidth}
	\centering
	\includegraphics[width=\textwidth]{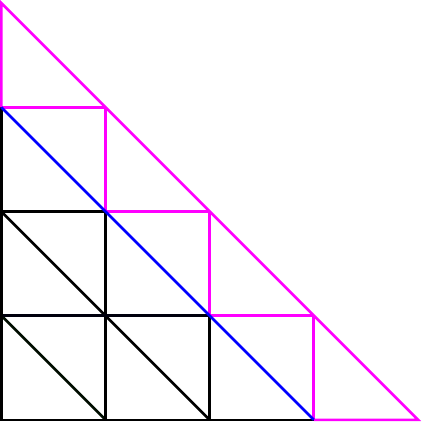}
	\caption{Dual subdivision and edges dual to the multi-bridge.}
	\label{FigHyperboHoney2Dual}
\end{subfigure}
\caption{\Cref{ExHoneyBridge} and \Cref{ExLast}.}
\label{FigHoneyBridge}
\end{figure}
 
\begin{example}
\label{ExHoneyBridge}
The non-singular tropical curve $C$ pictured in \Cref{FigHyperboHoney2} is a  honeycomb of degree 4 in $\T \pr^2$.
The set of edges of $C$ marked by blue points form a multi-bridge $B$, with all edges of $B$ having primitive integer direction $(1,1)$.
\Cref{FigHyperboHoney2Dual} represents the dual subdivision $\Delta_C$, and the blue edges in $\Delta_C$ are the edges dual to the edges of $B$.
Note that the blue edges in $\Delta_C$ lie on a common line.
\end{example}

\begin{proposition}
\label{PropHoneyBridge}
Let $C$ be a non-singular honeycomb of degree $d$ in $\T \pr^2$.
The $\Z_2$-vector space $\Div (C)$ is of dimension $3(d-1)$ with basis given by the set of multi-bridges on $C$, seen as vectors in the $\Z_2$-vector space $\Edge^0 (C)$.
\end{proposition} 

\begin{proof}
Let $g$ be the number of primitive cycles in $C$.
Since $C$ is of degree $d$ in $\T \pr^2$, we have $g = \binom{d-1}{2}$.
Since $C$ is a non-singular honeycomb, every primitive cycle of $C$ contains exactly two edges of direction $(1,0)$, two edges of direction $(0,1)$ and two edges of direction $(1,1)$.
Then by \Cref{DefAdm} and \Cref{RkDividing}, the $\Z_2$-vector space $\Div (C)$ is a subspace of $\Edge^0 (C)$ of codimension $3 \binom{d-1}{2}$.
A non-singular honeycomb of degree $d$ in $\T \pr^2$ has $3 \binom{d}{2}$ bounded edges (with $1+\cdots + (d-1) = \binom{d}{2}$ bounded edges in each direction in $\{ (1,0), (0,1) , (1,1) \}$).
By definition of the $\Z_2$-vector space structure on $\Edge^0 (C)$, we obtain $\dim \Edge^0 (C) = 3 \binom{d}{2}$.
Therefore the $\Z_2$-vector space $\Div (C)$ is of dimension $3 \binom{d}{2} - 3 \binom{d-1}{2} = 3 (d-1)$.
Now each multi-bridge $B$ on a non-singular honeycomb $C$ has edge support dual to edges in $\Delta_C$ contained in a common line.
For each $s\in \{ (1,0), (0,1) , (1,1) \}$, there are $(d-1)$ distinct lines of direction $s$ modulo 2 that contain edges of $\Delta_C$ dual to bounded edges of $C$.
Then $C$ contains exactly $3(d-1)$ multi-bridges, such that any two distinct multi-bridges have disjoint edge supports, hence the vectors in $\Edge^0 (C)$ coming from multi-bridges are all linearly independent.
Since a multi-bridge is twist-admissible and dividing, the set of multi-bridges on $C$ form a basis of the $\Z_2$-vector space $\Div (C)$.
\end{proof}

Using \Cref{PropHoneyBridge}, we obtain a way to check if a point belongs to the tropical hyperbolicity locus of a non-singular real honeycomb, without going through the relatively non-twisted condition of \Cref{ItemHyp3} in \Cref{Th4.4.4}.

\begin{theorem}
\label{propHypHoney}
Let $(C,\E )$ be a non-singular real honeycomb of degree $d$ in $\T \pr^2$ and with induced dividing set of twisted edges $T$.
Let $v$ be a point of $\T \pr^2 \backslash C$ generic with respect to $C$, and let $\Sigma_v$ be the polyhedral subdivision of $\T \pr^2$ with respect to $v$.
Then the real tropical curve $(C,\E )$ is hyperbolic with respect to $(v,\varepsilon )$ for some $\varepsilon \in \mathcal{R}^2$ if and only if for every bounded edge $e\subset C\cap \Int (\sigma_\eta)$ of direction $\eta$, the unique multi-bridge $B$ on $C$ containing $e$ is twisted.
\end{theorem}

\begin{proof}
If $(C,\E)$ is hyperbolic with respect to a symmetric copy of $v$, then for every bounded edge $e \subset C\cap \Int (\sigma_\eta)$ of direction $\eta$, the edge $e$ is twisted by \Cref{Th4.4.4}.
Since $T$ is a dividing set of twisted edges, we obtain by \Cref{PropHoneyBridge} that for every edge $e$ satisfying those conditions, the unique multi-bridge $B$ containing $e$ is contained in $T$.

Conversely, assume that for every bounded edge $e \subset C\cap \Int (\sigma_\eta)$ of direction $\eta$, the unique multi-bridge $B$ containing $e$ is contained in the dividing set of twisted edges $T$ on $C$.
We want to prove iteratively on the degree $d$ of $C$ that $(C,\E )$ is hyperbolic with respect to $(v,\varepsilon )$ for some $\varepsilon \in \mathcal{R}^2$.

By \Cref{RkDividing}, for $d= 1,2,3$, the sets $\Div (C)$ and $\Hyp (C)$ are equal (as sets and as $\Z_2$-vector spaces), hence the fact that $T$ is dividing already implies that $(C,\E )$ is hyperbolic. 
If $d=1$, then $(C,\E )$ is hyperbolic with respect to all real tropical points $(v' , \varepsilon ')$ not lying on $(C,\E)$.
If $d=2,3$, one can check by hand using \Cref{Th4.4.4} that \Cref{propHypHoney} holds.

Assume as an iterative hypothesis that \Cref{propHypHoney} is true up to degree $d-1$, with $d \geq 4$.
Let $C$ be of degree $d$.
Since $C$ is a non-singular honeycomb, there exists a multi-bridge $B$ on $C$ such that $C$ is given as the gluing along $B$ of a honeycomb $C'$ of degree $d-1$ with another honeycomb $C''$ (see for example \Cref{FigHyperboHoney2} with $C'$ of degree $d-1$ in black and $C''$ in purple, glued along the multi-bridge $B$ marked with blue dots).
This gluing induces a decomposition of the dual subdivision $\Delta_C$ into $\Delta_{C'} \cup \Delta_{C''}$, with intersection $\Delta_{C'} \cap \Delta_{C''}$ given as the edges dual to edges of $B$ (see for example in \Cref{FigHyperboHoney2Dual} the gluing of dual subdivisions $\Delta_{C'}$ of degree $d-1$, in black, with $\Delta_{C''}$, in purple, along the edges dual to $B$, in blue).  
Since $d\geq 4$, we can choose $B$ such that the point $v$ belongs to a connected component $D \subset \T \pr^2 \backslash C$ dual to an integer point $(p_1 , p_2) \in (\Delta_{C'} \backslash B^\vee) \cap \Z^2$. 

For $\varepsilon \in \mathcal{R}^2$, the non-singular real tropical curve $(C,\E)$ is hyperbolic with respect to the real tropical point $(v,\varepsilon)$ if and only if the symmetric copy $\varepsilon (D)$ is contained in the tropical hyperbolicity locus of $(C,\E)$. 
Then the set of integer points $(p_1 , p_2) \in \Delta_{C} \cap \Z^2$ such that their dual connected component $D'$ is contained in the tropical hyperbolicity locus of $(C,\E)$ depends only on the dual subdivision $\Delta_C$ of $C$ and the real phase structure on $C$, or equivalently on the distribution of signs on $\Delta_C$.
In particular, it does not depend on the length of the bounded edges of $C$.
Choose a point $v'$ in $D$ so that every bounded edge of $C$ intersected by the 1-dimensional faces of the polyhedral subdivision $\Sigma_{v'}$ of $\T \pr^2$ is dual to an edge of $(\Delta_{C'} \backslash B^\vee)$ (such a point $v'$ always exist up to replacing $C$ by another non-singular honeycomb of degree $d$ in $\T \pr^2$, also called $C$, with distinct bounded edges lengths).
In particular, the multi-bridge $B$ is contained in the face $\sigma_\eta$ of $\Sigma_{v'}$, for $\eta$ the direction of the edges of $B$.
By assumption, we obtain that the multi-bridge $B$ is contained in the set of twisted edges $T$.
  
By induction hypothesis, there exists $\varepsilon \in \mathcal{R}^2$ such that the non-singular real tropical curve $(C' , \E |_{C'})$ is hyperbolic with respect to $(v' ,\varepsilon )$.
Since every edge $e$ of the multi-bridge $B$ on $C$ is twisted, for any tropical line $L$ through $v$ containing $e$, by \Cref{Th4.3.9} the connected component $e$ of $C\cap L$ lifts to two distinct real points of multiplicity 1. 
All intersection components $E$ which are segments strictly contained in a bounded edge of $C$ are in fact strictly contained in a bounded edge of $C'$.
Since $(C' , \E |_{C'})$ is hyperbolic with respect to $(v' ,\varepsilon )$ for some $\varepsilon \in \mathcal{R}^2$, all these intersection components $E$ give rise to real points of multiplicity 1 by \Cref{ItemHyp3} of \ref{Th4.4.4} and \Cref{Th4.3.10}.
Therefore, the real tropical curve $(C,\E)$ satisfies \Cref{ItemHyp3} of \Cref{Th4.4.4} with respect to the real tropical point $(v' ,\varepsilon )$, hence  $(C ,\E )$ is hyperbolic with respect to $(v' ,\varepsilon )$, and thus $(C ,\E )$ is hyperbolic with respect to $(v,\varepsilon )$.
\end{proof}

Let $\alpha = (\alpha_1 , \alpha_2) \in \Delta_C \cap \Z^2$.
A multi-bridge $B$ is said to be \emph{on the left} of $\alpha^\vee $ if $B^\vee$ lies in a vertical line through a point $(\beta_1 , \alpha_2) \in \Delta_C \cap \Z^2$ with $\beta_1 < \alpha_1$. 
Similarly, the multi-bridge $B$ is said to be \emph{below} $\alpha^\vee $ if $B^\vee$ lies in a horizontal line through a point $(\alpha_1 , \beta_2) \in \Delta_C \cap \Z^2$ with $\beta_2 < \alpha_2$, and $B$ is said to be \emph{diagonally above} $\alpha^\vee $ if $B^\vee$ lies in a line of direction $(1,-1)$ going through the point $(\beta_1 , \beta_2) \in \Delta_C \cap \Z^2$ with $\alpha_1 + \alpha_2 < \beta_1 + \beta_2$.
We deduce from the proof of \Cref{propHypHoney} the following corollary, stated in terms of the dual subdivision.

\begin{corollary}[\Cref{IntroCorHypHoneyBridge}]
\label{CorHypHoneyBridge}
Let $(C,\E)$ be a non-singular real honeycomb with induced dividing set of twisted edges $T$.
Let $\alpha = (\alpha_1 , \alpha_2 ) \in \Delta_C \cap \Z^2$ be an integer point in the dual subdivision of $C$.
The dual connected component $\alpha^\vee \subset (\T\pr^2 \backslash C)$ is contained in the tropical hyperbolicity locus $H$ of $(C,\E)$ if and only if every multi-bridge on the left, below and diagonally above $\alpha^\vee$ is twisted.
\end{corollary}

\begin{example}
\label{ExLast}
Let $(C,\E)$ be a non-singular real honeycomb of degree 4 in $\T \pr^2$, with induced set of twisted edges given as exactly one multi-bridge, as described in \Cref{FigHyperboHoney2}.
Let $\alpha = (1,1) \in \Delta_C \cap \Z^2$ an integer point in the dual subdivision.
This point belongs to the black zone in \Cref{FigHyperboHoney2Dual}.
The set of twisted edges $T$ on $C$ satisfies \Cref{CorHypHoneyBridge} for the point $\alpha$, and uniquely for the point $\alpha$, hence the tropical hyperbolicity locus $H$ of $(C,\E )$ consists of the connected component $\alpha^\vee \subset (\T \pr^2 \backslash C)$.
Since $\alpha^\vee$ is bounded by a non-twisted primitive cycle of $C$, we can conclude that the signed tropical hyperbolicity locus $\R H$ of $(C,\E )$ consists of exactly one symmetric copy of $\alpha^\vee$. 
\end{example}

\begin{remark}
Let $\alpha$ be an integer point of the dual subdivision of a non-singular honeycomb $C$ of degree $d$ in $\T \pr^2$.
We can define a set $\Hyp_\alpha (C)$ of configurations of twists $T$ on $C$ so that the tropical hyperbolicity locus of any $(C,\E)$ inducing $T$ contains the connected component $\alpha^\vee \subset \T \pr^2 \backslash C$.
Using \Cref{CorHypHoneyBridge} and \Cref{RkDividing}, we can see that $\Hyp_\alpha (C)$ is a $\Z_2$-affine subspace of $\Div (C)$, with origin given by the configuration of twists $T$ with edge support the union of multi-bridges on the left, below and diagonally above $\alpha$.
The codimension of $\Hyp_\alpha (C)$ in $\Div (C)$ is given by the number of multi-bridges contained in $T$.
\end{remark}

 \bibliographystyle{alpha}
 \bibliography{bibliomoi1}

\end{document}